\documentclass[11pt,oneside,english,reqno,twoside]{amsart}
\usepackage{amssymb,nicefrac,bm,upgreek,mathtools,verbatim}
\usepackage{lmodern}

\usepackage[T1]{fontenc}
\usepackage[latin9]{luainputenc}
\usepackage{geometry}
\usepackage{amstext}
\usepackage{amsthm}
\usepackage{amssymb}
\usepackage{xcolor}
\usepackage{color}
\definecolor{dmagenta}{rgb}{.4,.1,.4}
\definecolor{dblue}{rgb}{.0,.0,.5}
\definecolor{mblue}{rgb}{.0,.4,.7}
\definecolor{ddblue}{rgb}{.0,.0,.4}
\definecolor{dred}{rgb}{.9,.0,.0}
\definecolor{dgreen}{rgb}{.0,.5,.0}
\definecolor{Eeom}{rgb}{.0,.0,.5}
\definecolor{dbrown}{rgb}{.6,.0,.0}

\makeatletter
\numberwithin{equation}{section}
\numberwithin{figure}{section}
\theoremstyle{plain}
\newtheorem{thm}{\protect\theoremname}[section]
  \theoremstyle{definition}
  \newtheorem{defn}[thm]{\protect\definitionname}
  \theoremstyle{plain}
  \newtheorem{lem}[thm]{\protect\lemmaname}
  \theoremstyle{remark}
  \newtheorem{rem}[thm]{\protect\remarkname}
  \theoremstyle{plain}
  \newtheorem{cor}[thm]{\protect\corollaryname}
  \newtheorem{example}{Example}[section]

\usepackage{versions}

\DeclareMathOperator*{\esssup}{ess\,sup}

\DeclareMathOperator*{\esssinf}{ess\,inf}

\DeclareMathOperator*{\esssliminf}{ess\,lim\,inf}

\def\Xint#1{\mathchoice
{\XXint\displaystyle\textstyle{#1}}%
{\XXint\textstyle\scriptstyle{#1}}%
{\XXint\scriptstyle\scriptscriptstyle{#1}}%
{\XXint\scriptscriptstyle\scriptscriptstyle{#1}}%
\!\int}
\def\XXint#1#2#3{{\setbox0=\hbox{$#1{#2#3}{\int}$ }
\vcenter{\hbox{$#2#3$ }}\kern-.6\wd0}}

\def\dashint{\Xint-}

\usepackage[colorlinks=true, linkcolor=black, urlcolor=blue, citecolor=black]{hyperref}

\newcommand{\Addresses}{{
\bigskip
\footnotesize
\textsc{Abhrojyoti Sen,  Goethe-Universit\"{a}t Frankfurt, Institut f\"{u}r Mathematik, Robert-Mayer-Str. 10, D-60629 Frankfurt, Germany}\par\nopagebreak
\textit{E-mail address}: \href{mailto:sen@math.uni-frankfurt.de}{sen@math.uni-frankfurt.de}
  
\textsc{Jarkko Siltakoski, Department of Mathematics and Statistics, University of Helsinki, P.O. Box 68, FI-00014 Helsinki, Finland}\par\nopagebreak
\textit{E-mail address}:
\href{mailto:jarkko.siltakoski@helsinki.fi}{jarkko.siltakoski@helsinki.fi}}}

\makeatother

\usepackage{babel}
  \providecommand{\definitionname}{Definition}
  \providecommand{\lemmaname}{Lemma}
  \providecommand{\remarkname}{Remark}
\providecommand{\theoremname}{Theorem}
\providecommand{\corollaryname}{Corollary}

\begin{document}
\global\long\def\d{\,d}
\global\long\def\tr{\mathrm{tr}}
\global\long\def\supp{\operatorname{spt}}
\global\long\def\div{\operatorname{div}}
\global\long\def\osc{\operatorname{osc}}
\global\long\def\essup{\esssup}
\global\long\def\aint{\dashint}
\global\long\def\essinf{\esssinf}
\global\long\def\essliminf{\esssliminf}

\excludeversion{old}

\excludeversion{note}

\excludeversion{note2}


\title[Lipschitz regularity for parabolic double phase]{Lipschitz regularity for parabolic double phase equations with gradient nonlinearity}

\author{Abhrojyoti Sen and Jarkko Siltakoski}
\begin{abstract}
We establish the local Lipschitz regularity in space for the viscosity solutions to the parabolic double phase equation of the form 
\[
\smash{\partial_{t}u-\div \left(|Du|^{p-2}D u+a(z)|D u|^{q-2}D u\right)=f(z, Du)}
\]
by employing the Ishii-Lions method. In addition, we obtain H\"{o}lder estimate in time which turns out to be sharp in the degenerate regime. Here, $1< p\leq  q<\infty,$ and the coefficient $a\geq 0$ is assumed to be bounded, locally Lipschitz continuous in space, and continuous in time. Furthermore, the non-homogeneity $f$ is assumed to be continuous on $\Omega\times \mathbb{R}\times \mathbb{R}^N,$ and to satisfy a suitable gradient growth condition. We also establish the equivalence between bounded viscosity solutions and weak solutions, under appropriate additional regularity assumption on the coefficient $a.$
\end{abstract}
\maketitle
\tableofcontents
\section{Introduction and main results}
\subsection{Overview} In this paper, we consider the parabolic double phase problem with gradient nonlinearity which reads,
\begin{equation}\label{eq:p-para f}
\partial_{t}u-\div \left(|Du|^{p-2}D u+a(z)|D u|^{q-2}D u\right)=f(z, Du)
\end{equation}
and address the following questions:
\begin{itemize}
\item[\textit{Q 1.}] Can a notion of viscosity solution be formulated for \eqref{eq:p-para f}?

\item[\textit{Q 2.}] Can a Lipschitz regularity estimate be established for viscosity solutions to \eqref{eq:p-para f}?

\item[\textit{Q 3.}] Can the notions of viscosity and weak solutions to \eqref{eq:p-para f} be shown to be equivalent for $1<p\leq q<\infty$ and continuous $f$ with an appropriate growth condition?
\end{itemize}
\noindent We begin by reviewing the existing literature and providing motivation for addressing these three questions in the context of the parabolic double phase problem \eqref{eq:p-para f}.

A classical solution to a partial differential equation is a smooth function that satisfies the equation pointwise. However, many equations arising in applications lack such solutions, necessitating broader notions. Two widely studied generalizations are \textit{ weak solutions}, defined via integration by parts, and \textit{viscosity solutions}, based on generalized pointwise derivatives. When both notions  of solutions can be meaningfully defined, establishing their equivalence is crucial. This issue has been extensively studied, beginning with \cite{Ishii95}. The equivalence of viscosity and weak solutions for the $p$-Laplace equation and its parabolic version was first proven in \cite{equivalence_plaplace},
with an alternative proof in the elliptic case given in \cite{newequivalence}.
More recent contributions include equivalence results for the normalized 
$p$-Poisson equation \cite{OptimalC1}, the non-homogeneous 
$p$-Laplace equation \cite{chilepaper}, and the normalized $p(x)$-Laplace equation \cite{siltakoski18}. In the parabolic setting, equivalence for radial solutions has been addressed in \cite{ParvVaz}.  We also mention that an unpublished
version of \cite{lindqvist12reg} applies \cite{newequivalence} to
sketch the equivalence of solutions to (\ref{eq:p-para f}) when $a(z)\equiv 0$ and $f\equiv 0$ in the
with $p\geq2$. lastly, the equivalence of weak and viscosity solution for non homogeneous elliptic double phase 
\begin{align*}
    -\div\left(|Du|^{p-2}Du+a(x)|Du|^{q-2}Du\right)=f(x, u, Du),
\end{align*}
where $1<p\leq q\leq p+p/n$ and $f$ is continuous with certain growth condition has been established in \cite{FRZ24}, also see \cite{FZ22}.

The comparison principle plays a central role in proving such equivalences. For quasilinear parabolic equations, various comparison results are available. For instance,  \cite{junning90} establishes a comparison principle for
for $\smash{\partial_{t}u-\operatorname{div}\left(|Du|^{p-2}Du\right)+f(u,x,t)=0}$ when $p>2$ and
$\smash{f}$ is a continuous function such that $\smash{\left|f(u,x,t)\right|\leq g(u)}$
for some $\smash{g\in C^{1}}$. The homogeneous case for the $p$ parabolic equation is also considered in \cite{KL96}
and the general equation $\smash{\partial_{t}u-\div\mathcal{A}(x,t,Du)=0}$ treated
in \cite{korteKuusiParv10}. Equations with  explicit gradient dependence, such as
$$\smash{\partial_{t}u-\operatorname{div}\left(|Du|^{p-2}Du\right)-\left|Du\right|^{\beta}=0}$$
have been studied in \cite{Attouchi12} for $\smash{p>2}$ and $\smash{\beta>p-1}$. In the works
\cite{bobkovTakac14,bobkovTakac18}, both positive results and counter
examples are provided for the comparison, strong comparison and maximum
principles for the equation $$\smash{\partial_{t}u-\operatorname{div}\left(|Du|^{p-2}Du\right)-\lambda\left|u\right|^{p-2}u-f(x,t)=0}.$$

At this point, we note that no prior work has addressed comparison principles or investigated the equivalence between weak and viscosity solutions for parabolic double phase problems of the form \eqref{eq:p-para f}. This gap in the literature motivates our consideration of questions \textit{Q1} and \textit{Q 3}.

Now let us turn our attention to the regularity theory of elliptic and parabolic double phase equations. This classical model, originally introduced by Zhikov for the elliptic case \cite{ Zhi93, Zhi95, Zhi97}, has been widely used in the study of double phase problems, which arise in the modelling of strongly anisotropic materials. These equations have been widely studied and continue to attract significant interest. In the elliptic case, a broad spectrum of regularity results-such as $C^{\alpha}$-estimate, $ C^{1, \alpha}$-estimate, Harnack inequality, have been established in pioneering works \cite{BO20, CM15a, CM15b, BCM15}, under the assumption that the coefficient $a(\cdot)$ is H\"{o}lder continuous. Moreover, Lipschitz continuity for viscosity solutions has been obtained in \cite{FRZ24} when $0<a(x)\in C^1(\mathbb{R})$ .

When $a(\cdot)\equiv 0$ and $f\equiv 0,$ the equation \eqref{eq:p-para f} reduces to the classical parabolic $p$-Laplace equation for which $C^{\alpha}, C^{1, \alpha}$- estimates  are now classical results, thanks to the seminal works \cite{DF84, DF85a, DF85b, dibenedetto93}. On the other hand, when $a(\cdot)\not\equiv 0$, regularity results such as gradient higher integrability \cite{KKM23, KS24}, Calder\'{o}n-Zygmund estimates \cite{Kim24, Kim25}, and H\"{o}lder continuity \cite{KMS25, QL25} have been obtained only recently. We note that, in analogy with the elliptic setting, the coefficient $a(\cdot)$ is assumed to be H\"{o}lder continuous in all these works.

In this context, we are interested in establishing regularity results for equation \eqref{eq:p-para f} that go beyond H\"{o}lder continuity, particularly in the presence of lower-order terms. This motivates our investigation of question \textit{Q 2}.

\subsection{Main results} In this section, we state our main results and compare them with the existing literature. Our first result is local Lipschitz regularity for \eqref{eq:p-para f} under the natural range of exponents.

\begin{thm} \label{Lip thm}
Suppose that $1 < p \leq q \leq p + 1$. Let the coefficient $a\geq0$ be continuous in time and Lipschitz continuous in space. Furthermore, assume that there is $C_a \geq 0$ such that
\begin{equation}\label{eq:a monocnd}
    a(x, t)\leq C_a a(x, s) \text{ for all } (x, t), (x, s)\in B_1\times (-1, 0) \text{ with } t \leq s.
\end{equation}
Let u be a continuous and bounded weak solution to \eqref{eq:p-para f} in $B_1 \times (-1, 0)$. If $q = p + 1$, suppose additionally that $u$ is uniformly continuous in space variable.
Then we have the following local estimate:
\[
    |u(x,t) - u(y,s)| \leq C\left(|x-y|+|t-s|^{\alpha}\right) \quad\text{for all}\quad (x, t), (y, s) \in B_{1/2}\times(-1/2, 0),
\]
where
\[
    \alpha=\begin{cases}
    \frac{p}{p+q} & \text{if }1<p<2,\\
    1/2 & \text{if }p\geq2.
    \end{cases}
\]If $q < p + 1$, then the constant $C$ depends only on $N$, $p$, $q$, $\Vert a \Vert_{L^\infty}$, $\Vert Da \Vert_{L^\infty}$, the constants appearing in the growth condition \eqref{eq:gcnd}, and $\osc_{B_1 \times (-1, 0)} u$. If $q = p + 1$, then the constant $C$ also depends on the optimal modulus of continuity of $u$ in space in $B_1 \times (-1,0)$.
\end{thm}
Combining the above theorem with H{\"older} estimates of bounded weak solutions \cite{KMS25} and Lemma 5.2, we have the following immediate corollary in the degenerate regime.
\begin{cor}
Suppose that $2\leq p\leq q\leq p+1$. Assume that $a\geq0$
is continuous in time,  Lipschitz in space, and that \eqref{eq:a monocnd} holds. Suppose
that $u$ is a bounded weak solution to \eqref{eq:p-para f} in $B_{1}\times(-1,0)$
with $f\equiv0$. Then 
\[
\left|u(x,t)-u(y,s)\right|\leq C\left(\left|x-y\right|+|t-s|^{\frac{1}{2}}\right)\quad\text{for all }(x,t),(y,s)\in B_{1/2}\times(-1/2,0).
\]
The constant $C$ depends only on $N,p,q$, $\left\Vert a\right\Vert _{L^{\infty}},\left\Vert Da\right\Vert _{L^{\infty}}$
and $\osc_{B_{1}\times(-1,0)}u$.
\end{cor}

Several remarks are in order to place our result in the context of the existing literature.

\begin{rem}
(i)\,  In \cite{BDM13}, it is shown that the weak solution \( u \in L^q_{\mathrm{loc}}(0, T; W^{1,q}_{\mathrm{loc}}(\Omega)) \) of \eqref{eq:p-para f} with bounded and space independent coefficients, i.e., $a(x, t)=a(t)$ and $f\equiv 0$
satisfies a quantitative Lipschitz bound for the range \( 2 \leq p \leq q \leq p + \frac{4}{n} \).
A similar result for the singular case was established in \cite{Sin15}. More precisely, it is shown there that for \( a(x, t) = a(t) \) and \( \frac{2n}{n+2} < p \leq 2 \), the solution of \eqref{eq:p-para f} with the same assumptions above is Lipschitz continuous for \( p \leq q \leq p + \frac{4}{n+2} \). However, no precise result is available in those works for the general case of space-time dependent coefficients.

Our result, in particular, applies to viscosity solutions. Furthermore, the range of admissible \( p \) and \( q \) in our setting is optimal and consistent with the H\"older regularity result established in \cite{KMS25}.\\

\noindent (ii)\, In \cite{KMS25}, the authors study the regularity of solutions to \eqref{eq:p-para f} under the assumption that the coefficient $a$ is H\"{o}lder continuous in both space and time. Under this assumption and $2\leq p\leq q\leq p+\alpha,$ they prove that the solution exhibits H\"{o}lder continuity, 
where $\alpha \in (0, 1]$ is the H\"{o}lder exponent of the coefficient $a.$ This range is known to be optimal in general from the regularity theory of elliptic double phase equations.

In contrast, in this paper, we consider a stronger regularity assumption in space but weaker in time on the coefficient: we assume that $a$ is Lipschitz continuous in space and continuous in time. Since Lipschitz continuity corresponds to the case $\alpha=1,$ our result on the Lipschitz continuity of solutions aligns naturally with the framework of \cite{KMS25}, and can be viewed as a sharp extension of their result to the endpoint case where the H\"{o}lder exponent attains its maximum. We point out that the local monotonicity type assumption \eqref{eq:a monocnd} on $a$ is required only to prove comparison principle, in particular to obtain ``time approximation'' (cf. Lemma \ref{lem:steklov lemma}). As of now, to the best of our knowledge, some control over the ``zero set of $a$'' is essential in this context.

Beyond establishing spatial Lipschitz regularity, we also obtain improved regularity in time, particularly for the range $2\leq p\leq q<\infty$ and for $f\equiv 0. $ Additionally, for the singular range $1<p<2,$ we impose the assumption that the solutions to \eqref{eq:p-para f} with $f\equiv 0$ are continuous. This assumption is essential for formulating the notion of viscosity solutions in that setting. We note that proving the existence of continuous solutions to \eqref{eq:p-para f} in the singular regime requires substantially different techniques and falls outside the primary scope of this paper. We therefore leave this aspect for future investigation.

\medskip
\noindent (iii)\, Our approach to proving Theorem \ref{Lip thm} is based on the celebrated Ishii-Lions method, which is a cornerstone technique in the theory of viscosity solutions. Although Theorem \ref{Lip thm} is stated under the assumption of continuous weak solutions, we apply the Ishii-Lions technique to a more general auxiliary class of functions. Specifically, we consider the class $\texttt{S}$ which is introduced in Section \ref{sec:Lipschitz estimates} of the paper. Roughly speaking, this auxiliary class contains functions that are both viscosity subsolutions to a relaxed ``upper'' equation of \eqref{eq:p-para f} and viscosity supersolutions to a relaxed ``lower'' equation of \eqref{eq:p-para f}. The idea is that we do not need continuity from the lower-order terms of the equation to apply the Ishii-Lions method -- we only need to bound them. In Lemma \ref{cws and S}, we show that weak solutions are contained within this class whenever they obey a comparison principle.

\medskip
\noindent (iv)\, Another important direction in the theory of partial differential equations involving different phases is the study of multi-phase equations. In such equations, the primary challenge lies in understanding the interaction between the different phases, each governed by distinct growth behavior. While the regularity theory for elliptic multi-phase problems is now fairly well-developed-see, for instance, \cite{BBO21, DFO19, DF22, FRZ22} the parabolic counterpart has only recently begun to attract attention.
The prototypical elliptic multi-phase equation takes the form:
\begin{align}\label{multi-phase}
-\operatorname*{div}\left(|Du|^{p-2}Du + \sum_{i=1}^m a_i(x)|Du|^{p_i-2}Du\right)
 =0
\end{align}
in $\Omega$, where $1<p\leq p_1\leq \cdots \leq p_m$ and $0 \leq a_i(\cdot)\in C^{0,\alpha_i}(\overline{\Omega})$ with $\alpha_i \in (0,1]$ for all $i\in\{1,\cdots, m\}$
The regularity theory for the parabolic analogue of this equation has only recently begun to emerge, with contributions in \cite{Sen25, KOS25, KO24}. We sketch the Ishii-Lions method to establish Lipschitz regularity for viscosity solutions in this setting, at least for $m=2.$ Naturally, such an approach would require the coefficients $a_i$ to satisfy the assumptions of Theorem \ref{Lip thm}.
\end{rem}

Next we discuss our second result of the paper regarding the relationship of viscosity and weak solutions.

To show that weak supersolutions are viscosity supersolutions, we adapt the argument of \cite{equivalence_plaplace}, which relies on a comparison principle for weak solutions. However, we could not find a correct reference for comparison principle for the equation (\ref{eq:p-para f}). To the best of our knowledge, a comparison principle covering \eqref{eq:p-para f} in the present generality (with space-time dependent modulating coefficient and a gradient nonlinearity on the right hand side) is not available in the literature; in fact, deriving such a result even with $f\equiv 0$ appears to be open. The main difficulty stems from the $(x,t)$-dependence of $a$  and the time part of the equation. Nevertheless, we show that a comparison principle holds under the monotonicity condition \eqref{assumption increasing}. This way, we are able to derive the following theorem.
\begin{thm}\label{thm:weak is visc_intro}
Let $1<p\leq q\leq p+1$. Suppose that $a:\Xi \rightarrow [0,\infty)$ is locally Lipschitz in space, continuous in time. Suppose that for any $(x,t) \in \Xi$, there is $r > 0$ such that $a$ satisfies \eqref{assumption increasing} in $Q_r(x,t)$ with some constant $C_a \geq 0$. Let $u\in C(\Xi)$ be a local weak supersolution to (\ref{eq:p-para f})
in $\Xi$. Then $u$ is a viscosity supersolution to (\ref{eq:p-para f}) in $\Xi$.
\end{thm}

To show that viscosity supersolutions are weak supersolutions, we adopt the technique introduced by Julin and Juutinen \cite{newequivalence}. In contrast to \cite{equivalence_plaplace}, we do not rely on the
uniqueness machinery for viscosity solutions. Instead, we approximate a viscosity supersolution $u$ by its inf-convolution $\smash{u_{\varepsilon}}$ (see also \cite{Sil21} for the parabolic $p$-Laplace equation with gradient nonlinearity). The key point is that $u_{\varepsilon}$ is a viscosity supersolution to a \emph{perturbed} equation, rather than to the original equation, in a smaller domain. Due to the $(x,t)$-dependence of the coefficients in \eqref{eq:p-para f}, the inf-convolution introduces an error term $E_{\varepsilon}(x,t)$, which is uniformly bounded and converges to zero almost everywhere as $\varepsilon \to 0$. Together with the pointwise properties of the inf-convolution, this implies that $\smash{u_{\varepsilon}}$ is also a weak supersolution in the smaller domain. We then prove a Caccioppoli-type estimate, which yields the convergence $Du_{\varepsilon} \to Du$ in a suitable Sobolev space. Finally, in Theorem~\ref{final thm: visc is weak}, we show that for $1 < p \leq q < p+1$, every bounded viscosity supersolution is a weak supersolution; in the borderline case $q - p = 1$, the same conclusion holds for (bounded) solutions. To control the error terms generated by inf-convolution, we need the following condition
\begin{equation} \label{eq:visc is weak acnd}
    \begin{cases} 
    \Gamma \subset \mathbb {R}^{N+1} \text{ is compact and } Da \text{ is continuous in } \mathbb R^{N+1} \setminus \Gamma,\\
    {\partial\{(x,t) \in \mathbb R^{N+1}:a(x,t)=0}\}\text{ has Lebesgue measure zero,}
    \end{cases}
\end{equation}
where, as before, $\Gamma$ is the set of points in $\mathbb R^{N+1}$ outside of which $a$ is differentiable in space. For discussion on why we need \eqref{eq:visc is weak acnd}, see Section \ref{sec: visc to weak}.

\begin{thm}\label{final thm: visc is weak_intro}
Let $1<p\leq q \leq p+1$. Suppose that $a:\Xi \rightarrow [0, \infty)$ is locally Lipschitz in space and continuous in time, and that \eqref{eq:visc is weak acnd} holds. Let $u$ be a bounded viscosity supersolution to
(\ref{eq:p-para f}) in $\Xi$. If $q = p +1$, suppose, moreover, that $u$ is locally Lipschitz continuous in space. Then $u$ is a local weak supersolution
to (\ref{eq:p-para f}) in $\Xi$.
\end{thm}

If $a$ satisfies assumptions of both of the Thoerems \ref{thm:weak is visc_intro} and \ref{final thm: visc is weak_intro}, then we have the following equivalence result for solutions.

\begin{cor}
Assume that $1<p\leq q \leq p+1$. Suppose that $a$ satisfies the same assumptions as in Theorem \ref{thm:weak is visc_intro} and \ref{final thm: visc is weak_intro}. Then $u \in C(\Xi)$ is a local weak solution to \eqref{eq:p-para f} in $\Xi$ if and only if it is a viscosity solution to \eqref{eq:p-para f} in $\Xi.$ 
\end{cor}




The paper is organized as follows.
\begin{itemize}
    \item In Section \ref{sec: prelims}, we set the notation and define solution spaces with definitions of weak and viscosity solutions.
    \item The main contribution of Section  \ref{sec: weak to visc} is to show that weak solutions are viscosity solutions. For that, we first prove comparison principle in Section \ref{sebsec: caomparison principle}.
    \item Section \ref{sec:Lipschitz estimates} is devoted to prove the spatial Lipschitz estimate whereas in Section \ref{sec: time holder}, we prove time H\"{o}lder regularity and complete the proof of Theorem \ref{Lip thm}. Moreover, Section \ref{sec: multiphase} sketches the similar regularity results for parabolic multi-phase equations.
    \item Lastly, in Section \ref{sec: visc to weak}, we show that the viscosity solutions are weak solutions.
\end{itemize}

\section{Preliminaries}\label{sec: prelims}
\subsection{Notations} We use the following notations throughout this paper.
\begin{itemize}
\item The symbols $\Xi$ and $\Omega$ are reserved for bounded domains
in $\smash{\mathbb{R}^{N}\times\mathbb{R}}$ and $\smash{\mathbb{R}^{N}}$,
respectively. 

\item By $dz$ we mean integration with respect to space and
time variables, i.e. $dz=dx\d t$. 
\item For $z:=(x, t)\in \Xi$ and $\xi \in \mathbb{R}^N,$ we shall use the notation
\[H(z, \xi)=|\xi|^p+a(z)|\xi|^q\] and define the sets
\begin{align*}
L^H(\Xi) &:= \left\{ u\in L^1(\Xi): \int_\Xi |u|^p+a(z)|u|^q\d z < \infty \right\}, \\
W^H_{\mathcal P}(\Xi) & :=\left\{ u\in L^1(\Xi):\int_{\Xi}\left|Du\right|^{p}+a(z)\left|Du\right|^q\d z<\infty\right\},
\end{align*}
where $Du$ denotes the distributional weak gradient of $u$ in $\Xi$ in space variable.
\item We shall denote 
\[\Gamma:= \left\{(x, t)\in \Xi : Da(x, t)\,\, \text{does not exist in the classical sense}\,\right\}.\]
\item For $r>0$ and $(x_{0},t_{0})\in\mathbb{R}^{N+1}$,
we denote the parabolic cylinders
\begin{align*}
Q_{r}(x_{0},t_{0}) & :=B_{r}(x_{0})\times(-r+t_{0},t_{0})\quad\text{and}\quad Q_{r}:=Q_{r}(0,0).
\end{align*}
We also set $Q_r^-(x_0 ,t_0) := B_r(x_0,t_0)\times(-r+t_0,t_0]$ and $Q_r^-:=Q_r^-(0,0)$. The parabolic boundary of a cylinder is denoted by
\begin{equation*}
    \partial_p Q_r^-(x_0,t_0) = \partial_p Q_r(x_0 ,t_0 ) := (\overline B_r(x_0)\times\{-r+t_0\})\cup(\partial B_r(x_0)\times [-r+t_0,t_0]).
\end{equation*}
\item For all $(x,t)\in Q_{1}$, $\eta\in\mathbb{R}^{N}$ and $X\in S(N)$ where $S(N)$ is the set of all $N\times N$ symmetric matrices, we set
\begin{align}
F((x,t),\eta,X):= & \:\left|\eta\right|^{p-2}\left(\tr X+(p-2)\frac{\eta^{\prime}X\eta}{\left|\eta\right|^{2}}\right)
+a(x,t)\left|\eta\right|^{q-2}\left(\tr X+(q-2)\frac{\eta^{\prime}X\eta}{\left|\eta\right|^{2}}\right)\label{eq:F}
\end{align}
and
\begin{align*}
g((x,t),\eta) & :=\left\Vert Da\right\Vert _{L^{\infty}}\left|\eta\right|^{q-1}+C_f\left(1+\left|\eta\right|^{\beta_{1}}+a(x,t)\left|\eta\right|^{\beta_{2}}\right),
\end{align*}
where $C_f\geq0$ and $1\leq\beta_{1}< p$, $1\leq\beta_{2}<q$.
\end{itemize}

\noindent\textbf{Growth condition:}
The function $f: \Omega \times \mathbb{R}^{+} \times \mathbb{R}^N \to \mathbb{R}$ is
assumed to be continuous and satisfy the growth condition
\begin{align}\tag{{G}}\label{eq:gcnd}
    |f(z,  Du)|\leq C_f \left(1+|Du|^{\beta_1}+a(z)|Du|^{\beta_2}\right),
\end{align}
where $1\leq\beta_1<p$ and $1\leq\beta_2<q$.
\vspace{0bp}

\begin{defn}[Bounded weak solutions]\label{def:weak solutions}
A function $u:\Xi\rightarrow\mathbb{R}$ is a \textit{weak super}(or \textit{sub})\textit{solution}
to (\ref{eq:p-para f}) in $\Xi$ if $u\in W^{H}(\Xi)\cap L^\infty(\Xi)$ and
\begin{equation}
\int_{\Xi}-u\partial_{t}\varphi+\left(\left|Du\right|^{p-2}Du+a(z)\left|Du\right|^{q-2}Du\right)\cdot D\varphi-\varphi f(z, Du)\d z\geq (\leq )\, 0\label{eq:p-para-f weak eq}
\end{equation}
for all non-negative \textit{test functions} $\varphi\in C_0^\infty(\Xi)$.
We say that $u$ is a \textit{weak solution} if it is both a weak supersolution and a weak subsolution. If $u$ is a weak supersolution to \eqref{eq:p-para f} in every $\Xi^\prime \Subset \Xi$, then we say that it is a \textit{local weak supersolution to \eqref{eq:p-para f} in $\Xi$}. Local weak subsolutions and solutions are defined analogously.
\end{defn}
To define viscosity solutions to (\ref{eq:p-para f}), observe that for
all $\varphi\in C^{2}$ with $D\varphi\not=0$, we have
\begin{align*}
&\div\left(|D\varphi|^{p-2}D\varphi\right)+\div\left(a(z)|D\varphi|^{q-2}D\varphi\right)\\
&=\left|D\varphi\right|^{p-2}\left(\Delta\varphi+\frac{p-2}{\left|D\varphi\right|^{2}}\left\langle D^{2}\varphi D\varphi,D\varphi\right\rangle \right)\\
&+ a(z)|D\varphi|^{q-2}\left(\Delta \varphi+\frac{q-2}{|D \varphi|^2}\left\langle D^{2}\varphi D\varphi,D\varphi\right\rangle\right)+|D\varphi|^{q-2}D\varphi \cdot D a(z).
\end{align*}

Note that the right-hand side is singular if $p<2$ and the gradient of $\varphi$ vanishes. Moreover, $Da$ is not defined on $\Gamma$. Because of these issues, we use a limit procedure in the definition of viscosity solutions. 

\begin{defn}[Viscosity solution]
A lower semicontinuous and bounded function $u:\Xi\rightarrow\mathbb{R}$
is a \textit{viscosity supersolution} to (\ref{eq:p-para f}) in $\Xi$
if whenever $\varphi\in C^{2}(\Xi)$ and $(x_{0},t_{0})\in\Xi$ are
such that
\[
\begin{cases}
\varphi(x_{0},t_{0})=u(x_{0},t_{0}),\\
\varphi(x,t)<u(x,t) & \text{when }(x,t)\not=(x_{0},t_{0}),\\
D\varphi(x,t)\not=0 & \text{when }x\not=x_{0},
\end{cases}
\]
then
\begin{align*}
\limsup_{\substack{\substack{(x,t)\rightarrow(x_{0},t_{0})\\
x\not=x_{0}, (x,t)\not \in \Gamma
}
}
}&\Bigg[\partial_{t}\varphi(x,t)-\div\left(|D \varphi(x, t)|^{p-2}D\varphi(x, t)+a(x, t)|D \varphi(x, t)|^{q-2}D\varphi(x, t)\right)\\
&-f\left(x, t, D\varphi(x, t)\right)\Bigg]\geq0.    
\end{align*}
In this case, we say that $\varphi$ \textit{touches $u$ from below at} $\smash{(x_{0},t_{0})}$.

An upper semicontinuous and bounded function $u:\Xi\rightarrow\mathbb{R}$
is a \textit{viscosity subsolution} to (\ref{eq:p-para f}) in $\Xi$
if whenever $\varphi\in C^{2}(\Xi)$ and $(x_{0},t_{0})\in\Xi$ are
such that
\[
\begin{cases}
\varphi(x_{0},t_{0})=u(x_{0},t_{0}),\\
\varphi(x,t)>u(x,t) & \text{when }(x,t)\not=(x_{0},t_{0}),\\
D\varphi(x,t)\not=0 & \text{when }x\not=x_{0,}
\end{cases}
\]
then
\begin{align*}
\liminf_{\substack{\substack{(x,t)\rightarrow(x_{0},t_{0})\\
x\not=x_{0}, (x,t)\not \in \Gamma
}
}
}&\Bigg[\partial_{t}\varphi(x,t)-\div\left(|D \varphi(x, t)|^{p-2}D\varphi(x, t)+a(x, t)|D \varphi(x, t)|^{q-2}D\varphi(x, t)\right)\\
&-f\left(x, t, D\varphi(x, t)\right)\Bigg]\leq0.    
\end{align*}
In this case, we say that $\varphi$ \textit{touches $u$ from above at} $\smash{(x_{0},t_{0})}$.
A function that is both viscosity sub- and supersolution is a \textit{viscosity
solution}.
\end{defn}

We show in Section \ref{sec: visc to weak} that if $a$ is $C^1$ in $\mathbb{R}^{N+1}\setminus \Gamma$ and the boundary of the zero set of $a$ has measure zero, then bounded viscosity super- and subsolutions are local weak super- and subsolutions. However, for our regularity results, such control over the boundary of the zero set of $a$ is not needed. For this reason, it is convenient to define a relaxed class of solutions that contains both continuous weak and viscosity solutions, called as \textit{Class \texttt{S}} (see Section \ref{sec:Lipschitz estimates} for the definition).

\section{Weak solutions are viscosity solutions}\label{sec: weak to visc}

\subsection{Comparison principles}\label{sebsec: caomparison principle}

To show that weak solutions are viscosity solutions, we need some kind of comparison principle for weak solutions. To this end, we have to use a weak solution as a test function for itself. However, since in general weak solutions lack smoothness, an approximation is required. In the case of the elliptic $p$-Laplace equation, the approximation is straightforward since the solution is in a Sobolev space in which smooth functions are dense. On the other hand, with the $p$-parabolic equation, one has to deal with a time derivative which the solution, let alone a super- or subsolution, may not possess. This issue is often remedied by mollifying the equation in time or by using Steklov averages, see \eqref{defn: steklov avg}. However, for the parabolic double-phase equation, approximation is more complicated. Indeed, already in the elliptic setting, the regularity of the function $a$ and the difference $q-p$ are connected to the density of smooth functions in the Sobolev space corresponding to the functional $\xi \mapsto |\xi|^p+a(x)|\xi|^q$, see for example \cite{B24} and the references therein. With the parabolic double-phase equation, a novel difficulty occurs with time-mollification: even if $a$ is Lipschitz and independent of $x$, the Steklov average of a function $u\in W^H(\Xi)\cap C(\Xi)$ may fail to be in $W^H(\Xi)$, see the example below.

For $\varphi:\Xi\rightarrow\mathbb{R}$ and $h>0$, we define the left and right Steklov averages
\begin{align}\label{defn: steklov avg}
    [\varphi]_h(x,t):=\frac{1}{h}\int_{t-h}^t \varphi(x,s)\d s\quad\text{and}\quad [\widetilde \varphi]_h(x,t) := \frac{1}{h}\int_t^{t+h} \varphi(x,s)\d s,
\end{align}
where $\varphi$ is extended by zero to $\mathbb R^{N+1}$. 

\begin{example} Let $1\leq p<q<\infty$. Then in general, $u \in W^H(\Xi)\cap C(\Xi)$ does not imply $[u]_h \in W^H(\Xi)$. To see this, let $N=1$ and $\Xi := (0,1)\times(-1,1)$. Define $a:\Xi\rightarrow[0,\infty)$, $a(x,t):=\max(-t,0)$.
For small $\epsilon>0$, consider the function $u(x,t):=\left|x\right|^{1-\frac{1}{p}+\epsilon}\max(t,0).$
Then
\[
\left|Du(x,t)\right|=(1-\frac{1}{p}+\epsilon)\left|x\right|^{-\frac{1}{p}+\epsilon}\max(t,0)
\]
from which it follows that $u\in W^{H}(\Xi)$ since
\[
\int_{-1}^{1}\int_0^1\left|Du\right|^{p}+\max(-t,0)\left|Du\right|^{q}\d x\d t=\int_{-1}^{1}\int_0^1\left|Du\right|^{p}\d x\d t<\infty.
\]
On the other hand, for $t\in(-h,0)$, we have
\begin{align*}
D[u]_{h}(x,t)=[Du]_{h}(x,t) & =\frac{1}{h}\int_{t}^{t+h}(1-\frac{1}{p}+\epsilon)\left|x\right|^{-\frac{1}{p}+\epsilon}\max(s,0)\d s\\
 & =\frac{1}{h}(1-\frac{1}{p}+\epsilon)\left|x\right|^{-\frac{1}{p}+\epsilon}\int_{t}^{t+h}\max(s,0)\d s,\\
 & \geq \left(\frac{(t+h)^{2}}{2h}-\frac{t^2}{2h}\right)(1-\frac{1}{p}+\epsilon)\left|x\right|^{-\frac{1}{p}+\epsilon}.
\end{align*}
Thus,
\begin{align*}
\int_{-1}^{1}\int_0^1\max(-t,0)\left|D[u]_{h}\right|^{q}\d x \d t & \geq \int_{-h}^{-h/2} \int_{0}^{1} \frac{h}{2}|D[u]_h|^q\ \d x \d t \\ &
\geq\int_{-h}^{-h/2}\int_0^1\frac{h}{2}\left(\frac{h+2t}{2}(1-\frac{1}{p}+\epsilon)\right)^{q}\left|x\right|^{-\frac{q}{p}+\epsilon}\d x \d t,
\end{align*}
where the second integral is not finite since $-\frac{q}{p}+\epsilon<-1$.
Hence, $[u]_{h}\not\in W^{H}(\Xi)$.
\end{example}

Nevertheless, we can establish a time-approximation result if $a$ satisfies one of the conditions \eqref{eq:a decreasing} or \eqref{eq:a increasing} below. The idea is, roughly speaking, that if $a$ is decreasing in time in $Q_r$, we can use the left Steklov average to approximate $u\in W^H(Q_r)$. On the other hand, if $a$ is increasing in $Q_r$, we can use the right Steklov average. Observe that if $a$ is bounded away from zero or identically zero in $Q_r$, then both \eqref{eq:a decreasing} and \eqref{eq:a increasing} are trivial.

\begin{lem}\label{lem:steklov lemma}
Let $\rho>0$ and $a:Q_\rho \rightarrow [0,\infty)$ be bounded and measurable. Suppose moreover that there exists a constant $C_a \geq 0$, independent of $(x,t)$, such that 
\begin{equation}\label{eq:a decreasing}
    a(x,t) \leq C_a a(x,s) \quad \text{for all } (x,t),(x,s)\in Q_\rho \text{ with } s \leq t.
\end{equation}
or
\begin{equation}\label{eq:a increasing}
    a(x,t) \leq C_a a(x,s) \quad \text{for all } (x,t),(x,s)\in Q_\rho \text{ with } s \geq t.
\end{equation}
Let $u \in W^H(Q_\rho)$. Then, if (\ref{eq:a decreasing}) holds, for any small $h_{0}>0$, we have $[u]_h \in W^H(B_\rho \times(-\rho+h_0, 0))$ for all $h\in(0,h_0)$. Further, up to a subsequence
\[
\lim_{h\rightarrow0}\int_{B_{\rho}\times(-\rho+h_{0},0)}\left|D[u]_h-Du \right|^p+a(z)\left|D[u]_h-Du\right|^{q} \d z =0.
\]
If (\ref{eq:a increasing}) holds, then we have the analogous result for $[\widetilde u]_h$.
\end{lem}
\begin{proof}
We only consider the case \eqref{eq:a decreasing} and $[u]_h$ as the proof for $[\widetilde u]_h$ when \eqref{eq:a increasing} holds is  analogous. 
Fix $h_{0}>0$ and suppose that $h<h_{0}$. For $i \in \{1,\ldots,N\}$, we denote $v := \partial_i u$. By standard properties of Steklov averages, we have $\partial_i[u]_h = [\partial_i u]_h = [v]_h$. Then for almost every $(x,t)\in B_{\rho}\times(-\rho+h_{0},0)$ we have
\begin{equation}\label{eq:time approx 1}
a(x,t)\left|[v]_{h}(x,t)-v(x,t)\right|^{q}\leq C(q)(a(x,t)\left|[v]_{h}(x,t)\right|^{q}+a(x,t)\left|v(x,t)\right|^{q}),
\end{equation}
where by (\ref{eq:a decreasing}), it holds that
\begin{align}\label{eq:time approx 2}
a(x,t)\left|[v]_{h}(x,t)\right|^{q}  =\left|\frac{1}{h}\int_{t-h}^{t}a^{1/q}(x,t)v(x,s)\d s\right|^{q}
  &\leq\left(\frac{1}{h}\int_{t-h}^{t}C_{a}^{1/q}a^{1/q}(x,s)\left|v(x,s)\right|\d s\right)^{q}\nonumber\\
 & =C_{a}([a^{1/q}\left|v\right|]_{h}(x,t))^{q}.
\end{align}
Observe that since $a^{1/q}|v| \in L^q(Q_\rho)$, we have $[a^{1/q}|v|]_h \in L^q(B_\rho \times (-\rho +h_0, \rho))$ by standard properties of Steklov average \cite[Lemma 2.4]{CDG17}, and so the above estimate implies that 
\[
    \int_{B_\rho \times (-\rho +h_0, \rho)} a(x,t) |[v]_h (x,t)|^q \d x \d t < \infty.
\]
Therefore, $[v]_h\in W^H(B_\rho \times(-\rho +h_0, \rho)).$
On the other hand, combining \eqref{eq:time approx 1} and \eqref{eq:time approx 2}, we have for almost every $(x,t)\in B_{\rho}\times(-\rho+h_{0},0)$
that
\[
a(x,t)\left|[v]_{h}(x,t)-v(x,t)\right|^{q}\leq C_{a}([a^{1/q}\left|v\right|]_{h}(x,t))^{q}+a(x,t)\left|v(x,t)\right|^{q}.
\]
Note that since $a^{1/q}\left|v\right|\in L^{q}(Q_{\rho})$, it follows from the standard properties of Steklov average (see \cite[Lemma 2.4 and Lemma 2.5]{CDG17}) that
\[
\lim_{h\rightarrow0}\int_{B_{\rho}\times(-\rho+h_{0},0)}C_{a}([a^{1/q}\left|v\right|]_{h})^{q}\d x\d t=\int_{B_{\rho}\times(-\rho+h_0, 0)}C_{a}(a^{1/q}\left|v\right|)^{q}\d x\d t.
\]
Moreover, since $u\in L^p(B_{\rho}\times(-\rho+h_{0},0))$, we have
$[v]_{h}\rightarrow v$ almost everywhere in $B_{\rho}\times(-\rho+h_{0},0)$
up to a subsequence. It now follows from a well known generalization
of Lebesgue's convergence theorem that
\[
\lim_{h\rightarrow0}\int_{B_{\rho}\times(-\rho+h_{0},0)}a(x,t)\left|[v]_{h}(x,t)-v(x,t)\right|^{q}\d x\d t=0.\qedhere
\]
\end{proof}

Next we consider space approximation. The following theorem is proven similarly to the elliptic result in \cite{BGS22}, but we give details for the benefit of the reader.

\begin{thm} \label{lem: spcae approximation}
Let $1< p\leq q\leq p+1$ and $r>0$. Assume that $a:Q_{r}\rightarrow[0,\infty)$
is bounded, measurable, and Lipschitz continuous in space. Suppose
that $u\in W^{H}(Q_{r})\cap L^{\infty}(-r,0;L^{1}(B_{r}))$ is compactly
supported in $Q_{r}$. For $\delta>0$, define
\[
u_{\delta}(x,t):=\int_{\mathbb{R}^{N}}u\left(x-y,t\right)\varrho_{\delta}(y)\d y,
\]
where $\varrho_{\delta}$ is the standard mollifier and the integrand
is read as zero when $x-y\not\in Q_{r}$. Then $\supp u_{\delta}\Subset Q_{r}$
for all small enough $\delta>0$. Furthermore, $u_{\delta}\in W^{H}(Q_{r})$ and
\begin{equation}
\lim_{\delta\rightarrow0}\int_{Q_{r}}\left|Du-Du_{\delta}\right|^{p}+a(z)\left|Du-Du_{\delta}\right|^{q}\d z=0.\label{eq:space approx grad conv}
\end{equation}
\end{thm}

\begin{proof}
We need to verify \eqref{eq:space approx grad conv} and that
\[\int_{Q_r}|Du_{\delta}|^p+a(z)|Du_{\delta}|^q\, dz< \infty.\]
\textbf{Step 1:} Let $d:=\operatorname{dist}(\supp u,\partial Q_{r})$, $\delta_{0}:=d/8$
and $r^{\prime}:=r-d/4$. Then $u_{\delta}\equiv0$ in $Q_{r}\setminus Q_{r^{\prime}}$
for all $\delta<\delta_{0}$. Fix $(x,t)\in Q_{r^{\prime}}$, where
$t$ is such that $Du(\cdot,t)\in L^{1}(B_{r})$ and $\left\Vert u(\cdot,t)\right\Vert _{L^{1}(B_{r})}\leq\left\Vert u\right\Vert _{L^{\infty}(-r,0;L^{1}(B_{r}))}$.
Further, assume that $\delta<\delta_{0}$ and let $x_{t}^{\ast}\in\overline{B}_{\delta}(x)$
be such that 
\[
a(x_{t}^{\ast},t)=\inf_{y\in B_{\delta}(x)}a(y,t),
\]
which is well defined since $B_{\delta}(x)\Subset B_{r}$. By Young's
convolution inequality, we have
\begin{align*}
\left|Du_{\delta}(x,t)\right|=\left|D\int_{\mathbb{R}^{N}}u\left(x-y,t\right)\varrho_{\delta}(y)\d y\right| & =\left|\int_{\mathbb{R}^{N}}u(y,t)D\varrho_{\delta}(x-y)\d y\right|\\
 & \leq\left\Vert D\varrho_{\delta}\right\Vert _{L^{1}(\mathbb{R}^{N})}\left\Vert u(\cdot,t)\right\Vert _{L^{1}(B_{r})}\\
 & \leq\delta^{-1}\left\Vert D\varrho\right\Vert _{L^{1}(\mathbb{R}^{N})}\left\Vert u\right\Vert _{L^{\infty}(-r,0;L^{1}(B_{r}))},\\
 & =:\delta^{-1}C_{0}.
\end{align*}
where we used that $\left\Vert D\varrho_{\delta}\right\Vert _{L^{1}(\mathbb{R}^{N})}=\delta^{-1}\left\Vert D\varrho\right\Vert _{L^{1}(\mathbb{R}^{N})}$.
Using the above display, Lipschitz continuity of $a$ in space variable,
as well as that $\left|x-x_{t}^{\ast}\right|\leq\delta$, we obtain
\begin{align}
 & \left|Du_{\delta}(x,t)\right|^{p}+a(x,t)\left|Du_{\delta}(x,t)\right|^{q}\nonumber \\
 & =\frac{\left|Du_{\delta}(x,t)\right|^{p}+a(x,t)\left|Du_{\delta}(x,t)\right|^{q}}{\left|Du_{\delta}(x,t)\right|^{p}+a(x_{t}^{\ast},t)\left|Du_{\delta}(x,t)\right|^{q}}\left(\left|Du_{\delta}(x,t)\right|^{p}+a(x^{\ast},t)\left|Du_{\delta}(x,t)\right|^{q}\right)\nonumber \\
 & =\left(1+\frac{(a(x,t)-a(x_{t}^{\ast},t))\left|Du_{\delta}(x,t)\right|^{q}}{\left|Du_{\delta}(x,t)\right|^{p}+a(x_{t}^{\ast},t)\left|Du_{\delta}(x,t)\right|^{q}}\right)\left(\left|Du_{\delta}(x,t)\right|^{p}+a(x^{\ast},t)\left|Du_{\delta}(x,t)\right|^{q}\right)\nonumber \\
 & \leq\left(1+\frac{\delta\left\Vert Da\right\Vert _{L^{\infty}(Q_{r})}\left|Du_{\delta}(x,t)\right|^{q}}{\left|Du_{\delta}(x,t)\right|^{p}}\right)\left(\left|Du_{\delta}(x,t)\right|^{p}+a(x_{t}^{\ast},t)\left|Du_{\delta}(x,t)\right|^{q}\right)\nonumber \\
 & \leq\left(1+\delta\left\Vert Da\right\Vert _{L^{\infty}(Q_{r})}\left(\delta^{-1}C_{0}\right)^{q-p}\right)\left(\left|Du_{\delta}(x,t)\right|^{p}+a(x_{t}^{\ast},t)\left|Du_{\delta}(x,t)\right|^{q}\right)\nonumber \\
 & =\left(1+C_{0}\left\Vert Da\right\Vert _{L^{\infty}(Q_{r})}\right)\left(\left|Du_{\delta}(x,t)\right|^{p}+a(x_{t}^{\ast},t)\left|Du_{\delta}(x,t)\right|^{q}\right)\nonumber \\
 & =:C_{1}\left(\left|Du_{\delta}(x,t)\right|^{p}+a(x_{t}^{\ast},t)\left|Du_{\delta}(x,t)\right|^{q}\right).\label{eq:space approx 1-1}
\end{align}
Further, since $\int_{B_{\delta}}\rho_{\delta}(y)\d y=1$, we can
apply Jensen's inequality to estimate
\begin{align*}
 & \left|Du_{\delta}(x,t)\right|^{p}+a(x_{t}^{\ast},t)\left|Du_{\delta}(x,t)\right|^{q}\\
 & =\left|\int_{B_{\delta}}Du(x-y,t)\rho_{\delta}(y)\d y\right|^{p}+a(x_{t}^{\ast},t)\left|\int_{B_{\delta}}Du\left(x-y,t\right)\rho_{\delta}(y)\d y\right|^{q}\\
 & \leq\int_{B_{\delta}}\left|Du(x-y,t)\right|^{p}\rho_{\delta}(y)\d y+\int_{B_{\delta}}a(x_{t}^{\ast},t)\left|Du\left(x-y,t\right)\right|^{q}\rho_{\delta}(y)\d y\\
 & \leq\int_{B_{\delta}}\left|Du(x-y,t)\right|^{p}\rho_{\delta}(y)\d y+\int_{B_{\delta}}a(y,t)\left|Du\left(x-y,t\right)\right|^{q}\rho_{\delta}(y)\d y,\\
 & =(\left|Du\right|^{p})_{\delta}(x,t)+(a\left|Du\right|^{q})_{\delta}(x,t),
\end{align*}
where the last estimate follows from the choice of $x_{t}^{\ast}$.
Combining the last two displays, we arrive at
\begin{equation}
\left|Du_{\delta}(x,t)\right|^{p}+a(x,t)\left|Du_{\delta}(x,t)\right|^{q}\leq C_{1}(\left|Du\right|^{p}+a\left|Du\right|^{q})_{\delta}(x,t)\label{eq:space approx 4-1}
\end{equation}
for all $x\in B_{r^{\prime}}$ and almost all $t\in(-r^{\prime},0)$.

\textbf{Step 2:} We define $w:\mathbb{R}^{N+1}\rightarrow\mathbb{R}$
by setting $w=(\left|Du\right|^{p}+a\left|Du\right|^{q})$ in $Q_{r}$
and $w=0$ in $\mathbb{R}^{N+1}\setminus Q_{r}.$ Since $w\in L^{1}(\mathbb{R}^{N+1})$
with compact support, we have $w_{\delta}\rightarrow w$ in $L^{1}(\mathbb{R}^{N+1})$
by standard properties of mollifiers. On the other hand, by (\ref{eq:space approx 4-1}),
we have
\[
\left|Du_{\delta}\right|^{p}+a\left|Du_{\delta}\right|^{q}\leq C_{1}w_{\delta}\quad\text{almost everywhere in }Q_{r^{\prime}}.
\]
It follows that $\left|Du_{\delta}\right|^{p}+a\left|Du_{\delta}\right|^{q}$
is uniformly integrable in $Q_{r^{\prime}}$ with respect to $\delta$.
Therefore, by Vitali's convergence theorem, 
\[
\left|Du_{\delta}\right|^{p}+a\left|Du_{\delta}\right|^{q}\rightarrow\left|Du\right|^{p}+a\left|Du\right|^{q}\quad\text{in }L^{1}(Q_{r^{\prime}}).
\]
Then observe
\begin{align*}
 & \int_{Q_{r^{\prime}}}\left|Du-Du_{\delta}\right|^{p}+a(z)\left|Du-Du_{\delta}\right|^{q}\d z\\
 & \leq\int_{Q_{r^{\prime}}}2^{p-1}(\left|Du\right|^{p}+\left|Du_{\delta}\right|^{p})+2^{q-1}a(z)(\left|Du\right|^{q}+\left|Du_{\delta}\right|^{q})\d z,
\end{align*}
so that $\left|Du-Du_{\delta}\right|^{p}+a\left|Du-Du_{\delta}\right|^{q}$
is again uniformly integrable, and by Vitali's convergence theorem,
the left-hand side of the above inequality tends to zero as $\delta\rightarrow0$.
Since both $\left|Du\right|$ and $\left|Du_{\delta}\right|$ vanish
in $Q_{r}\setminus Q_{r^{\prime}}$ (provided that $\delta<\delta_{0}$),
the claim follows.
\end{proof}

With time and space approximation at hand, we prove a comparison principle. Part of our argument for the comparison principle is inspired by the ideas of \cite{IMJ97}. Although we shall work in the range $1<p\leq q \leq p+1,$ we split the cases $1<p<q<2,$ $1<p<2\leq q$ and $2\leq p.$ To prove the  comparison principle, we only use the assumption
\begin{align}\label{assumption increasing}
a(x,t) \leq C_a a(x,s) \quad \text{for all } (x,t),(x,s)\in Q_r \text{ with } t \leq s.    
\end{align}
Note that the above assumption on $a$ is not optimal. For instance, one can easily see that the function
\[a(x, t)=\max\left\{-(x+t+1), 0\right\},\] 
does not satisfy the assumption on a cube $Q_{r}(x, t)$ centered at a point $(x, t)$ that lies on the line $x=-(t+1).$
However, a comparison principle can be proven for this $a$ by using a suitable change of variable. This approach gives hope to generalize the above assumption \eqref{assumption increasing}. We plan to return to this matter in a subsequent work.

\begin{lem}\label{lem: first comparison}
Let $1<p\leq q\leq 2< p+1,$  and $\rho >0$. Suppose that $a:Q_\rho \rightarrow [0,\infty)$ is Lipschitz in space, continuous in time, satisfies \eqref{assumption increasing} and $f(z, \eta)$ is locally Lipschitz in $\eta$.  Let $u,v\in W^{H}(Q_\rho)\cap C(\overline{Q_\rho})$ be weak sub- and supersolutions, respectively, to
\[
\partial_{t}u-\div(\left|Du\right|^{p-2}Du+a(z)\left|Du\right|^{q-2}Du)=f(z, Du)\quad\text{in }\,\,\, Q_\rho.
\]
Also assume that $Du\in L^{\infty}(Q_{\rho}).$ Then, if $u\leq v$ on $\partial_{\mathcal{P}}Q_\rho$, it follows that $u\leq v$ in $Q_{\rho}.$
\end{lem}

\begin{proof}
We split the proof into three steps. We shall start with the forward Steklov average given in \eqref{defn: steklov avg} denoted as $[\tilde{\varphi}]_h.$

\textbf{Step 1:} Suppose that $\varphi\in W^H(Q_\rho)\cap C(\overline{Q}_{\rho})$ with $\supp \varphi \Subset Q_\rho$. Then the Steklov average $[\tilde{\varphi}]_h$ is well defined for small enough $h>0$. Moreover, $[\tilde{\varphi}]_h \in W^H(Q_\rho)$ by Lemma \ref{lem:steklov lemma}.
Further, set
\[
([\tilde{\varphi}]_{h})_{\delta}(x,t):=\int_{\mathbb{R}^{N}}[\tilde{\varphi}]_{h}(y,t)\rho_{\delta}(y-x)\d y.
\]
Then by Lemma \ref{lem: spcae approximation} we have that
\begin{equation}
\int_{Q_\rho}\left|D([\tilde{\varphi}]_{h})_{\delta}-D[\tilde{\varphi}]_{h}\right|^{p}+a(z)\left|D([\tilde{\varphi}]_h)_{\delta}-D[\tilde{\varphi}]_{h}\right|^{q}\d z\rightarrow0\quad\text{as }\delta\rightarrow0.\label{eq:modular lp}
\end{equation}
Since $([\tilde{\varphi}]_{h})_{\delta}$ is Lipschitz and compactly supported
in $Q_\rho$, it can be used as a test function in the weak formulation \eqref{eq:p-para-f weak eq}. This way, we obtain
\begin{equation}
\int_{Q_\rho}-u\partial_{t}([\tilde{\varphi}]_{h})_{\delta}+\left|Du\right|^{p-2}Du\cdot D([\tilde{\varphi}]_{h})_{\delta}+a(z)\left|Du\right|^{q-2}Du\cdot D([\tilde{\varphi}]_{h})_{\delta}- ([\tilde{\varphi}]_h)_{\delta}f(z, Du)\,dz \leq 0.\label{eq:regularized est}
\end{equation}
Now, observe that by H\"{o}lder's inequality
\begin{align*}
 & \int_{Q_\rho}\left|Du\right|^{p-2}Du\cdot(D([\tilde{\varphi}]_{h})_{\delta}-D[\tilde{\varphi}]_h)+a(z)\left|Du\right|^{q-2}Du\cdot(D([\tilde{\varphi}]_{h})_{\delta}-D[\tilde{\varphi}]_{h})\d z\\
 & \leq\left(\int_{Q_\rho}\left|Du\right|^{p}\d z\right)^{\frac{p-1}{p}}\left(\int_{Q_\rho}\left|D([\tilde{\varphi}]_{h})_{\delta}-D[\tilde{\varphi}]_{h}\right|^{p}\d z\right)^{\frac{1}{p}}\\
 & \phantom{\leq}+\left(\int_{Q_\rho}a(z)\left|Du\right|^{q}\d z\right)^{\frac{q-1}{q}}\left(\int_{Q_\rho}a(z)\left|D([\tilde{\varphi}]_{h})_{\delta}-D[\tilde{\varphi}]_{h}\right|^{q}\d z\right)^{\frac{1}{q}},
\end{align*}
where the right-hand side converges to zero as $\delta\rightarrow0$
by \eqref{eq:modular lp}. Furthermore, convergence of the time term
in \eqref{eq:regularized est} is clear since $[\tilde{\varphi}]_{h}$ is Lipschitz
in time. For the nonhomogeneous term, we note that
\begin{align*}
    \int_{Q_{\rho}}f(z, Du)\left(([\tilde{\varphi}]_h)_{\delta}-[\tilde{\varphi}]_h\right)\, dz\leq \int_{Q_{\rho}}C_f\left(1+|Du|^{\beta_1}+a(z)|Du|^{\beta_2}\right)\left|([\tilde{\varphi}]_h)_{\delta}-[\tilde{\varphi}]_h\right|\, dz.
\end{align*}
Since $\varphi \in C(\overline{Q_{\rho}})$, it is bounded and $|Du|^{\beta_1}([\tilde{\varphi}]_h)_{\delta}\to |Du|^{\beta_1}[\tilde{\varphi}]_h$ pointwise, dominated convergence can be used to conclude
\[
\int_{Q_\rho}-u\partial_{t}[\tilde{\varphi}]_{h}+\left|Du\right|^{p-2}Du\cdot D[\tilde{\varphi}]_{h}+a(z)\left|Du\right|^{q-2}Du\cdot D[\tilde{\varphi}]_{h}-f(z, Du)[\tilde{\varphi}]_h\leq  0.
\]
Computing the derivative of the Steklov average and using that $\varphi$
has compact support in time, we derive
\begin{align*}
\int_{Q_{\rho}}-u\partial_{t}[\tilde{\varphi}]_{h}\d z & =\int_{B_{\rho}}\frac{1}{h}\int_{-\rho}^{0}-u(x,t)(\varphi(x,t+h)-\varphi(x,t))\d t\d x\\
 & =\int_{B_{\rho}}\left[\frac{1}{h}\int_{-\rho}^{0}-u(x,t-h)\varphi(x,t)\d t+\int_{-\rho}^{0}u(x,t)\varphi(x,t)\,dt\right]\, dx\\
 & =\int_{B_{\rho}}\int_{-\rho}^{0}\frac{1}{h}(u(x,t)-u(x,t-h))\varphi(x,t)\, dt\, dx\\
 & =\int_{Q_{\rho}}\partial_{t}[u]_{h}\varphi\d z.
\end{align*}
Hence, we arrive at the following.
\[
\int_{Q_\rho}\partial_{t}[u]_{h}\varphi+\left|Du\right|^{p-2}Du\cdot D[\tilde{\varphi}]_{h}+a(z)\left|Du\right|^{q-2}Du\cdot D[\tilde{\varphi}]_{h}-f(z, Du)[\tilde{\varphi}]_h\d z\leq0.
\]
The analogous inequality holds for the supersolution $v$.

\textbf{Step 2:} Let $v_{\ell}:=v+\ell$ for $\ell>0$. We consider the function
\[
\varphi(x,t)=m(u(x,t)-v_{\ell}(x,t))\beta^{\tau,\varepsilon}(t),
\]
where
\[
m(s) := \max(s,0)\quad \text{and}\quad\beta^{\tau,\varepsilon}(t):=\begin{cases}
1, & t\leq\tau-\varepsilon,\\
1-(t-\tau+\varepsilon)/\varepsilon, & \tau-\varepsilon<t\leq\tau,\\
0, & t>\tau.
\end{cases}
\]
for small $\varepsilon>0$ and $\tau\in(-\rho, 0)$. Observe that $\varphi\in W^H(Q_\rho)$.
Moreover, $\varphi$ is compactly supported in $Q_\rho$ since
$u\leq v$ on $\partial_{\mathcal{P}}Q_\rho$ and $\beta$ vanishes
near zero. Thus, we have
\begin{equation}
\int_{Q_\rho}\partial_{t}[u-v_{\ell}]_{h}\varphi\leq-\int_{Q_\rho}A\cdot D[\tilde{\varphi}]_{h}+\int_{Q_{\rho}}|f(z, Du)-f(z, Dv)|[\tilde{\varphi}]_h\d z\label{eq:A ineq}
\end{equation}
where we denoted
\[
A:=(\left|Du\right|^{p-2}Du-\left|Dv\right|^{p-2}Dv)+a(z)(\left|Du\right|^{q-2}Du-\left|Dv\right|^{q-2}Dv).
\]
We set
\[
G(s):=\begin{cases}
s^2/2 &s>0,\\
0, & s\leq0,
\end{cases}
\]
and observe that $G^{\prime}=m$. Since $G$
is convex, at any point $a\in\mathbb{R}$, its graph has a supporting
hyperplane $s\mapsto G(a)+G^{\prime}(a)\cdot(s-a)$.
That is, for all $a,b\in\mathbb{R}$, we have
\[
m(a)\cdot(a-b)\geq G(a)-G(b).
\]
Denote $w:=u-v_{\ell}$ and recall that $\partial_{t}[w]_{h}(x,t)=\frac{1}{h}w(x,t)-w(x,t-h)$.
So that taking $a=w(x,t)$ and $b=w(x,t-h)$, we obtain
\begin{align*}
\partial_{t}[w]_{h}(x,t)\varphi(x,t) & =\frac{\beta^{\tau,\varepsilon}(t)}{h}m(w(x,t))(w(x,t)-w(x,t-h))\\
 & \geq\frac{\beta^{\tau,\varepsilon}(t)}{h}\big(G(w(x,t))-G(w(x,t-h))\big)\\
 & =\beta^{\tau,\varepsilon}(t)\partial_{t}[G(w)]_{h}(x,t).
\end{align*} 

Using the above estimate and integrating by parts, we obtain from
(\ref{eq:A ineq}) that 
\[
-\int_{Q_{\rho}}[G(u-v_{\ell})]_{h}\partial_{t}\beta^{\tau,\varepsilon}\d z\leq-\int_{Q_{\rho}}A\cdot D[\tilde{\varphi}]_{h}+\int_{Q_{\rho}}|f(z, Du)-f(z, Dv)|[\tilde{\varphi}]_h\,\d z.
\]
By Lemma \ref{lem:steklov lemma} and H\"{o}lder's inequality,
we can let $h\rightarrow0$, to obtain
\[
-\int_{Q_{\rho}}G(u-v_{\ell})\partial_{t}\beta^{\tau,\varepsilon}\d z\leq-\int_{Q_{\rho}}\beta^{\tau,\varepsilon}A\cdot D(m(u-v_{\ell}))+ \int_{Q_{\rho}}\beta^{\tau, \varepsilon}|f(z, Du)-f(z, Dv)|m(u-v_l)\d z.
\]
Letting $\varepsilon\rightarrow0$, this yields by continuity for
all $\tau\in(-\rho, 0)$ 
\begin{align*}
\int_{B_\rho\times\left\{ \tau\right\} }G(u-v_{\ell})\d x & \leq-\int_{B_\rho \times (-\rho, \tau)}A\cdot D(m(u-v_{\ell}))\d z\\
&+ \int_{B_{\rho}\times(-\rho, \tau)}|f(z, Du)-f(z, Dv)|m(u-v_l)\, dz
\end{align*}
Observing that $G(u-v_\ell) = \frac{1}{2}m^2(u-v_\ell)$, the above inequality can be written as
\begin{align}\label{eq:comparison main est 1<p<2}
\frac{1}{2}\int_{B_\rho \times \{\tau\}} m^2(u-v_\ell) \d x \leq &\int_{B_{\rho}\times(-\rho, \tau)}|f(z, Du)-f(z, Dv)|m(u-v_l)\, dz \nonumber\\
&- \int_{B_\rho \times (-\rho, \tau)} A \cdot D(m(u-v_\ell))\, dz:=I_1-I_2
\end{align}

\textbf{Step 3:} Our aim is to absorb some part of $\smash{I_{1}}$ into $\smash{I_{2}}$ so that we can conclude from Gr\"{o}nwall's inequality that $m(w):=m(u-v_{\ell})\equiv 0$ almost everywhere. Since $\smash{f}$ is locally Lipschitz continuous in $\eta$, there are constants
$\smash{M\geq\max\left\{2\left\Vert Du\right\Vert _{L^{\infty}(Q_{\rho})}, 1\right\}}$
and $\smash{L=L(M)}$ such that
\begin{equation}
\left|f(x, t, \xi)-f(x, t, \eta)\right|\leq L\left|\xi-\eta\right|\text{ when }\left|\xi\right|,\left|\eta\right|<M.\label{eq:comparison lipcnd}
\end{equation}
We denote $Q_{\tau}:=\left\{ (x,t)\in B_{\rho}\times (-\rho, \tau):m(w)>0\right\} $,
\[
A:=Q_{\tau}\cap\left\{ \left|Dv\right|<M\right\} \text{ and }B:=Q_{\tau}\cap\left\{ \left|Dv\right|\geq M\right\} .
\]
Observe that in $B,$ using the growth condition (\ref{eq:gcnd}),
choice of $M$ and the assumption that $\beta_1, \beta_2\geq 1$ we have
\begin{align}
\left|f(z, Du)\right|&\leq C_{f}\left(1+\left|Du\right|^{\beta_1}+a(z)|Du|^{\beta_2}\right)\nonumber\\
&\leq C_{f}\left(M+M^{\beta_1}+a(z)M^{\beta_2}\right)\leq2C_{f}\left(M^{\beta_1}+a(z)M^{\beta_2}\right)\label{eq:comparison B est}
\end{align}
and also
\begin{equation}
\left|f(z, Dv)\right|\leq 2C_{f}\left(\left|Dv\right|^{\beta_1}+a(z)|Dv|^{\beta_2}\right)\label{eq:comparison growthcnd}
\end{equation}
Now we will estimate $I_1$ in \eqref{eq:comparison main est 1<p<2}. It follows from (\ref{eq:comparison lipcnd}), (\ref{eq:comparison B est}),
(\ref{eq:comparison growthcnd}) and Young's inequality that
\begin{align}
I_{1}\leq & \int_{A}L\left|Du-Dv\right|m(w)\d z+\int_{B}\left(\left|f(z, Du)\right|+\left|f(z, Dv)\right|\right)m(w)\d z\nonumber \\
\leq & \int_{A}L\left|Du-Dv\right|m(w)\d z+\int_{B}4C_{f}\left(\left|Dv\right|^{\beta_1}+a(z)|Dv|^{\beta_2}\right)m(w)\d z\nonumber \\
\leq & \int_{A}\varepsilon\left|Du-Dv\right|^{2}+C(\varepsilon,L)m^2(w)\d z+\int_{B}\varepsilon\left|Dv\right|^{\frac{\beta_1 p}{\beta_1}}+C(\varepsilon,p,\beta_1,L,C_{f})m(w)^{\frac{p}{p-\beta_1}}\d z\nonumber \\
& +\int_{B}\varepsilon a(z) |Dv|^{q}+C(\varepsilon,q,\beta_2,L,C_{f})a(z)m(w)^{\frac{q}{q-\beta_2}}\d z \nonumber\\
\leq & \varepsilon\int_{A}\left|Du-Dv\right|^{2}\d z+\varepsilon\int_{B}\left|Dv\right|^{p}\d z+\varepsilon \int_{B}a(z)|Dv|^q\d z\nonumber\\
&+ C(\varepsilon,p, q,\beta_1, \beta_2, L,C_{f},\left\Vert w\right\Vert _{L^{\infty}(Q_{\rho})}, \left\Vert a\right\Vert_{L^{\infty}(Q_{\rho})})\int_{Q_{\tau}}m^2(w)\d z,\label{eq:comparison (1<p<2) I_1 est}
\end{align}
where in the last step we used that $\frac{p}{p-\beta_1}>2$ and $\frac{q}{q-\beta_2}>2$ to estimate
\begin{align*}
\int_{B}m(w)^{p/(p-\beta_1)}\d z&=\int_{B}m(w)^{p/\left(p-\beta_1\right)-2}m(w)^{2}\d z\\
&\leq\left\Vert w\right\Vert _{L^{\infty}(Q_{\rho})}^{p/(p-\beta_1)-2}\int_{B_{\rho}\times (-\rho, \tau)}m^2(w)\d z,
\end{align*}
and 
\begin{align*}
\int_{B}a(z)m(w)^{\frac{q}{(q-\beta_2)}}\d z&=\int_{B}a(z)m(w)^{\frac{q}{\left(q-\beta_2\right)}-2}m^2(w)\d z\\
&\leq\left\Vert w\right\Vert _{L^{\infty}(Q_{\rho})}^{\frac{q}{(q-\beta_2)}-2}\left\Vert a\right\Vert_{L^{\infty}(Q_{\rho})}\int_{B_{\rho}\times(-\rho, \tau)}m^2(w)\d z.
\end{align*}
To estimate $I_2$ of \eqref{eq:comparison main est 1<p<2}, we use the following vector inequalities from \cite[p98]{lindqvist_plaplace}.
\begin{align}
\left(\left|a\right|^{p-2}a-\left|b\right|^{p-2}b\right)\cdot\left(a-b\right)\geq\left(p-1\right)\left|a-b\right|^{2}\left(1+\left|a\right|^{2}+\left|b\right|^{2}\right)^{\frac{p-2}{2}},\label{eq:algebraic ineq 1<p<2}
\end{align}
and
\begin{align}
\left(\left|a\right|^{q-2}a-\left|b\right|^{q-2}b\right)\cdot\left(a-b\right)\geq\left(q-1\right)\left|a-b\right|^{2}\left(1+\left|a\right|^{2}+\left|b\right|^{2}\right)^{\frac{q-2}{2}}.\label{eq:algebraic ineq 1<q<2}
\end{align}
Using these inequalities, we obtain the following lower bound of $I_2.$
\begin{align}
I_{2}= & \int_{B_{\rho}\times (-\rho, \tau)}\left(\left|Du\right|^{p-2}Du-\left|Dv\right|^{p-2}Dv\right)\cdot Dm(w)\d z\nonumber\\
&+\int_{B_{\rho}\times (-\rho, \tau)}a(z)\left(\left|Du\right|^{q-2}Du-\left|Dv\right|^{q-2}Dv\right)\cdot Dm(w)\d z \nonumber \\
\geq & (p-1)\int_{Q_{\tau}}\frac{\left|Du-Dv\right|^{2}}{\left(1+\left|Du\right|^{2}+\left|Dv\right|^{2}\right)^{\frac{2-p}{2}}}\d z+ (q-1)\int_{Q_{\tau}}\frac{a(z)\left|Du-Dv\right|^{2}}{\left(1+\left|Du\right|^{2}+\left|Dv\right|^{2}\right)^{\frac{2-q}{2}}}\d z\nonumber \\
\geq & (p-1)\int_{A}\frac{\left|Du-Dv\right|^{2}}{\left(1+M^{2}+M^{2}\right)^{\frac{2-p}{2}}}\d z+(p-1)\int_{B}\frac{\left(\left|Dv\right|-\left|Du\right|\right)^{2}}{\left(3\left|Dv\right|^{2}\right)^{\frac{2-p}{2}}}\d z\nonumber \\
+ & (q-1)\int_{A}\frac{a(z)\left|Du-Dv\right|^{2}}{\left(1+M^{2}+M^{2}\right)^{\frac{2-q}{2}}}\d z+(q-1)\int_{B}\frac{a(z)\left(\left|Dv\right|-\left|Du\right|\right)^{2}}{\left(3\left|Dv\right|^{2}\right)^{\frac{2-q}{2}}}\d z \nonumber\\
\geq & C(p,M)\int_{A}\left|Du-Dv\right|^{2}\d z+\left(p-1\right)\int_{B}\frac{\left(\left|Dv\right|-\frac{1}{2}M\right)^{2}}{\left(3\left|Dv\right|^{2}\right)^{\frac{2-p}{2}}}\d z\nonumber \\
+ & C(q,M)\int_{A}a(z)\left|Du-Dv\right|^{2}\d z+\left(q-1\right)\int_{B}\frac{a(z)\left(\left|Dv\right|-\frac{1}{2}M\right)^{2}}{\left(3\left|Dv\right|^{2}\right)^{\frac{2-q}{2}}}\d z\nonumber \\
\geq &\min\{ C(p,M), C(q, M)\}\int_{A}\left|Du-Dv\right|^{2}\d z+\left(p-1\right)\int_{B}\frac{\left(\frac{1}{2}\left|Dv\right|\right)^{2}}{\left(3\left|Dv\right|^{2}\right)^{\frac{2-p}{2}}}\d z\nonumber\\
&+\left(q-1\right)\int_{B}\frac{a(z)\left(\frac{1}{2}\left|Dv\right|\right)^{2}}{\left(3\left|Dv\right|^{2}\right)^{\frac{2-q}{2}}}\d z\label{eq:comparison (1<p<2) I_2 est}
\end{align}
where $C(p, M), C(q,M), C(p), C(q)>0$. Combining the estimates (\ref{eq:comparison (1<p<2) I_1 est})
and (\ref{eq:comparison (1<p<2) I_2 est}) we arrive at
\begin{align*}
I_{1}-I_{2}\leq & \left(\varepsilon-C(p,M)\right)\int_{A}\left|Du-Dv\right|^{2}\d z+\left(\varepsilon-C(p)\right)\int_{B}\left|Dv\right|^{p}\d z\\
&+(\varepsilon-C(q))\int_{B}a(z)|Dv|^{q}\d z+C_{0}\int_{B_{\rho}\times(-\rho, \tau)}m^2(w)\d z,
\end{align*}
where $C_{0}=C(\varepsilon,p,q, \beta_1, \beta_2, L,C_{f},\left\Vert w\right\Vert _{L^{\infty}})$.
Recalling (\ref{eq:comparison main est 1<p<2}) and taking small enough
$\epsilon$ yields
\[
\int_{B_{\rho}\times\{\tau\}}m^2(w)\d x\leq 2C_{0}\int_{B_{\rho}\times(-\rho, \tau)}m^2(w)\d z.
\]
Since this holds for a.e.\ $\tau\in(-\rho,0)$, Gr\"{o}nwall's inequality implies
that $m(w)\equiv0$ a.e.\ in $Q_{\rho}$. Finally, letting $l\rightarrow0$
yields that $u-v\leq0$ a.e.\ in $Q_{\rho}$.
\end{proof}
\begin{lem}\label{lem:comparison lemma q<2}
Let $1<p<2\leq q$ and $\rho >0$. Suppose that $a:Q_\rho \rightarrow [0,\infty)$ is Lipschitz in space, continuous in time, satisfies \eqref{assumption increasing} and $f(z, \eta)$ is locally Lipschitz in $\eta$.  Let $u,v\in W^{H}(Q_\rho)\cap C(\overline{Q_\rho})$ be weak sub- and supersolutions, respectively, to
\[
\partial_{t}u-\div(\left|Du\right|^{p-2}Du+a(z)\left|Du\right|^{q-2}Du)-f(z, Du)\leq -\delta\quad\text{in }\,\,\, Q_\rho.
\]
Also assume that $Du\in L^{\infty}(Q_{\rho}).$ Then, if $u\leq v$ on $\partial_{\mathcal{P}}Q_\rho$, it follows that $u\leq v$ in $Q_{\rho}.$
\end{lem}
\begin{proof}
Keeping the notations fixed and repeating the second step of the previous lemma, we get
\begin{align}
\frac{1}{2}\int_{B_\rho \times \{\tau\}} m^2(u-v_\ell) \d x \leq &\int_{B_{\rho}\times(-\rho, \tau)}|f(z, Du)-f(z, Dv)|m(u-v_l)\, dz \nonumber\\
&- \int_{B_\rho \times (-\rho, \tau)} A \cdot D(m(u-v_\ell))\, dz
-\int_{B_{\rho}\times(-\rho, \tau)}\delta m(u-v_{\ell})\d z\nonumber \\
=: & I_{1}-I_{2}-\int_{B_{\rho}\times (-\rho, \tau)}\delta m(u-v_{\ell})\d z.\label{eq:comparison main est p>2}
\end{align}
Moreover, recalling $w=u-v_{\ell}$ and following the estimate of $I_1$ in the above Lemma \ref{lem: first comparison}, we get an inequality
\begin{align*}
    I_1 \leq &\int_{A}\varepsilon|Du-Dv|^2\,dz+ \int_{B}\varepsilon \left(|Dv|^p+a(z)|Dv|^q\right)\,dz\\
    &+ C(\varepsilon, p, q, \beta_1, \beta_2, L, C_f, \left\Vert a\right\Vert_{L^{\infty}(Q_{\rho})}) \int_{B_{\rho}\times(-\rho, \tau)} m(w)^2+m(w)^{\frac{p}{p-\beta_1}}+m(w)^{\frac{q}{q-\beta_2}}-\delta m(w)\, dz.
\end{align*}
To estimate $I_2,$ we use the following algebraic inequalities (see \cite[p95]{lindqvist_plaplace})
\begin{align}
 &\left(\left|a\right|^{p-2}a-\left|b\right|^{p-2}b\right)\cdot\left(a-b\right)\geq\left(p-1\right)\left|a-b\right|^{2}\left(1+\left|a\right|^{2}+\left|b\right|^{2}\right)^{\frac{p-2}{2}},\label{eq:second algebraic ineq 1<p<2}\\ 
&\left(\left|a\right|^{q-2}a-\left|b\right|^{q-2}b\right)\cdot\left(a-b\right)\geq2^{2-q}\left|a-b\right|^{q}.\label{eq:first algebraic ineq p>2} 
\end{align}
Indeed, using these inequalities, we estimate
\begin{align}
I_{2}= & \int_{B_{\rho}\times(-\rho, \tau)}\left(\left|Du\right|^{p-2}Du-\left|Dv\right|^{p-2}Dv\right)\cdot Dm(w)\d z \nonumber\\
&+\int_{B_{\rho}\times (-\rho, \tau)}a(z)\left(\left|Du\right|^{q-2}Du-\left|Dv\right|^{q-2}Dv\right)\cdot Dm(w)\d z \nonumber \\
 & \geq (p-1)\int_{Q_{\tau}}\frac{\left|Du-Dv\right|^{2}}{\left(1+\left|Du\right|^{2}+\left|Dv\right|^{2}\right)^{\frac{2-p}{2}}}\d z+ 2^{2-q}\int_{Q_{\tau}}a(z)\left|Du-Dv\right|^{q}\d z\nonumber \\
 & \geq C(p,M)\int_{A}\left|Du-Dv\right|^{2}\d z+C(p)\int_{B}|Dv|^p\d z+ 2^{2-q}\int_{B}a(z)|Du-Dv|^q\d z\nonumber \\
 & \geq C(p,M)\int_{A}\left|Du-Dv\right|^{2}\d z+C(p)\int_{B}|Dv|^p\d z+ C(q)\int_{B}a(z)|Dv|^q\d z, \label{Eq3.13}
\end{align}
where in the last step we used 
\begin{align*}
\left|Du-Dv\right|^{q}\geq\left(\left|Dv\right|-\left|Du\right|\right)^{q}\geq\left(\left|Dv\right|-\frac{1}{2}M\right)^{q}\geq C(q)\left|Dv\right|^{q}.    
\end{align*}
Therefore, combining the estimates \eqref{eq:comparison (1<p<2) I_1 est} and \eqref{Eq3.13}, we get
\begin{align*}
  I_{1}-I_{2}-\int_{B_{\rho}\times(-\rho, \tau)}\delta m(w) \, dz\leq & \left(\varepsilon-C(p,M)\right)\int_{A}\left|Du-Dv\right|^{2}\d z+\left(\varepsilon-C(p)\right)\int_{B}\left|Dv\right|^{p}\d z\\
  &+(\varepsilon-C(q))\int_{B}a(z)|Dv|^{q}\d z\\
  &+C_0\int_{B_{\rho}\times(-\rho, \tau)}m^2(w)+m(w)^{\frac{p}{p-\beta_1}}+m(w)^{\frac{q}{q-\beta_2}}-\delta m(w)\d z,
\end{align*}
where $C_0$ depends on $\varepsilon, p, q, L, \beta_1, \beta_2, C_f, \left\Vert a\right\Vert_{L^{\infty}(Q_{\rho})}.$ Thus, for small $\varepsilon, $ we have
\begin{align}\label{eq: main estimate: second comparison}
    \frac{1}{2}\int_{B_{\rho}\times \{\tau\}}m^2(w)\, dx\leq C_0 \int_{B_{\rho}\times(-\rho, \tau)}m^2(w)+m(w)^{\frac{p}{p-\beta_1}}+m(w)^{\frac{q}{q-\beta_2}}-\delta m(w)\d z,
\end{align}

Observe that since $w$ is bounded and $\frac{p}{p-\beta_1}, \frac{q}{q-\beta_2}>1$,
the integrand at the right-hand side is bounded by some constant times
$m^2(w)$. To argue this rigorously, we write down the following algebraic
fact.

If $a_{0},\delta,\gamma>0$ and $\alpha>1$, then there exists $C(\alpha,\gamma,\delta,a_{0})>0$
such that
\[
\gamma a^{\alpha}\leq\delta a+C(\alpha,\gamma,\delta,a_{0})a^{2}\text{ for all }a\in[0,a_{0}).
\]
To see this, let first $\alpha<2$. Then by Young's inequality
\begin{align*}
\gamma a^{\alpha}=\gamma a\cdot a^{\alpha-1} & \leq \frac{\delta}{1+a_{0}^{\frac{2}{3-\alpha}}}a^{\frac{2}{3-\alpha}}+C(\alpha,\gamma,\delta,a_{0})a^{\left(\alpha-1\right)\cdot\frac{2}{\alpha-1}}\\
& \leq \delta a+C(\alpha,\gamma,\delta,a_{0})a^{2}.
\end{align*}
If $\alpha\geq2,$ then
\[
\gamma a^{\alpha}=\gamma a^{\alpha-2}\cdot a^{2}\leq\gamma a_{0}^{\alpha-2}a^{2}.
\]

Applying the algebraic fact on \eqref{eq: main estimate: second comparison}
we get
\[
\int_{B_{\rho}\times\{\tau\}}m^2(w)\d x\leq C(p,q, \beta_1, \beta_2, L,C_{f},\delta,\left\Vert a\right\Vert _{L^{\infty}})\int_{B_{\rho}\times (-\rho, \tau)}m^2(w)\d z.
\]
The conclusion now follows from Gr\"{o}nwall's inequality and letting
$l\rightarrow0$.
\end{proof}
\begin{lem}
\label{lem:comparison lemma p>2}Let $2\leq p$ and $\rho >0$. Suppose that $a:Q_\rho \rightarrow [0,\infty)$ is Lipschitz in space, continuous in time, satisfies \eqref{assumption increasing} and $f(z, \eta)$ is locally Lipschitz in $\eta$.  Let $u,v\in W^{H}(Q_\rho)\cap C(\overline{Q_\rho})$ be weak sub- and supersolutions, respectively, to
\[
\partial_{t}u-\div(\left|Du\right|^{p-2}Du+a(z)\left|Du\right|^{q-2}Du)-f(z, Du)\leq -\delta\quad\text{in }\,\,\, Q_\rho.
\]
Also assume that $Du\in L^{\infty}(Q_{\rho}).$ Then, if $u\leq v$ on $\partial_{\mathcal{P}}Q_\rho$, it follows that $u\leq v$ in $Q_{\rho}.$
\end{lem}
\begin{proof}
Although the proof is similar to the previous ones, we provide all the details for completeness. Repeating the Step 2 of Lemma \ref{lem: first comparison}, we have
\begin{align}
\frac{1}{2}\int_{B_\rho \times \{\tau\}} m^2(u-v_\ell) \d x \leq &\int_{B_{\rho}\times(-\rho, \tau)}|f(z, Du)-f(z, Dv)|m(u-v_l)\, dz \nonumber\\
&- \int_{B_\rho \times (-\rho, \tau)} A \cdot D(m(u-v_\ell))\, dz
-\int_{B_{\rho}\times(-\rho, \tau)}\delta m(u-v_{\ell})\d z\nonumber \\
=: & I_{1}-I_{2}-\int_{B_{\rho}\times (-\rho, \tau)}\delta m(u-v_{\ell})\d z.\label{eq:comparison main est p>2}
\end{align}
Then by (\ref{eq:comparison lipcnd}), (\ref{eq:comparison B est}),
(\ref{eq:comparison growthcnd}) and Young's inequality
\begin{align}
I_{1}\leq & \int_{A}L\left|Du-Dv\right|m(w)\d z+\int_{B}\left(\left|f(x, Du)\right|+\left|f(z, Dv)\right|\right)m(w)\d z\nonumber \\
\leq & \int_{A}\varepsilon\left|Du-Dv\right|^{p}+C(\varepsilon,L)m(w)^{\frac{p}{p-1}}\d z+\int_{B}4C_{f}\left(\left|Dv\right|^{\beta_1}+a(z)|Dv|^{\beta_2}\right)m(w)\d z\nonumber \\
\leq & \varepsilon\int_{A}\left|Du-Dv\right|^{p}\d z+\varepsilon\int_{B}\left|Dv\right|^{p}\d z+\varepsilon \int_{B}a(z)|Dv|^q\d z \nonumber \\
&+C(\varepsilon,p,q, \beta_1, \beta_2, L,C_{f})\int_{B_{\rho}\times(-\rho, \tau)}m(w)^{\frac{p}{p-1}}+m(w)^{\frac{p}{p-\beta_1}}+ m(w)^{\frac{q}{q-\beta_2}}\d z.\label{eq:comparison (p>2) I_1 est}
\end{align}
Using the vector inequality
\begin{equation}
\left(\left|a\right|^{r-2}a-\left|b\right|^{r-2}b\right)\cdot\left(a-b\right)\geq2^{2-r}\left|a-b\right|^{r},\label{eq:algebraic ineq p>2}
\end{equation}
which holds when $r\geq2$ \cite[p95]{lindqvist_plaplace}, we get
\begin{align*}
I_{2}\geq  C(p)\int_{A}\left|Du-Dv\right|^{p}\d z+C(p)\int_{B}\left|Du-Dv\right|^{p}\d z+ c(q)\int_{B}a(z)|Du-Dv|^q\d z.
\end{align*}
Furthermore, since in $B$ it holds
\[
\left|Du-Dv\right|^{r}\geq\left(\left|Dv\right|-\left|Du\right|\right)^{r}\geq\left(\left|Dv\right|-\frac{1}{2}M\right)^{r}\geq C(p)\left|Dv\right|^{r},
\]
we arrive at
\begin{equation}
I_{2}\geq C(p)\int_{A}\left|Du-Dv\right|^{p}\d z+C(p)\int_{B}\left|Dv\right|^{p}\d z+ c(q)\int_{B}a(z)|Dv|^q\d z.\label{eq:comparison (p>2) I_2 est}
\end{equation}
Combining (\ref{eq:comparison (p>2) I_1 est}) and (\ref{eq:comparison (p>2) I_2 est})
with (\ref{eq:comparison main est p>2}) we get
\begin{align*}
&\frac{1}{2}\int_{B_{\rho}\times \{\tau\}}m(w)^{2}\d x\\
&\leq \left(\varepsilon-C(p)\right)\left(\int_{A}\left|Du-Dv\right|^{p}\d z+\int_{B}\left|Dv\right|^{p}\d z\right)+ \left(\varepsilon-C(q)\right)\int_{B}a(z)|Dv|^q\d z\\
 & +\int_{B_{\rho}\times(-\rho, \tau)}C(\epsilon, p, q,\beta_1,\beta_2, L,C_{f})\left(m(w)^{\frac{p}{p-1}}+m(w)^{\frac{p}{p-\beta_1}}+m(w)^{\frac{q}{q-\beta_2}}\right)-\delta m(w)\d z.
\end{align*}
By taking small enough $\varepsilon$, the above becomes
\begin{align*}
&\int_{B_{\rho}\times\{\tau\}}m(w)^{2}(x,s)\d x\\
&\leq\int_{B_{\rho}\times(-\rho, \tau)}C(p, q, \beta_1, \beta_2, L,C_{f})\left(m(w)^{\frac{p}{p-1}}+m(w)^{\frac{p}{p-\beta_1}}+m(w)^{\frac{q}{q-\beta_2}}\right)-\delta m(w)\d z.\label{eq:comparison (p>2) some ineq}
\end{align*}
Again, since $\frac{p}{p-1}, \frac{p}{p-\beta_1}, \frac{q}{q-\beta_2}>1,$ we can use the algebraic fact from the above Lemma \ref{lem:comparison lemma q<2} to conclude
\begin{align*}
    \int_{B_{\rho}\times\{\tau\}} m(w)^2\, dx \leq C (\varepsilon,p,q, \beta_1, \beta_2, L,C_{f}, \left\Vert a\right\Vert_{L^{\infty}(Q_{\rho})})\int_{B_{\rho}\times(-\rho, \tau)}m(w)^2 \, dz.
\end{align*}
Passing to the limit $l\to 0$ and applying Gr\"{o}nwall's inequality we complete the proof.
\end{proof}
Next we use the previous comparison results to prove the comparison
principle for general continuous $f$.
\begin{lem}\label{thm:comparison principle}
Let $1<p\leq q\leq p+1,$  and $\rho >0$. Suppose that $a:Q_\rho \rightarrow [0,\infty)$ is Lipschitz in space, continuous in time, satisfies \eqref{assumption increasing}.  Let $u,v\in W^{H}(Q_\rho)\cap C(\overline{Q_\rho})$ be weak sub- and supersolutions, respectively, to
\[
\partial_{t}u-\div(\left|Du\right|^{p-2}Du+a(z)\left|Du\right|^{q-2}Du)=f(z, Du)\quad\text{in }\,\,\, Q_\rho.
\]
Also assume that $Du\in L^{\infty}(Q_{\rho}).$ Then, if $u\leq v$ on $\partial_{\mathcal{P}}Q_\rho$, it follows that $u\leq v$ in $Q_{\rho}.$
\end{lem}
\begin{proof}
For any $\varepsilon>0$, define
\[
u_{\varepsilon}:=u-\frac{\varepsilon}{T-t/2}.
\]
Then for any non-negative test function $\varphi\in C_{0}^{\infty}(Q_{\rho})$
we have by integration by parts
\begin{align*}
\int_{Q_{\rho}}-u_{\varepsilon}\partial_{t}\varphi\d z= & \int_{Q_{\rho}}-u\partial_{t}\varphi+\frac{\varepsilon}{T-t/2}\partial_{t}\varphi\d z\\
= & \int_{Q_{\rho}}-u\partial_{t}\varphi-\varphi\frac{\varepsilon}{2\left(T-t/2\right)^{2}}\d z\\
\leq & \int_{Q_{\rho}}-u\partial_{t}\varphi-\varphi\frac{\varepsilon}{2T^{2}}\d z.
\end{align*}
Since $f$ is continuous, there exists sequence of locally Lipschitz continuous
functions $\{\smash{f_{\varepsilon}}\}$ such that $\smash{\left\Vert f-f_{\varepsilon}\right\Vert _{L^{\infty}(\mathbb{R}^{N})}\leq\frac{\varepsilon}{4T^2}}$
(see e.g.\ \cite{miculescu00}). Then, since $u$ is a weak subsolution,
we have
\begin{align*}
&\int_{Q_{\rho}} -u_{\varepsilon}\partial_{t}\varphi+\left(\left|Du_{\varepsilon}\right|^{p-2}Du_{\varepsilon}+a(z)|Du_\varepsilon|^{q-2}Du_{\varepsilon}\right)\cdot D\varphi-\varphi f_{\varepsilon}(z, Du_{\varepsilon})\d z\\
&\leq \int_{Q_{\rho}}-u\partial_{t}\varphi+\left(\left|Du\right|^{p-2}Du+a(z)|Du|^{q-2}Du\right)\cdot D\varphi-\varphi f(z, Du)+\varphi\left\Vert f-f_{\varepsilon}\right\Vert _{L^{\infty}(\mathbb{\mathbb{R}^{N}})}-\varphi\frac{\varepsilon}{2T^{2}}\d z\\
&\leq\int_{Q_{\rho}}-\frac{\varepsilon}{4T^{2}}\varphi\d z.
\end{align*}
Hence $u_{\varepsilon}$ is a weak subsolution to 
\[
\partial_{t}u_{\varepsilon}-\div\left(|Du_\varepsilon|^{p-2}Du+a(z)|Du_\varepsilon|^{q-2}Du_\varepsilon\right)-f_{\varepsilon}(z, Du_{\varepsilon})\leq-\frac{\varepsilon}{4T^{2}}\quad\text{in }Q_{\rho}.
\]
Similarly, since $v$ is a weak supersolution, we define
\[
v_{\varepsilon}:=v+\frac{\varepsilon}{T-t/2}
\]
 and deduce that $v_{\varepsilon}$ is a weak supersolution to
\[
\partial_{t}v_{\varepsilon}-\div\left(|Du_{\varepsilon}|^{p-2}Du_{\varepsilon}+a(z)|Dv_{\varepsilon}|^{q-2}Dv_{\varepsilon}\right)-f_{\varepsilon}(z, Dv_{\varepsilon})\geq\frac{\varepsilon}{4T^2}\quad\text{in }Q_{\rho}.
\]
Furthermore, we have $u_{\varepsilon}\leq v_{\varepsilon}$ on $\partial_{\mathcal{P}}Q_{\rho}.$ Hence, by using the comparison Lemmas \ref{lem: first comparison}, \ref{lem:comparison lemma q<2}
and \ref{lem:comparison lemma p>2}, we conclude that $u_{\varepsilon}\leq v_{\varepsilon}$
a.e.\ in $Q_{\rho}$. Hence, we have
\[
u\leq v+\frac{2\varepsilon}{T-t/2}\text{\quad a.e.\ in }Q_{\rho}.
\]
Letting $\varepsilon\rightarrow 0$ finishes the proof.
\end{proof}
Now, using the comparison principle, we will show that weak solutions are viscosity solutions. 

\begin{thm}
\label{thm:weak is visc}Let $1<p\leq q\leq p+1$. Suppose that $a:\Xi \rightarrow [0,\infty)$ is locally Lipschitz in space, continuous in time. Suppose that for any $(x,t) \in \Xi$, there is $r > 0$ such that $a$ satisfies \eqref{assumption increasing} in $Q_r(x,t)$ with some constant $C_a=C_a(x,t,r)$. Let $u\in C(\Xi)$ be a local weak supersolution to (\ref{eq:p-para f})
in $\Xi$. Then $u$ is a viscosity supersolution to (\ref{eq:p-para f}) in $\Xi$.
\end{thm}
\begin{proof}
On the contrary, let there exists a $\smash{\varphi\in C^{2}(\Xi)}$
touching $u$ from below at $\smash{(x_{0},t_{0})\in\Xi}$, $\smash{D\varphi(x,t)\not=0}$
for $\smash{x\not=x_{0}}$ and
\begin{align}
\limsup_{\substack{\substack{(x,t)\rightarrow(x_{0},t_{0})\\
x\not=x_{0}}}}
&\Big(\partial_{t}\varphi(x,t)-\div\left(|D\varphi(x, t)|^{p-2}\varphi(x,t)+a(x, t)|D\varphi(x, t)|^{q-2}D\varphi(x, t)\right)\nonumber\\
&-f(x, t, D\varphi(x,t))\Big)<0.\label{eq:weak is visc lipschitz 1<p<2}
\end{align}
It follows from above that there are $r>0$ and $\varepsilon>0$ such that
\begin{equation}
\partial_{t}\varphi-\div\left(|D\varphi|^{p-2}D\varphi+a(z)|D\varphi|^{q-2}D\varphi\right)-f(z, D\varphi)<-\varepsilon/2\quad\text{in }\mathcal{Q}_r\setminus\left\{ x=x_{0}\right\}\cup \Gamma,\label{eq:weak is visc lipschitz 1<p<2 someineq}
\end{equation}
where  $\mathcal{Q}_r:=B_{r}(x_{0})\times(t_{0}-r,t_0).$
Indeed, otherwise there would be a sequence $(x_{n},t_{n})\rightarrow(x_{0},t_{0})$
such that $x_{n}\not=x_{0}$ and 
\begin{align*}
&\partial_{t}\phi(x_n, t_n)-\div\left(|D\phi(x_n, t_n)|^{p-2}D\phi(x_n, t_n)+a(x_n , t_n)|D\phi(x_n, t_n)|^{q-2}D\phi(x_n, t_n)\right)\\
&-f(x_n, t_n, D\phi(x_n , t_n))>-\frac{1}{n},
\end{align*}
which contradicts (\ref{eq:weak is visc lipschitz 1<p<2}). Further, taking smaller $r>0$ if necessary, we may assume that $a$ satisfies \eqref{assumption increasing} in $\mathcal Q_r$ with some constant $C_a \geq 0$.
Let $a_{j}$ be the standard mollification of $a$ in space direction.
By taking large enough $j\in\mathbb{N}$, we have that $\left\Vert a-a_{j}\right\Vert _{L^{\infty}(\mathcal{Q}_{r})}<\varepsilon/4$. 
Since $|\Gamma|=0,$ using
Gauss divergence theorem, we obtain for any non-negative test function $\psi\in C_{0}^{\infty}(\mathcal{Q}_{r})$
that 
\begin{align}\label{div thm}
&\text{LHS}:=\int_{\mathcal{Q}_{r}\setminus\{x-x_0| \leq \rho\}} -\varphi\partial_{t}\psi+\left(\left|D\varphi\right|^{p-2}D\varphi+a_j(z)|D\varphi|^{q-2}D\varphi\right)\cdot D\psi-\psi f(z, D\varphi)\d z\nonumber\\
&=\Bigg(\int_{\mathcal{Q}_{r}\setminus\left\{ \left|x-x_{0}\right|\leq\rho\right\} }\psi\partial_{t}\varphi-\psi\div\left(\left|D\varphi\right|^{p-2}D\varphi+a_{j}(z)\left|D\varphi\right|^{q-2}D\varphi\right)-\psi f(z, D\varphi)\d z\nonumber\\
 & \ \ \ \ \ \ +\int_{t_{0}-r}^{t_{0}}\int_{\left\{ \left|x-x_{0}\right|=\rho\right\} }\psi\left(\left|D\varphi\right|^{p-2}D\varphi+a_{j}(z)\left|D\varphi\right|^{q-2}D\varphi\right)\cdot\frac{(x-x_{0})}{\rho}\d S\d t\Bigg):=\text{RHS}
\end{align}
Note that, since $D\varphi(x, t)\neq 0$ for $x\neq x_0,$ we use dominated convergence theorem in the first term of RHS above as $j \to \infty,$ and obtain
\begin{align*}
 &\int_{\mathcal{Q}_{r}\setminus\left\{ \left|x-x_{0}\right|\leq\rho\right\} }\psi\left(\partial_{t}\varphi-\div\left(\left|D\varphi\right|^{p-2}D\varphi+a_{j}(z)\left|D\varphi\right|^{q-2}D\varphi\right)- f(z, D\varphi)\right)\d z\\
&=\int_{\mathcal{Q}_r\setminus \{|x-x_0|\leq \rho\}}\psi\Big(\partial_t \varphi-\div\left(|D\varphi|^{p-2}D\varphi\right)-a_j(z)\div\left(|D\varphi|^{q-2}D\varphi\right)\\
&-|D\varphi|^{q-2}D\varphi\cdot Da_j(z)-f(z, D\varphi)\Big)\, dz\\
\overset{j \to \infty}{\to}\\
&\int_{\mathcal{Q}_r\setminus \{|x-x_0|\leq \rho\}}\psi\Big(\partial_t \varphi-\div\left(|D\varphi|^{p-2}D\varphi\right)-a(z)\div\left(|D\varphi|^{q-2}D\varphi\right)\\
&-|D\varphi|^{q-2}D\varphi\cdot Da(z)-f(z, D\varphi)\Big)\, dz\\
&=\int_{\mathcal{Q}_{r}\setminus\left\{ \left|x-x_{0}\right|\leq\rho\right\} }\psi\left(\partial_{t}\varphi-\div\left(\left|D\varphi\right|^{p-2}D\varphi+a(z)\left|D\varphi\right|^{q-2}D\varphi\right)- f(z, D\varphi)\right)\d z
\end{align*}
Now using \eqref{eq:weak is visc lipschitz 1<p<2} and $\rho \to 0$ in both sides of \eqref{div thm}, we obtain
\begin{align*}
\text{LHS}:=\int_{\mathcal{Q}_{r}} -\varphi\partial_{t}\psi+\left(\left|D\varphi\right|^{p-2}D\varphi+a_j(z)|D\varphi|^{q-2}D\varphi\right)\cdot D\psi-\psi f(z, D\varphi)\d z \leq \int_{\mathcal{Q}_r}-\psi\varepsilon/2\, dz    
\end{align*}
Let $l:=\min_{\partial_{p}\mathcal{Q}_{r}}\left(u-\varphi\right)>0$ and set $\widetilde{\varphi}:=\varphi+l$.
Then tending $j\to \infty$ in the above inequality, we conclude that $\widetilde{\varphi}$ is a weak subsolution
to
\[
\partial_{t}\widetilde{\varphi}-\div\left(\left|D\widetilde{\varphi}\right|^{p-2}D\widetilde{\varphi}+a(z)\left|D\widetilde{\varphi}\right|^{q-2}D\widetilde{\varphi}\right)-f(z, D\widetilde{\varphi})\leq-\varepsilon/2\quad\text{in }\mathcal{Q}_{r}
\]
and on $\partial_{p}\mathcal{Q}_{r}$ it holds $\tilde{\varphi}=\varphi+l\leq\varphi+u-\varphi=u$.
Since $a$ satisfies \eqref{assumption increasing} in $\mathcal Q_r$, it follows from Theorem \ref{thm:comparison principle} that $\widetilde{\varphi}\leq u$
in $\mathcal{Q}_{r}$. But this is not possible since $\widetilde{\varphi}(x_{0},t_{0})>u(x_{0},t_{0})$ and $u$ is continuous.

\begin{old}\textbf{(Case
$p\geq2$)} Assume on contrary that there is a $\varphi\in C^{2}(\overline{\Omega}\times[0,T])$
testing $u$ from below at $(x_{0},t_{0})\in\Omega_{T}$ and
\[
\partial_{t}\phi(x_{0},t_{0})-\Delta_{p}\phi(x_{0},t_{0})-f(D\phi(x_{0},t_{0}))<-\delta
\]
for some $\delta>0$. By continuity we have $r>0$ such that 
\[
\partial_{t}\phi-\Delta_{p}\phi-f(D\phi)<-\delta\text{ in }Q_{r}(x_{0},t_{0}).
\]
Multiplying the above inequality by $\varphi\in C_{0}^{\infty}(Q_{r}(x_{0},t_{0}))$
and integrating by parts we arrive at
\[
\int_{\Omega_{T}}\phi\partial_{t}\varphi+\left|D\phi\right|^{p-2}D\phi\cdot D\varphi-\varphi f(D\phi)\d z\leq-\int_{\Omega_{T}}\delta\varphi\d z.
\]
Let $l=\min_{\partial_{p}Q_{r}(x_{0},t_{0})}u-\phi$ and set $\tilde{\phi}=\phi+l$.
Then $\tilde{\phi}$ is a weak subsolution to 
\[
\partial_{t}\tilde{\phi}-\Delta_{p}\tilde{\phi}-f(D\tilde{\phi})\text{ in }Q_{r}(x_{0},t_{0})
\]
and on $\partial_{p}Q_{r}(x_{0},t_{0})$ it holds that $\tilde{\phi}=\phi+l\leq\phi+u-\phi=u$.
Hence Lemma \ref{lem:comparison lemma p>2} implies that $\tilde{\phi}\leq u$
in $Q_{r}(x_{0},t_{0})$. But this is not possible since $\tilde{\phi}(x_{0},t_{0})>u(x_{0},t_{0})$.\end{old}
\end{proof}
\begin{rem}
\begin{enumerate}
\item [(1)] Note that in Lemmas \ref{lem: first comparison}, \ref{lem:comparison lemma q<2}, \ref{lem:comparison lemma p>2}, and \ref{thm:comparison principle}, we assumed that the subsolutions are locally Lipschitz continuous. This assumption does not diminish the novelty of our results, as in the final proof of Theorem \ref{thm:weak is visc}, these results are applied to a $C^2$-regular function.
\item[(2)]In our comparison principle, we assume the condition \eqref{assumption increasing} on the whole cylinder. A more general comparison result is expected via covering arguments, but we do not pursue that here and postpone it for future work.
\item[(3)] The continuity assumption on the sub (super) solution could also be replaced by boundedness and lower (upper) semicontinuity, provided that the assumption $u \leq v$ on the boundary is interpreted in a suitable way.
\end{enumerate}
\end{rem}
\section{Lipschitz regularity for class $\texttt{S}$}\label{sec:Lipschitz estimates}

In this section, we apply the Ishii-Lions method from the theory
of viscosity solutions to prove Lipschitz estimates for an auxiliary class of relaxed viscosity solutions, denoted by $\texttt{S}$. In Section \ref{sec: time holder} we use these results on weak solutions to establish Theorem \ref{Lip thm}.
To define the class \texttt{S}, we need the concept of relative semi-jets, given below.

We use $Q_r^-(x,t)$ to denote parabolic cylinders that include the top. That is, we set
\[
    Q_r^-(x,t) := B_r (x)\times(t-r,t]\quad\text{and}\quad Q_r^- := Q_r^- (0,0).
\]

\begin{defn}[Relative semi-jets] Let $u: Q_r^- \to \mathbb{R}$ be continuous. For $(x,t)\in Q_{r}^-$, we define the sets
\begin{align*}
\mathcal{P}_{Q_{r}^-}^{2,-}u(x,t):=\Big\{(\theta,\eta,X):u(y,s)\geq &\ u(x,t)+(s-t)\theta+(x-y)\cdot\eta+\frac{1}{2}(x-y)^{\prime}X(x-y)\\
 & +o(\left|x-y\right|^{2}+\left|s-t\right|)\quad\text{as}\quad(y,s)\rightarrow(x,t),(y,s)\in Q_{r}^-\Big\}.
\end{align*}
and 
\begin{align*}
\mathcal{P}_{Q_{r}^-}^{2,+}u(x,t):=\Big\{(\theta,\eta,X):u(y,s)\leq & \ u(x,t)+(s-t)\theta+(x-y)\cdot\eta+\frac{1}{2}(x-y)^{\prime}X(x-y)\\
 & +o(\left|x-y\right|^{2}+\left|s-t\right|)\quad\text{as}\quad(y,s)\rightarrow(x,t),(y,s)\in Q_{r}^-\Big\}.
\end{align*}
We call $\mathcal{P}_{Q_{r}^-}^{2,-}u(x, t)$ the second order subjet of $u$ at $(x, t)$ relative to $Q_r^-$ and $\mathcal{P}_{Q_{r}^-}^{2,+}u(x, t)$ the second order superjet of $u$ at $(x, t)$ relative to $Q_r^-.$
\end{defn}
Suppose that $a:Q_1^- \rightarrow [0, \infty)$ is Lipschitz in space variable. Then we define
\[
    g(x,t,\eta) := ||Da||_{L^\infty(Q_1)}|\eta|^{q-1}+C_f(1+|\eta|^{\beta_1}+a(x,t)|\eta|^{\beta_2}),
\]
where $C_f,\beta_1$ and $\beta_2$ are the constants of the growth condition \eqref{eq:gcnd}.

\begin{defn}[Class $\texttt{S}$]
\label{def:Class S}Let $u:Q_{1}^-\rightarrow\mathbb{\mathbb{R}}$
be continuous. We denote $u\in \texttt{S}(Q_{1}^-)$ if whenever $(\theta,\eta,X)\in\mathcal{P}_{Q_{1}^-}^{2,-}u(x,t)$
with $\eta\not=0$, we have
\begin{align}\label{relaxed visc ineq 1}
\theta-F((x,t),\eta,X)+g((x,t),\eta)\ge0
\end{align}
and whenever $(\theta,\eta,X)\in\mathcal{P}_{Q_{1}^-}^{2,+}u(x,t)$ with $\eta \not=0$,
we have
\begin{align}\label{relaxed visc ineq 2}
\theta-F((x,t),\eta,X)-g((x,t),\eta)\leq0.
\end{align}
\end{defn}

In the definition of class \texttt{S}, we rule out test functions whose gradient vanishes because in the Ishii-Lions method, we only need test functions whose gradient is large.
In the next lemma, we verify that continuous weak solutions to \eqref{eq:p-para f}
belong in \texttt{S} under the natural range of exponents, provided that weak solutions satisfy the comparison principle. The argument is the same as the one used to show that weak solutions are viscosity solutions, but we give details for completeness. The general technique goes back to at least \cite{parabolic viscosity solutions}. In what follows, we denote

\texttt{CWS}:= set of all continuous weak solutions to \eqref{eq:p-para f}, and \texttt{VS}:= set of all viscosity solutions to \eqref{eq:p-para f}.

\begin{lem}[\texttt{CWS} $\subset$ \texttt{S}]\label{cws and S}
Let $1<p\leq q<\infty$, $\beta_{1}\in[1,p)$, $\beta_{2}\in[1,q)$
and $a:Q_{1}^-\rightarrow[0,\infty)$ be Lipschitz in space and continuous in time. Suppose, moreover, that $a$ satisfies \eqref{assumption increasing} in $Q_1$. Let $u:Q_{1}^-\rightarrow\mathbb{R}$ be continuous
and a weak solution to \eqref{eq:p-para f} in $Q_1$.  Then $u\in \texttt{S}(Q_{1}^-)$.
\end{lem}

\begin{proof}
Suppose on the contrary that $u\not\in \texttt{S}(Q_{1}^-)$. Then one of the
inequalities in Definition \ref{def:Class S} fails. We may assume
that it is the first one, as the case of the other one is analogical.
There exists $(\theta,\eta,X)\in\mathcal{P}_{Q_{1}^-}^{2,-}u(x,t)$,
$\eta\not=0$, such that for some $\delta>0$
\[
\theta-F(x,t,\eta,X)+g(x,t,\eta)<-\delta.
\]
A standard argument (as in the elliptic case \cite[Lemma 10, p29]{nikos}) yields $\varphi\in C^{2}(Q_{1}^-)$ such that $(\theta,\eta,X)=(\partial_{t}\varphi(x,t),D\varphi(x,t),D^{2}\varphi(x,t))$,
where $\partial_{t}\varphi(x,t)$ means the left-derivative if $t=T$,
and $\varphi(x,t)=u(x,t)$, $\varphi(y,s)<u(y,s)$ if $(y,s)\not=(x,t)$.
It follows from continuity and the fact $\eta \not=0$ that there is $\rho>0$ such that
\begin{equation}
\partial_{t}\varphi(y,s)-F(y,s,D\varphi(y,s),D^{2}\varphi(y,s))+g(y,s,D\varphi(y,s))<-\delta/2\quad\text{in }Q_{\rho}(x,t).\label{eq:weak is in S 1}
\end{equation}
and $D\varphi \not=0$ in $Q_\rho (x,t)$. Let $a_{j}$ be the standard mollification of $a$ in space direction.
By taking large enough $j\in\mathbb{N}$, we can make $\left\Vert a-a_{j}\right\Vert _{L^{\infty}(Q_{1/2})}$ arbitrarily small. 
Then, assuming $\rho <1/2$ if necessary, we can ensure
that (below we write simply $D\varphi=D\varphi(y,s)$ and so on)
\[
\left|a(y,s)-a_{j}(y,s)\right|\left|\left|D\varphi\right|^{q-2}\left(\tr D^{2}\varphi+(q-2)\frac{D\varphi^{\prime}D^{2}\varphi D\varphi}{\left|D\varphi\right|^{2}}\right)\right|\leq\frac{\delta}{4}\quad\text{in }Q_{\rho}(x,t)
\]
for all large enough $j$. Further, by Lipschitz continuity, we have
$\left\Vert Da_{j}\right\Vert _{L^{\infty}(Q_{1/2})}\le\left\Vert Da\right\Vert _{L^{\infty}(Q_{1})}$.
Now we have in $Q_{\rho}(x,t)$ (below we apply the growth condition \eqref{eq:gcnd} on $f$)
\begin{align*}
 & \partial_{t}\varphi-\div(\left|D\varphi\right|^{p-2}D\varphi+a_{j}(y,s)\left|D\varphi\right|^{q-2}D\varphi)-f(y,s,D\varphi)\\
 & =\partial_{t}\varphi-\div(\left|D\varphi\right|^{p-2}D\varphi)-a_{j}(y,s)\left|D\varphi\right|^{q-2}\left(\tr D^{2}\varphi+(q-2)\frac{D\varphi^{\prime}D^{2}\varphi D\varphi}{\left|D\varphi\right|^{2}}\right)\\
 & \ \ \ -\left|D\varphi\right|^{q-2}D\varphi\cdot Da_{j}(y,s)-f(y,s,D\varphi)\\
 & \leq\partial_{t}\varphi-\div(\left|D\varphi\right|^{p-2}D\varphi)-a(y,s)\left|D\varphi\right|^{q-2}\left(\tr D^{2}\varphi+(q-2)\frac{D\varphi^{\prime}D^{2}\varphi D\varphi}{\left|D\varphi\right|^{2}}\right)+\frac{\delta}{4}\\
 & \ \ \ +\left|D\varphi\right|^{q-1}\left\Vert Da\right\Vert _{L^{\infty}(Q_1)}+C_{f}(1+\left|\eta\right|^{\beta_1}+a(y,s)|\eta|^{\beta_2})\\
 & =\partial_{t}\varphi-F(y,s,D\varphi,D^{2}\varphi)+g(y,s,D\varphi)+\frac{\delta}{4}\\
 & <-\frac{\delta}{4},
\end{align*}
where in the last inequality we used (\ref{eq:weak is in S 1}). Now,
let $\phi\in C_{0}^{\infty}(Q_{\rho}(x,t))$ be a non-negative test
function. Multiplying the above inequality by $\phi$ and integrating
by parts, we obtain
\[
\int_{Q_{\rho}(x,t)}-\varphi \partial_{t}\phi-\left|D\varphi\right|^{p-2}D\varphi-a_{j}(z)\left|D\varphi\right|^{q-2}D\phi-\phi f(z,D\varphi)\d z\leq0.
\]
Thus, letting $j\rightarrow\infty$, we see that $\varphi$ is a weak
subsolution to \eqref{eq:p-para f} in $Q_{\rho}(x,t)$. Then so is $\tilde{\varphi}:=\varphi+l$,
where $l:=\inf_{(y,s)\in\partial_{\mathcal{P}}Q_{\rho}(x,t)}u-\varphi>0$.
Since $\tilde{\varphi}\leq u$ on $\partial_{\mathcal{P}}Q_{\rho}(x,t)$,
it follows from comparison principle in Lemma \ref{thm:comparison principle} that $\tilde{\varphi}\leq u$
in $Q_{\rho}(x,t)$. But this is a contradiction since $\tilde{\varphi}(x,t)=\varphi(x,t)+l=u(x,t)+l>u(x,t)$.
\end{proof}

Next we show that viscosity solutions are in $\texttt{S}$. To extend the viscosity inequality to the top of the cylinder, we use an argument similar to the one in \cite[p209]{DFO14}.

\begin{lem}[\texttt{VS} $\subset$ \texttt{S}]Let $1<p\leq q<\infty$, $\beta_{1}\in[1,p)$, $\beta_{2}\in[1,q)$
and $a:Q_{1}^{-}\rightarrow[0,\infty)$ be Lipschitz in space and
continuous in time. Suppose that $u\in C(Q_{1}^{-})$ is a viscosity
solution to \eqref{eq:p-para f} in $Q_{1}$. Then $u\in\texttt{S}(Q_{1}^{-})$.
\end{lem}
\begin{proof}
Let $(\theta,\eta,X)\in\mathcal{P}_{Q_{1}^{-}}^{2,-}u(x_{0},t_{0})$,
$\eta\not=0$, with $(x_{0},t_{0})\in Q_{1}^{-}$. We need to verify
the inequality \eqref{relaxed visc ineq 1}. If $(x_{0},t_{0})\in Q_{1}$, then $\mathcal{P}_{Q_{1}^{-}}^{2,-}u(x_{0},t_{0})=\mathcal{P}^{2,-}u(x_{0},t_{0})$
and the inequality \eqref{relaxed visc ineq 1} follows immediately from the definition of viscosity solutions. Therefore, it suffices
to consider the case $t_{0}=0$. We may also suppose that $x_{0}=0$
for brevity. By a standard construction, there exists $\varphi\in C^{2}(Q_{1}^{-})$
such that it touches $u$ from below at $(0,0)$ in $Q_{1}^{-}$ and
$(\theta,\eta,X)=(\partial_{t}\varphi(0,0),D\varphi(0,0),D^{2}\varphi(0,0))$,
where $\partial_{t}\varphi(0,0)$ is the left-derivative. That is,
we have
\[
\varphi(0,0)=u(0,0)\quad\text{and}\quad\varphi(x,t)<u(x,t)\quad\text{when }(x,t)\in Q_{1}^{-}\setminus\left\{ (0,0)\right\} .
\]
We denote $w(x,t):=u(x,t)-\varphi(x,t)$ and 
\begin{equation}
\nu(h):=\inf_{(x,t)\in Q_{3/4}\cap\left\{ \left|x\right|+\left|t\right|>h\right\} }w(x,t).\label{eq:dist}
\end{equation}
Observe that $\nu(h)>0$ for all $h>0$. For $\varepsilon\in(0,1)$,
define $\varphi^{\varepsilon}\in C^{2}(Q_{1})$ by
\[
\varphi^{\varepsilon}(x,t):=\varphi(x,t)-\frac{\varepsilon}{\left|t\right|}=\varphi(x,t)+\frac{\varepsilon}{t}
\]
so that
\[
w^{\varepsilon}(x,t):=u(x,t)-\varphi^{\varepsilon}(x,t)=w(x,t)+\varepsilon/\left|t\right|.
\]
Now, let $(x^{\varepsilon},t^{\varepsilon})\in\overline{B}_{1/2}\times[-1/2,0)$
be such that
\begin{equation}
\tilde{w}(x^{\varepsilon},t^{\varepsilon})=\inf_{(y,s)\in\overline{B}_{1/2}\times[-1/2,0)}\tilde{w}(y,s).\label{eq:minimizer}
\end{equation}
We have that $(x^{\varepsilon},t^{\varepsilon})\rightarrow(0,0)$
as $\varepsilon\rightarrow0$. Indeed, otherwise, passing to a subsequence
if necessary, we would have $\left|x^{\varepsilon}\right|+\left|t^{\varepsilon}\right|>h$
for some $h>0$ and all $\varepsilon>0$. Then by (\ref{eq:dist})
\[
\tilde{w}(x^{\varepsilon},t^{\varepsilon})>\nu(h)\quad\text{for all }\varepsilon>0.
\]
However, this is not possible, since $(0,\varepsilon^{1/2}/2)\in\overline{B}_{1/2}\times[-1/2,0)$,
and so by definition of $(x^{\varepsilon},t^{\varepsilon})$ in (\ref{eq:minimizer}),
we have
\[
\tilde{w}(x^{\varepsilon},t^{\varepsilon})\leq\tilde{w}(0,\varepsilon^{1/2}/2)=w(0,\varepsilon^{1/2}/2)+\frac{2\varepsilon}{\varepsilon^{1/2}}\rightarrow0\quad\text{as }\varepsilon\rightarrow0,
\]
where we used continuity of $w$ and that $w(0,0)=0$. 

Since $(x^{\varepsilon},t^{\varepsilon})\rightarrow(0,0)$ and $\eta=\varphi(0,0)\not=0$,
by continuity we have $D\varphi(x^{\varepsilon},t^{\varepsilon})\not=0$
for all small $\varepsilon>0$. Therefore, since $u-\varphi^{\varepsilon}$
has a minimum at $(x^{\varepsilon},t^{\varepsilon})$ and $u$ is a viscosity solution, we have
\[
\partial_{t}\varphi^{\varepsilon}(x^{\varepsilon},t^{\varepsilon})\geq\Delta_{p}\varphi(x^{\varepsilon},t^{\varepsilon})+a(x^{\varepsilon},t^{\varepsilon})\Delta_{q}\varphi(x^{\varepsilon},t^{\varepsilon})-g(x^{\varepsilon},t^{\varepsilon},D\varphi(x^{\varepsilon},t^{\varepsilon})).
\]
On the other hand, we have also
\[
\partial_{t}\varphi^{\varepsilon}(x^{\varepsilon},t^{\varepsilon})=\partial_{t}\varphi(x^{\varepsilon},t^{\varepsilon})-\frac{\varepsilon}{t^{2}}\leq\partial_{t}\varphi(x^{\varepsilon},t^{\varepsilon}).
\]
Letting $\varepsilon\rightarrow0$ and using that by continuity $(\partial_{t}\varphi(x^{\varepsilon},t^{\varepsilon}),D\varphi(x^{\varepsilon},t^{\varepsilon}),D^{2}\varphi(x^{\varepsilon},t^{\varepsilon}))$
converges to $(\theta,\eta,X)$, we obtain \eqref{relaxed visc ineq 1}.
\end{proof}

\subsection{The Ishii-Lions method} Comparing two solutions of a partial differential equation is a fundamental step toward establishing uniqueness. For a wide class of elliptic and parabolic equations, such comparison principles can indeed be proved. When working within the framework of viscosity solutions, a common and powerful technique is the so-called ``doubling of variables''. A similar idea appears in the context of hyperbolic conservation laws and Hamilton-Jacobi equations, where the $L^1$-contraction principle is also derived via doubling the variables, which essentially leads to the uniqueness of solutions.

In the context of the comparison principle, the core idea is to construct a test function that penalizes the difference between the solutions at two points and analyze its maximum. To illustrate, one considers the function
\begin{align*}
    \Phi_{u, v, \gamma}(x, y):=u(x)-v(y)-\underbrace{\frac{\gamma}{2}|x-y|^2}_{=\varphi(|x-y|)}
\end{align*}
studies its maximum over $\Omega\times \Omega,$ sending $\gamma \to +\infty.$ This structure allows one to compare to $u$ and $v$ and derive uniqueness results.

In their seminal work \cite{IL90}, Ishii and Lions observed that this method can also be adapted to obtain regularity estimates for viscosity solutions. Specifically, to prove that a solution $u$ is spatially continuous with modulus $\varphi$, it suffices to show that
\[
\sup_{(x,y) \in \overline{\Omega}\times \overline{\Omega}} \left( u(x) - u(y) - \varphi(|x - y|) \right) \leq 0.
\]
Assuming the contrary--i.e., the supremum is positive--one arrives at a contradiction by analyzing the behavior at the maximum point. If this contradiction can be reached, the conclusion is that $u$ is continuous with the modulus of continuity governed by $\varphi$.

In the parabolic setting, we need to modify this idea to account for the time variable. That is, we consider
\[
\Phi(x, y, t) := u(x, t) - u(y, t) - L \varphi(|x - y|),
\]
and aim to show that $\Phi(x, y, t) \leq 0$ for all $(x, y, t) \in B_1 \times B_1 \times (0, T)$. To prove this, we assume toward a contradiction that there exists a point $(x_0, y_0, t_0)$ such that $\Phi(x_0, y_0, t_0) > 0$. The contradiction then arises by applying the viscosity inequality at the maximum point and carefully estimating the derivatives using the structure of the equation. The proof proceeds in two stages:

\begin{enumerate}
\item [(1)] \textbf{H\"older regularity:} We first choose $\varphi(s) = s^{\alpha}$ for some $\alpha \in (0, 1]$ and establish that $u$ is spatially H\"older continuous. The concavity of $\varphi$ plays a key role in handling the second-order terms in the viscosity inequalities.
    
\item [(2)] \textbf{Lipschitz regularity:} With the H\"older continuity already established, we refine the estimate by choosing a test function of the form $\varphi(s) = s - \kappa s^{\beta}$ for appropriate constants $\kappa > 0$ and $\beta \in (1, 2)$. This perturbation accounts for the fact that the second derivative of the Lipschitz candidate $\varphi(s) = s$ vanishes, and the small convex correction $\kappa s^{\beta}$ provides the necessary curvature to close the argument. This ultimately yields the desired spatial Lipschitz continuity of the solution $u$.
\end{enumerate}

To formalize the common structure shared by both stages, we isolate a key lemma, which encapsulates the basic estimate (see \eqref{eq:lemma main est}). This lemma is then applied with different choices of $\varphi$ to derive the respective continuity estimates. Here we follow \cite{LPS25}.

\begin{lem}
\label{lem:Ishii-Lions lemma} Suppose that $1\leq p\leq q<\infty$,
$\beta_{1}\in[1,\infty),$ $\beta_{2}\in[1,\infty)$. Suppose that $a:Q_1^- \rightarrow [0,\infty)$ is continuous in time and Lipschitz in space. Let $u\in \texttt{S}(Q_{1}^-)$
be bounded and suppose that there is an increasing modulus $\omega:[0,\infty)\rightarrow[0,\infty)$,
$\lim_{s\rightarrow0}\omega(s)=0$, such that 
\[
\left|u(x,t)-u(y,t)\right|\leq\omega(\left|x-y\right|)\quad\text{for all }(x,t),(y,t)\in Q_{1}.
\]
 Let $(x_{0},t_{0}),(y_{0},t_{0})\in Q_{1/2}^-$ and define the function
\begin{equation}
\Psi(x,y,t):=u(x,t)-u(y,t)-L\varphi(\left|x-y\right|)-\frac{K}{2}\left|x-x_{0}\right|^{2}-\frac{K}{2}\left|y-y_{0}\right|^{2}-\frac{K}{2}\left|t-t_{0}\right|^{2},\label{eq:lemma blah}
\end{equation}
where $K:=8\osc_{Q_{1}}u$, and $\varphi:[0,2]\rightarrow[0,\infty)$
is a $C^{2}$-function in $(0,2)$ such that
\begin{equation}
\varphi(0)=0,\quad\ensuremath{\left|\varphi^{\prime\prime}(s)\right|<\frac{\varphi^{\prime}(s)}{s}}\quad\text{and}\quad\varphi^{\prime\prime}<0<c_{\varphi}<\varphi^{\prime}\text{\ensuremath{\quad\text{for }\text{some }c_{\varphi}>0.}}\label{eq:Ishii-Lions lemma cnd}
\end{equation}
Then, if $L>L^{\prime}$ for some $L^{\prime}$ that depends only
on $c_{\varphi}$ and $\osc_{Q_{1}}u$, the following holds: If $\Psi$
has a positive maximum at $(\hat{x},\hat{y},\hat{t})\in\overline{B}_{1}\times\overline{B}_{1}\times[-1,0]$,
we have
\begin{align}
-K\leq\: & C(L\varphi^{\prime})^{p-2}\Big(L\varphi^{\prime\prime}+(L\varphi^{\prime})^{\gamma}+\sqrt{K}(L\varphi^{\prime})^{\gamma-1}\frac{\omega^{1/2}(\left|z\right|)}{\left|z\right|}\Big)\nonumber \\
 & +C(a(\hat{x},\hat{t})+a(\hat{y},\hat{t}))(L\varphi^{\prime})^{q-2}\left(L\varphi^{\prime\prime}+(L\varphi^{\prime})^{\gamma}\right),\label{eq:lemma main est}
\end{align}
where $\left|z\right|:=\left|\hat{x}-\hat{y}\right|$, $\varphi^{\prime}:=\varphi^{\prime}(\left|z\right|)$,
$\varphi^{\prime\prime}:=\varphi^{\prime\prime}(\left|z\right|)$
and 
\begin{equation}
\gamma:=\max(q-p,\beta_{1}-p+1,\beta_{2}-q+1)+1.\label{eq:gamma}
\end{equation}
The constant $C\geq1$ depends only on $N$, $p$, $q$, $\left\Vert a\right\Vert _{L^\infty (Q_1)}$,
$\left\Vert Da\right\Vert _{L^\infty (Q_1)}$, $C_f$, $\beta_{1}$
and $\beta_{2}$. Finally, we have the estimate
\begin{equation}
\varphi^{\prime}(\left|z\right|)\leq\frac{\omega(\left|z\right|)}{L\left|z\right|}.\label{eq:lemma derivative est}
\end{equation}
\end{lem}

\begin{proof}
First, we have $\left|z\right|\not=0$ since otherwise the maximum
at $(\hat{x},\hat{y},\hat{t})$ would be non-positive. Secondly, we
have
\begin{align}
0 & <\left|u(\hat{x},\hat{t})-u(\hat{y},\hat{t})\right|-L\varphi(\left|\hat{x}-\hat{y}\right|)-\frac{K}{2}\left|\hat{x}-x_{0}\right|^{2}-\frac{K}{2}\left|\hat{y}-y_{0}\right|^{2}-\frac{K}{2}\left|\hat{t}-t_{0}\right|^{2}\label{eq:lemma whatever}
\end{align}
so that 
\begin{equation}
\left|\hat{t}-t_{0}\right|,\left|\hat{x}-x_{0}\right|,\left|\hat{y}-y_{0}\right|\leq\sqrt{\frac{2}{K}\left|u(\hat{x},\hat{t})-u(\hat{y},\hat{t})\right|}\leq\sqrt{\frac{1}{4}}=\frac{1}{2}.\label{eq:lemma est 11}
\end{equation}
Since $x_{0},y_{0}\in B_{1/2}$ and $t_{0}\in(-1/2,0]$, this implies
$\hat{x},\hat{y}\in B_{1}$ and $\hat{t}\in(-1,0]$, which means that
the maximum point $(\hat{x},\hat{y},\hat{t})$ is in $B_{1}\times B_{1}\times(-1,0]$.

By definition of $\omega$ and uniform continuity of $u$, we have
\begin{equation}
\left|u(x,s)-u(y,s)\right|\leq\omega(\left|x-y\right|)\quad\text{for all }(x,s),(y,s)\in Q_1^-,\label{eq:3-1}
\end{equation}
Therefore, (\ref{eq:lemma est 11}) implies that 
\begin{align*}
K\left|\hat{t}-t_{0}\right|,K\left|\hat{x}-x_{0}\right|,K\left|\hat{y}-y_{0}\right|\leq & \sqrt{K}\omega^{1/2}(\left|\hat{x}-\hat{y}\right|).
\end{align*}
We also have by concavity of $\varphi$ and (\ref{eq:lemma whatever})
that $\left|z\right|\varphi^{\prime}(\left|z\right|)\leq\int_{0}^{\left|z\right|}\varphi^{\prime}(s)\d s\leq\frac{\omega(\left|z\right|)}{L}$
so that (\ref{eq:lemma derivative est}) holds.

Since $\hat{x}\not=\hat{y}$, the function $\phi(x,y)\mapsto\varphi(\left|x-y\right|)$
is $C^{2}$ in the neighborhood of $(\hat{x},\hat{y})$. Therefore,
we may invoke the parabolic Theorem of sums \cite[Theorem 9]{DFO14}
(see also \cite[Theorem 8.3]{userguide}). For any $\mu>0$, there
exist matrices $X,Y\in S(N)$ and $b_{1},b_{2}\in\mathbb{R}$ such
that 
\[
b_{1}+b_{2}\geq\partial_{t}(L\varphi(\left|x-y\right|))(\hat{x},\hat{y})=0
\]
and
\begin{align*}
(b_{1},D_{x}(L\varphi(\left|x-y\right|))(\hat{x},\hat{y}),X) & \in\overline{\mathcal{P}}_{Q_{1}^-}^{2,+}(u-\frac{K}{2}\left|x-x_{0}\right|^{2}-\frac{K}{2}\left|t-t_{0}\right|^{2})(\hat{x},\hat{t}),\\
(-b_{2},-D_{y}(L\varphi(\left|x-y\right|))(\hat{x},\hat{y}),Y) & \in\overline{\mathcal{P}}_{Q_{1}^-}^{2,-}(u+\frac{K}{2}\left|y-y_{0}\right|^{2})(\hat{x},\hat{y}).
\end{align*}
Denoting $z:=\hat{x}-\hat{y}$ and
\begin{align*}
\eta_{1} & :=L\varphi^{\prime}(\left|z\right|)\frac{z}{\left|z\right|}+K(\hat{x}-x_{0}),\\
\eta_{2} & :=L\varphi^{\prime}(\left|z\right|)\frac{z}{\left|z\right|}-K(\hat{y}-y_{0}),\\
\theta_{1} & :=b_{1}+K(\hat{t}-t_{0}),\\
\theta_{2} & :=-b_{2},
\end{align*}
these can be written as
\begin{align*}
(\theta_{1},\eta_{1},X+KI)\in\overline{\mathcal{P}}_{Q_{1}^-}^{2,+}u(\hat{x},\hat{t}),\quad & (\theta_{2},\eta_{2},Y-KI)\in\overline{\mathcal{P}}_{Q_{1}^-}^{2,-}u(\hat{y},\hat{t}).
\end{align*}
Assuming $L$ is large enough depending on $K$, $\osc_{Q_{1}}u$
and $c_{\varphi}$, we have
\begin{align}
\left|\eta_{1}\right|,\left|\eta_{2}\right| & \leq L\varphi^{\prime}(\left|z\right|)+\sqrt{K}\omega^{1/2}(\left|\hat{x}-\hat{y}\right|)\leq L\varphi^{\prime}(\left|z\right|)+Lc_{\varphi}\leq2L\varphi^{\prime}(\left|z\right|)\nonumber \\
\left|\eta_{1}\right|,\left|\eta_{2}\right| & \geq\frac{1}{2}L\varphi^{\prime}(\left|z\right|)+\frac{1}{2}Lc_{\varphi}-\sqrt{K}\omega^{1/2}(\left|\hat{x}-\hat{y}\right|)\geq\frac{1}{2}L\varphi^{\prime}(\left|z\right|)\geq1,\label{eq:lemma gradient estimate}
\end{align}
Furthermore, by Theorem of sums, we have 
\begin{align}
-(\mu+4\left\Vert B\right\Vert ) & I\leq\begin{pmatrix}X & 0\\
0 & -Y
\end{pmatrix}\leq\begin{pmatrix}B & -B\\
-B & B
\end{pmatrix}+\frac{2}{\mu}\begin{pmatrix}B^{2} & -B^{2}\\
-B^{2} & B^{2}
\end{pmatrix},\label{eq:lemma matrix ineq}
\end{align}

where
\begin{align*}
B & =L\varphi^{\prime\prime}(\left|z\right|)\frac{z\otimes z}{\left|z\right|^{2}}+L\frac{\varphi^{\prime}(\left|z\right|)}{\left|z\right|}\left(I-\frac{z\otimes z}{\left|z\right|^{2}}\right),\\
B^{2}=BB & =L^{2}\left|\varphi^{\prime\prime}(\left|z\right|)\right|^{2}\frac{z\otimes z}{\left|z\right|^{2}}+L^{2}\frac{\left|\varphi^{\prime}(\left|z\right|)\right|^{2}}{\left|z\right|^{2}}\left(I-\frac{z\otimes z}{\left|z\right|^{2}}\right).
\end{align*}
By assumptions on $\varphi$, we have
\[
\left|\varphi^{\prime\prime}(\left|z\right|)\right|\leq\varphi^{\prime}(\left|z\right|)/\left|z\right|\quad\text{and}\quad\varphi^{\prime\prime}(\left|z\right|)<0<\varphi^{\prime}(\left|z\right|).
\]
Thus we deduce that 
\begin{equation}
\left\Vert B\right\Vert \leq\frac{L\varphi^{\prime}(\left|z\right|)}{\left|z\right|}\quad\text{and}\quad\left\Vert B^{2}\right\Vert \leq L^{2}\frac{(\varphi^{\prime}(\left|z\right|))^{2}}{\left|z\right|^{2}}.\label{eq:lemma B est}
\end{equation}
Moreover, we choose
\[
\mu:=\frac{2L\varphi^{\prime}(\left|z\right|)}{\left|z\right|},
\]
so that
\begin{equation}
\left\langle B\frac{z}{\left|z\right|},\frac{z}{\left|z\right|}\right\rangle +\frac{2}{\mu}\left\langle B^{2}\frac{z}{\left|z\right|},\frac{z}{\left|z\right|}\right\rangle \leq\frac{L}{2}\varphi^{\prime\prime}(\left|z\right|).\label{eq:lemma est 33}
\end{equation}
Denote below
\begin{align*}
F((z,s),\zeta,Z) & :=\left|\zeta\right|^{p-2}(\tr Z+(p-2)\frac{\zeta^{\prime}Z\zeta}{\left|\zeta\right|^{2}})+a((z,s))\left|\zeta\right|^{q-2}(\tr Z+(q-2)\frac{\zeta^{\prime}Z\zeta}{\left|\zeta\right|^{2}}).
\end{align*}
Since $u\in \texttt{S}(Q_{1}^-)$, we have
\[
\theta_{1}\leq F((\hat{x},\hat{t}),\eta_{1},X+KI)+g(\hat{y},\left|\eta_{1}\right|)\quad\text{and\ensuremath{\quad}}\theta_{2}\geq F((\hat{y},\hat{t}),\eta_{2},Y-KI)-g(\hat{x},\left|\eta_{2}\right|)
\]
We denote $\tilde{\eta}:=L\varphi^{\prime}(\left|z\right|)z/\left|z\right|$.
Then we subtract the equations and use that $\theta_{1}-\theta_{2}=b_{1}+b_{2}+K(\hat{t}-t_{0})\geq K(\hat{t}-t_{0})$,
we obtain 
\begin{align*}
K(\hat{t}-t_{0}) & \leq F((\hat{x},\hat{t}),\eta_{1},X+KI)-F((\hat{y},\hat{t}),\eta_{2},Y-KI)+g((\hat{x},\hat{t}),\left|\eta_{1}\right|)+g((\hat{y},\hat{t}),\left|\eta_{2}\right|)\\
 & =F((\hat{x},\hat{t}),\eta_{1},X)-F((\hat{y},\hat{t}),\eta_{2},Y)\\
 & \ \ \ +F((\hat{x},\hat{t}),\eta_{1},X+KI)-F((\hat{x},\hat{t}),\eta_{1},X)-F((\hat{y},\hat{t}),\eta_{2},Y-KI)+F((\hat{y},\hat{t}),\eta_{2},Y)\\
 & \ \ \ +g((\hat{x},\hat{t}),\left|\eta_{1}\right|)+g((\hat{y},\hat{t}),\left|\eta_{2}\right|)\\
 & =F((\hat{x},\hat{t}),\tilde{\eta},X)-F((\hat{y},\hat{t}),\tilde{\eta},Y)\\
 & \ \ \ +F((\hat{x},\hat{t}),\eta_{1},X)-F((\hat{x},\hat{t}),\tilde{\eta},X)+F((\hat{y},\hat{t}),\tilde{\eta},Y)-F((\hat{y},\hat{t}),\eta_{2},Y)\\
 & \ \ \ +F((\hat{x},\hat{t}),\eta_{1},X+KI)-F((\hat{x},\hat{t}),\eta_{1},X)-F((\hat{y},\hat{t}),\eta_{2},Y-KI)+F((\hat{y},\hat{t}),\eta_{2},Y)\\
 & \ \ \ +g((\hat{x},\hat{t}),\left|\eta_{1}\right|)+g((\hat{x},\hat{t}),\left|\eta_{2}\right|)\\
 & =:T_{1}+T_{2}+T_{3}+T_{4}.
\end{align*}

\textbf{Estimate of $T_{1}$:} Using the matrix inequality (\ref{eq:lemma matrix ineq})
with $b,c\in[0,\infty)$, we obtain for all $\xi\in\mathbb{R}^{N}$
with $\left|\xi\right|=1$ that
\begin{align}
\xi^{\prime}(bX-cY)\xi & =\begin{pmatrix}b^{1/2}\xi\\
c^{1/2}\xi
\end{pmatrix}^{\prime}\begin{pmatrix}X & 0\\
0 & -Y
\end{pmatrix}\begin{pmatrix}b^{1/2}\xi\\
c^{1/2}\xi
\end{pmatrix}\nonumber \\
 & \le\begin{pmatrix}b^{1/2}\xi\\
c^{1/2}\xi
\end{pmatrix}^{\prime}\left(\begin{pmatrix}B & -B\\
-B & B
\end{pmatrix}+\frac{2}{\mu}\begin{pmatrix}B^{2} & -B^{2}\\
-B^{2} & B^{2}
\end{pmatrix}\right)\begin{pmatrix}b^{1/2}\xi\\
c^{1/2}\xi
\end{pmatrix}\nonumber \\
 & \leq(b^{1/2}-c^{1/2})^{2}\left(\xi^{\prime}B\xi+\frac{2}{\mu}\xi^{\prime}B^{2}\xi\right).\nonumber \\
 & \leq(b^{1/2}-c^{1/2})^{2}\left(\left\Vert B\right\Vert +\frac{2}{\mu}\left\Vert B\right\Vert ^{2}\right)\nonumber \\
 & \leq2(b^{1/2}-c^{1/2})^{2}\frac{L\varphi^{\prime}(\left|z\right|)}{\left|z\right|},\label{eq:bX-cY est 1}
\end{align}
where in the last inequality we used (\ref{eq:lemma B est}). On the
other hand, we have similarly
\begin{align}
\frac{z^{\prime}}{\left|z\right|}(bX-cY)\frac{z}{\left|z\right|} & =\begin{pmatrix}b^{1/2}z/\left|z\right|\\
-c^{1/2}z/\left|z\right|
\end{pmatrix}^{\prime}\begin{pmatrix}X & 0\\
0 & -Y
\end{pmatrix}\begin{pmatrix}b^{1/2}z/\left|z\right|\\
-c^{1/2}z/\left|z\right|
\end{pmatrix}\nonumber \\
 & \leq(b^{1/2}+c^{1/2})^{2}\left(\frac{z^{\prime}}{\left|z\right|}B\frac{z}{\left|z\right|}+\frac{2}{\mu}\frac{z^{\prime}}{\left|z\right|}B^{2}\frac{z}{\left|z\right|}\right)\nonumber \\
 & =(b^{1/2}+c^{1/2})^{2}\frac{L}{2}\varphi^{\prime\prime}(\left|z\right|),\label{eq:bX-cY est 2}
\end{align}
where in the last estimate we used (\ref{eq:lemma est 33}). Now,
denoting by $\lambda_{\min}$ and $\lambda_{\max}$ the minimum and
maximum eigenvalues of a matrix, we obtain using (\ref{eq:bX-cY est 1})
and (\ref{eq:bX-cY est 2}) with $b=a(\hat{x},\hat{t})$ and $c=a(\hat{y},\hat{t})$
\begin{align*}
 & \tr(a(\hat{x},\hat{t})X-a(\hat{y},\hat{t})Y)+(q-2)\frac{\tilde{\eta}^{\prime}(a(\hat{x},\hat{t})X-a(\hat{y},\hat{t})Y)\tilde{\eta}}{\left|\tilde{\eta}\right|^{2}}\\
 & \leq\lambda_{\min}(a(\hat{x},\hat{t})X-a(\hat{y},\hat{t})Y)+(N-1)\lambda_{\max}(a(\hat{x},\hat{t})X-a(\hat{y},\hat{t})Y)\\
 & \ \ \ +(q-2)\frac{\tilde{\eta}^{\prime}(a(\hat{x},\hat{t})X-a(\hat{y},\hat{t})Y)\tilde{\eta}}{\left|\tilde{\eta}\right|^{2}}\\
 & \leq(q-1)\frac{\tilde{\eta}^{\prime}(a(\hat{x},\hat{t})X-a(\hat{y},\hat{t})Y)\eta}{\left|\tilde{\eta}\right|^{2}}+(N-1)\lambda_{\max}(a(\hat{x},\hat{t})X-a(\hat{y},\hat{t})Y)\\
 & \leq\frac{q-1}{2}(a^{1/2}(\hat{x},\hat{t})+a^{1/2}(\hat{y},\hat{t}))^{2}L\varphi^{\prime\prime}(\left|z\right|)+2(N-1)(a^{1/2}(\hat{x},\hat{t})-a^{1/2}(\hat{y},\hat{t}))^{2}L\frac{\varphi^{\prime}(\left|z\right|)}{\left|z\right|}.\\
 & \leq C(q)(a(\hat{x},\hat{t})+a(\hat{y},\hat{t}))L\varphi^{\prime\prime}(\left|z\right|)+C(N,\left\Vert Da\right\Vert _{L^{\infty}})L\left|\tilde{\eta}\right|^{q-2}\varphi^{\prime}(\left|z\right|),
\end{align*}
where we used that $\left|a^{1/2}(\hat{x},\hat{t})-a^{1/2}(\hat{y},\hat{t})\right|^{2}\leq C\left|a(\hat{x},\hat{t})-a(\hat{x},\hat{t})\right|\leq C\left\Vert Da\right\Vert _{L^{\infty}}\left|\hat{x}-\hat{y}\right|$.
Thus we obtain
\begin{align*}
T_{1} & =F((\hat{x},\hat{t}),\tilde{\eta},X)-F((\hat{y},\hat{t}),\tilde{\eta},Y)\\
 & =\left|\tilde{\eta}\right|^{p-2}\left(\tr(X-Y)+(p-2)\frac{\tilde{\eta}^{\prime}(X-Y)\tilde{\eta}}{\left|\tilde{\eta}\right|^{2}}\right)\\
 & \ \ \ +\left|\tilde{\eta}\right|^{q-2}\left(\tr(a(\hat{x},\hat{t})X-a(\hat{y},\hat{t})Y)+(q-2)\frac{\tilde{\eta}^{\prime}(a(\hat{x},\hat{t})X-a(\hat{y},\hat{t})Y)\tilde{\eta}}{\left|\tilde{\eta}\right|^{2}}\right)\\
 & \le C(p)\left|\tilde{\eta}\right|^{p-2}L\varphi^{\prime\prime}(\left|z\right|)+C(q)(a(\hat{x},\hat{t})+a(\hat{y},\hat{t}))\left|\tilde{\eta}\right|^{q-2}L\varphi^{\prime\prime}(\left|z\right|)\\
 & \ \ \ +C(N,\left\Vert Da\right\Vert _{L^{\infty}})L\left|\tilde{\eta}\right|^{q-2}\varphi^{\prime}(\left|z\right|).
\end{align*}

\textbf{Estimate of $T_{2}$:} First we show the following algebraic
inequality
\begin{align}
 & \left|\left|\xi_{1}\right|^{p-2}(\tr Z+(p-2)\left|\xi_{1}\right|^{-2}\xi_{1}^{\prime}Z\xi_{1})-\left|\xi_{2}\right|^{p-2}(\tr Z+(p-2)\left|\xi_{2}\right|^{-2}\xi_{2}^{\prime}Z\xi_{2})\right|\nonumber \\
 & \ \ \ \leq C(N,p)\max(\left|\xi_{1}\right|^{p-3},\left|\xi_{2}\right|^{p-3})\left|\xi_{1}-\xi_{2}\right|\left\Vert Z\right\Vert \label{eq:ineq}
\end{align}
for all $\xi_{1},\xi_{2}\in\mathbb{R}^{N}\setminus\left\{ 0\right\} $
and $Z\in S(N)$. To this end, we observe that by \cite{The75}
\begin{align*}
\left|\xi_{1}^{\prime}Z\xi_{1}-\xi_{2}Z\xi_{2}\right| & =\left|\tr((\xi_{1}\otimes\xi_{1}-\xi_{2}\otimes\xi_{2})Z)\right|\\
 & \leq N\left\Vert \xi_{1}\otimes\xi_{1}-\xi_{2}\otimes\xi_{2}\right\Vert \left\Vert Z\right\Vert \\
 & =N\left\Vert (\xi_{1}-\xi_{2})\otimes\xi_{1}-\xi_{2}\otimes(\xi_{2}-\xi_{1})\right\Vert \left\Vert Z\right\Vert \\
 & \leq N(\left|\xi_{1}\right|+\left|\xi_{2}\right|)\left\Vert \xi_{1}-\xi_{2}\right\Vert \left\Vert Z\right\Vert .
\end{align*}
We also use the elementary inequality
\[
\left|a^{r}-b^{r}\right|\leq |r|\max(a^{r-1},b^{r-1})\left|a-b\right|\quad\text{for all }a,b>0\quad\text{and}\quad r\in\mathbb{R}.
\]
Furthermore, may assume that $\left|\xi_{2}\right|\leq\left|\xi_{1}\right|$
by changing notation if necessary. Using these, we estimate
\begin{align*}
 & \left|\left|\xi_{1}\right|^{p-2}-\left|\xi_{2}\right|^{p-2}\right|\left|\tr Z\right|+(p-2)\left|\xi_{1}\right|^{p-4}\left|\xi_{1}^{\prime}Z\xi_{1}-\xi_{2}^{\prime}Z\xi_{2}\right|+(p-2)\left|\left|\xi_{1}\right|^{p-4}-\left|\xi_{2}\right|^{p-4}\right|\left|\xi_{2}^{\prime}Z\xi_{2}\right|\\
 & \leq N\left|p-2\right|\max(\left|\xi_{1}\right|^{p-3},\left|\xi_{2}\right|^{p-3})\left|\xi_{1}-\xi_{2}\right|\left\Vert Z\right\Vert +N\left|p-2\right|\left|\xi_{1}\right|^{p-4}(\left|\xi_{1}\right|+\left|\xi_{2}\right|)\left\Vert \xi_{1}-\xi_{2}\right\Vert \left\Vert Z\right\Vert \\
 & \ \ \ +\left|p-2\right|\left|p-4\right|\max(\left|\xi_{1}\right|^{p-5},\left|\xi_{2}\right|^{p-5})\left|\xi_{1}-\xi_{2}\right|\left|\xi_{2}\right|^{2}\left\Vert Z\right\Vert \\
 & \leq C(N,p)\left|p-2\right|\max(\left|\xi_{1}\right|^{p-3},\left|\xi_{2}\right|^{p-3})\left|\xi-\xi_{2}\right|\left\Vert Z\right\Vert ,
\end{align*}
which proves the algebraic inequality (\ref{eq:ineq}). 

Observe now that
\begin{equation}
\left|\eta_{1}-\tilde{\eta}\right|+\left|\eta_{2}-\tilde{\eta}\right|=K(\left|x_{0}-\hat{x}\right|+\left|y_{0}-\hat{y}\right|)\leq2\sqrt{K}\omega^{1/2}(\left|z\right|).\label{eq:lemma gradient tilde est}
\end{equation}
and that by (\ref{eq:lemma matrix ineq}) we have
\begin{equation}
\left\Vert X\right\Vert ,\left\Vert Y\right\Vert \leq\max(\left\Vert B\right\Vert +\frac{2}{\mu}\left\Vert B\right\Vert ^{2},\mu+4\left\Vert B\right\Vert )\leq C\frac{L\varphi^{\prime}(\left|z\right|)}{\left|z\right|}.\label{eq:matrix norms}
\end{equation}
Therefore, using (\ref{eq:ineq}), we obtain
\begin{align*}
T_{2} & =F((\hat{x},\hat{t}),\eta_{1},X)-F((\hat{x},\hat{t}),\tilde{\eta},X)+F((\hat{y},\hat{t}),\tilde{\eta},Y)-F((\hat{y},\hat{t}),\eta_{2},Y)\\
 & =\left|\eta_{1}\right|^{p-2}\tr(X-(p-2)\left|\eta_{1}\right|^{-2}\eta_{1}^{\prime}X\eta_{1})-\left|\tilde{\eta}\right|^{p-2}\tr(X-(p-2)\left|\tilde{\eta}\right|^{-2}\tilde{\eta}^{\prime}X\tilde{\eta})\\
 & \ \ \ +\left|\tilde{\eta}\right|^{p-2}\tr(Y-(p-2)\left|\tilde{\eta}\right|^{-2}\tilde{\eta}^{\prime}Y\tilde{\eta}))-\left|\eta_{2}\right|^{p-2}\tr(Y-(p-2)\left|\eta_{2}\right|^{-2}\eta_{2}^{\prime}Y\eta_{2})\\
 & \ \ \ +a(\hat{x},\hat{t})\left|\eta_{1}\right|^{q-2}\tr(X-(q-2)\left|\eta_{1}\right|^{-2}\eta_{1}^{\prime}X\eta_{1})-a(\hat{x},\hat{t})\left|\tilde{\eta}\right|^{q-2}\tr(X-(p-2)\left|\tilde{\eta}\right|^{-2}\tilde{\eta}^{\prime}X\tilde{\eta})\\
 & \ \ \ +a(\hat{y},\hat{t})\left|\tilde{\eta}\right|^{q-2}\tr(Y-(q-2)\left|\tilde{\eta}\right|^{-2}\tilde{\eta}^{\prime}Y\tilde{\eta}))-a(\hat{y},\hat{t})\left|\eta_{2}\right|^{q-2}\tr(Y-(q-2)\left|\eta_{2}\right|^{-2}\eta_{2}^{\prime}Y\eta_{2})\\
 & \leq C(N,p)\max(\left|\eta_{1}\right|^{p-3},\left|\eta_{2}\right|^{p-3},\left|\tilde{\eta}\right|^{p-3})(\left\Vert X\right\Vert +\left\Vert Y\right\Vert )(\left|\eta_{1}-\tilde{\eta}\right|+\left|\eta_{2}-\tilde{\eta}\right|)\\
 & \ \ \ +C(N,q,\left\Vert a\right\Vert _{L^{\infty}})\max(\left|\eta_{1}\right|^{q-3},\left|\eta_{2}\right|^{q-3},\left|\tilde{\eta}\right|^{q-3})(\left\Vert X\right\Vert +\left\Vert Y\right\Vert )(\left|\eta_{1}-\tilde{\eta}\right|+\left|\eta_{2}-\tilde{\eta}\right|)\\
 & \leq C(N,p)\left|\tilde{\eta}\right|^{p-3}\sqrt{K}\omega^{1/2}(\left|z\right|)L\frac{\varphi^{\prime}(\left|z\right|)}{\left|z\right|}+C(N,q,\left\Vert a\right\Vert _{L^{\infty}})\left|\tilde{\eta}\right|^{q-3}\sqrt{K}\omega^{1/2}(\left|z\right|)L\frac{\varphi^{\prime}(\left|z\right|)}{\left|z\right|},
\end{align*}
where in the last estimate we used (\ref{eq:matrix norms}), (\ref{eq:lemma gradient tilde est})
and that (\ref{eq:lemma gradient estimate}) implies $\left|\eta_{1}\right|,\left|\eta_{2}\right|\in[\frac{1}{2}\left|\tilde{\eta}\right|,2\left|\tilde{\eta}\right|]$.

\textbf{Estimate of $T_{3}$:} Since $F$ is quasilinear, we have
\begin{align*}
T_{3} & =F((\hat{x},\hat{t}),\eta_{1},X+KI)-F((\hat{x},\hat{t}),\eta_{1},X)-F((\hat{y},\hat{t}),\eta_{2},Y-KI)+F((\hat{y},\hat{t}),\eta_{2},Y)\\
 & =F((\hat{x},\hat{t}),\eta_{1},KI)+F((\hat{y},\hat{t}),\eta_{2},KI)\\
 & \le C(N,p,q,\left\Vert a\right\Vert _{L^{\infty}})K(\left|\eta_{1}\right|^{p-2}+\left|\eta_{2}\right|^{p-2}+\left|\eta_{1}\right|^{q-2}+\left|\eta_{2}\right|^{q-2})\\
 & \leq C(N,p,q,\left\Vert a\right\Vert _{L^{\infty}})K(\left|\tilde{\eta}\right|^{p-2}+\left|\tilde{\eta}\right|^{q-2}),
\end{align*}
where we also used (\ref{eq:lemma gradient tilde est}) and definition
of $\tilde{\eta}$.

\textbf{Estimate of $T_{4}$:} We have by (\ref{eq:lemma gradient estimate})
and definition of $\tilde{\eta}$
\begin{align*}
T_{4} & =g((\hat{x},\hat{t}),\left|\eta_{1}\right|)+g((\hat{y},\hat{t}),\left|\eta_{2}\right|)\\
 & =\left\Vert Da\right\Vert _{L^{\infty}}\left|\eta_{1}\right|^{q-1}+C_{f}(1+\left|\eta_{1}\right|^{\beta_{1}}+a(\hat{x},\hat{t})\left|\eta_{1}\right|^{\beta_{2}})\\
 & \ \ \ +\left\Vert Da\right\Vert _{L^{\infty}}\left|\eta_{2}\right|^{q-1}+C_{f}(1+\left|\eta_{2}\right|^{\beta_{1}}+a(\hat{x},\hat{t})\left|\eta_{2}\right|^{\beta_{2}})\\
 & \le C(q,\left\Vert Da\right\Vert _{L^{\infty}},C_{f},\beta_{1},\beta_{2})(\left|\tilde{\eta}\right|^{q-1}+\left|\tilde{\eta}\right|^{\beta_{1}}+(a(\hat{x},\hat{t})+a(\hat{y},\hat{t}))\left|\tilde{\eta}\right|^{\beta_{2}}).
\end{align*}

Combining the estimates of $T_{1}$ to $T_{4}$ and using $(\hat{t}-t_0)>-1,$  we have for some
number $C\geq1$, depending only on $N,p,q,\left\Vert a\right\Vert _{L^{\infty}},\left\Vert Da\right\Vert _{L^{\infty}},C_{f},\beta_{1}$
and $\beta_{2}$, that
\begin{align*}
-K & \leq CL\left|\tilde{\eta}\right|^{p-2}\varphi^{\prime\prime}(\left|z\right|)+C(a(\hat{x},\hat{t})+a(\hat{y},\hat{t}))L\left|\tilde{\eta}\right|^{q-2}\varphi^{\prime\prime}(\left|z\right|)+C\left|\tilde{\eta}\right|^{q-2}L\varphi^{\prime}(\left|z\right|)\\
 & \ \ \ +C\left|\tilde{\eta}\right|^{p-3}\sqrt{K}\omega^{1/2}(\left|z\right|)L\frac{\varphi^{\prime}(\left|z\right|)}{\left|z\right|}+C\left|\tilde{\eta}\right|^{q-3}\sqrt{K}\omega^{1/2}(\left|z\right|)L\frac{\varphi^{\prime}(\left|z\right|)}{\left|z\right|}\\
 & \ \ \ +CK(\left|\tilde{\eta}\right|^{p-2}+\left|\tilde{\eta}\right|^{q-2})+C(\left|\tilde{\eta}\right|^{q-1}+\left|\tilde{\eta}\right|^{\beta_{1}}+(a(\hat{x},\hat{t})+a(\hat{y},\hat{t}))\left|\tilde{\eta}\right|^{\beta_{2}}).
\end{align*}
Using that $\left|\tilde{\eta}\right|=L\varphi^{\prime}(\left|z\right|)$,
and rearranging the terms, we obtain
\begin{align*}
-K\leq\,\, & C\left|\tilde{\eta}\right|^{p-2}\Big(L\varphi^{\prime\prime}(\left|z\right|)+\left|\tilde{\eta}\right|^{q-p+1}+\sqrt{K}(1+\left|\tilde{\eta}\right|^{q-p})\Big(\frac{\omega^{1/2}(\left|z\right|)}{\left|z\right|}+\sqrt{K}\Big)+\left|\tilde{\eta}\right|^{2+\beta_{1}-p}\Big)\\
 & +C(a(\hat{x},\hat{t})+a(\hat{y},\hat{t}))\left|\tilde{\eta}\right|^{q-2}\left(L\varphi^{\prime\prime}(\left|z\right|)+\left|\tilde{\eta}\right|^{2+\beta_{2}-q}\right)
\end{align*}
Since $\varphi^{\prime}>c_{\varphi}$, assuming $L$ to be large enough
depending on $c_{\varphi}$, we can ensure that $\left|\tilde{\eta}\right|\geq1$.
This way, we have that $1+\left|\tilde{\eta}\right|^{q-p}\leq2\left|\tilde{\eta}\right|^{q-p}.$ Also, we note that $K\leq 8\omega(2).$
Finally, we can also estimate the various powers of $\left|\tilde{\eta}\right|$
by $\left|\tilde{\eta}\right|^{\gamma}$, where $\gamma$ is the largest
of these powers. Hence the desired estimate (\ref{eq:lemma main est})
follows.
\end{proof}
Next, we prove the local H\"{o}lder continuity of $u\in \texttt{S}(Q_{1}^-)$. In
the borderline case $p=q+1$ this result is only qualitative. The qualitative result also holds if $\beta_{1}=p$ or $\beta_{2}=q$.
\begin{lem}[H\"{o}lder continuity for class $\texttt{S}$]
\label{lem:holder for S} Suppose that $1\leq p\leq q \leq p+1$,
$\beta_{1}\in[1,p],$ $\beta_{2}\in[1,q]$. Let $u\in \texttt{S}(Q_{1}^-)$ be
bounded and suppose that there is an increasing modulus $\omega:[0,\infty)\rightarrow[0,\infty)$,
$\lim_{s\rightarrow0}\omega(s)=0$, such that 
\[
\left|u(x,t)-u(y,t)\right|\leq\omega(\left|x-y\right|)\quad\text{for all }(x,t),(y,t)\in Q_{1}.
\]
 Then for any $\alpha\in(0,1)$, there exists $C_{H}>0$ such that
\[
\left|u(x_{0},t_{0})-u(y_{0},t_{0})\right|\leq C_{H}\left|x_{0}-y_{0}\right|^{\alpha}\quad\text{for all }(x_{0},t_{0}),(y_{0},t_{0})\in Q_{1/2}^-.
\]
If $\gamma<2$, where $\gamma$ is as in (\ref{eq:gamma}),
then the constant $C_{H}$ depends only on $N$, $p$, $q$, $||a||_{L^\infty(Q_1)}$, $||Da||_{L^\infty(Q_1)}$, $C_f$, $\beta_1$, $\beta_2$ and $\osc_{Q_{1}}u$.
If $\gamma=2$, then the constant $C_{H}$ depends also on $\omega$.
\end{lem}

\begin{proof}
Following the notation in Lemma \ref{lem:Ishii-Lions lemma}, we set
$K:=8\osc_{Q_{1}}u$, fix points $(x_{0},t_{0}),(y_{0},t_{0})\in Q_{1/2}^-$,
and consider the auxiliary function
\[
\Psi(x,y,t):=u(x,t)-u(y,t)-C_{H}\varphi(\left|x-y\right|)-\frac{K}{2}\left|x-x_{0}\right|^{2}-\frac{K}{2}\left|y-y_{0}\right|^{2}-\frac{K}{2}\left|t-t_{0}\right|^{2}
\]
for some large $C_{H}>0$ to be selected later. We define $\varphi:[0,2]\rightarrow[0,\infty)$
by $\varphi(s):=s^{\alpha}$, where $\alpha\in(0,1)$. We have
\[
\varphi^{\prime\prime}(s)=-\alpha\left|\alpha-1\right|s^{\alpha-2}<0<\alpha2^{\alpha-1}<\varphi^{\prime}(s)\quad\text{and}\quad\left|\varphi^{\prime\prime}(s)\right|\leq\frac{\varphi^{\prime}(s)}{s}
\]
which means that the condition (\ref{eq:Ishii-Lions lemma cnd}) holds.
It now suffices to show that $\Psi$ is non-positive, so suppose on
the contrary that $\Psi$ has a positive maximum $(\hat{x},\hat{y},\hat{t})\in\overline{B}_{1}\times\overline{B}_{1}\times[-1,1]$. Now applying
Lemma \ref{lem:Ishii-Lions lemma}, we have the inequality
\begin{align}
-K & \leq C(C_H\varphi^{\prime})^{p-2}\Big(C_H\varphi^{\prime\prime}+(C_H\varphi^{\prime}){}^{\gamma}+\sqrt{K}(C_H\varphi^{\prime})^{\gamma-1}\frac{\omega^{1/2}(\left|z\right|)}{\left|z\right|}\Big)\nonumber \\
 & \ \ \ +C(a(\hat{x},\hat{t})+a(\hat{y},\hat{t}))(C_H\varphi^{\prime})^{q-2}(C_H\varphi^{\prime\prime}+(C_H\varphi^{\prime})^{\gamma}),\label{h=0000F6lder main}
\end{align}
where $\gamma$ is given by (\ref{eq:gamma}). By \eqref{eq:lemma derivative est} we also have the estimate
$\alpha\left|z\right|^{\alpha-1}\leq\omega(\left|z\right|)/(C_H\left|z\right|)$
so that
\begin{equation}
\left|z\right|\leq\left(\frac{\omega(\left|z\right|)}{\alpha C_H}\right)^{1/\alpha}.\label{eq:h=0000F6lder z est}
\end{equation}
Using (\ref{eq:h=0000F6lder z est}), that $(\gamma-1)(\alpha-1)+1>0$
and $(\gamma-2)(\alpha-1)\geq0$, we estimate
\begin{align}
 & C_H\varphi^{\prime\prime}+(C_H\varphi^{\prime})^{\gamma}+\sqrt{K}(C_H\varphi^{\prime})^{\gamma-1}\frac{\omega^{1/2}(\left|z\right|)}{\left|z\right|}\nonumber \\
 & =-\alpha\left|\alpha-1\right|C_H\left|z\right|^{\alpha-2}+\alpha^{\gamma}C_H^{\gamma}\left|z\right|^{\gamma(\alpha-1)}+\sqrt{K}\alpha^{\gamma-1}C_H^{\gamma-1}\left|z\right|^{(\gamma-1)(\alpha-1)-1}\omega^{1/2}(\left|z\right|)\nonumber \\
 & =C_H\left|z\right|^{\alpha-2}\Big(-\alpha\left|\alpha-1\right|+\alpha^{\gamma}C_H^{\gamma}\left|z\right|^{(\gamma-1)(\alpha-1)+1}+\sqrt{K}\alpha^{\gamma-1}C_H^{\gamma-2}\left|z\right|^{(\gamma-2)(\alpha-1)}\omega^{1/2}(\left|z\right|)\Big)\nonumber \\
 & \leq C_H\left|z\right|^{\alpha-2}\Big(-\alpha\left|\alpha-1\right|+C_0(K,\alpha,\gamma)\Big(\omega(\left|z\right|)^{\frac{(\gamma-1)(\alpha-1)+1}{\alpha}}C_H^{\gamma-1-\frac{(\gamma-1)(\alpha-1)+1}{\alpha}}+C_H^{\gamma-2}\omega^{1/2}(\left|z\right|)\Big)\Big)\nonumber \\
 & =C_H\left|z\right|^{\alpha-2}(-\alpha\left|\alpha-1\right|+C_0(K,\alpha,\gamma)\Big(\omega(\left|z\right|)^{\frac{(\gamma-1)(\alpha-1)+1}{\alpha}}C_H^{\frac{\gamma-2}{\alpha}}+\omega^{1/2}(\left|z\right|)C_H^{\gamma-2}\Big)\Big).\label{eq:h=0000F6lder mid}
\end{align}
If $\gamma<2$ then the power of $C_H$ at the right-hand side of (\ref{eq:h=0000F6lder mid})
is negative. In this case, we estimate $\omega(\left|z\right|)\leq K$
and observe that for large enough $C_H$, for instance $C_H\geq \max\{C, C_0, K, \alpha, \gamma\},$ (\ref{eq:h=0000F6lder mid}) implies that
\begin{equation}
C_H\varphi^{\prime\prime}+(C_H\varphi^{\prime})^{\gamma}+\sqrt{K}(C_H\varphi^{\prime})^{\gamma-1}\frac{\omega^{1/2}(\left|z\right|)}{\left|z\right|}\leq-\frac{\alpha\left|\alpha-1\right|}{2}\left|z\right|^{\alpha-2}C_H.\label{eq:h=0000F6lder second}
\end{equation}
If $\gamma=2$, then the power of $C_H$ in (\ref{eq:h=0000F6lder mid})
is zero. In this case, we still obtain (\ref{eq:h=0000F6lder second})
by estimating $\omega(\left|z\right|)$ using (\ref{eq:h=0000F6lder z est}),
and taking large enough $C_H$ as above. However, in this case, $C_H$ will also
depend on $\omega$.

Now, we combine (\ref{eq:h=0000F6lder second}) with (\ref{h=0000F6lder main}).
In particular, the term with $a$ is negative and can be discarded.
So we obtain
\begin{align*}
-K & \leq-\frac{\alpha\left|\alpha-1\right|C}{2}(C_H\varphi^{\prime})^{p-2}C_H\left|z\right|^{\alpha-2}\\
 & =-C\left|z\right|^{\alpha-2+(p-2)(\alpha-1)}C_H^{p-1}.
\end{align*}
Since $\alpha-2+(p-2)(\alpha-1)<0$ and $\left|z\right|\leq2$, a
contradiction for this choice of $C_H.$ This shows that $\Psi$ is non-positive and hence
\begin{align*}
u(x_0, t_0)-u(y_0, t_0)\leq C_H\varphi(|x_0-y_0|)= C_H|x_0-y_0|^{\alpha}.
\end{align*} 
The conclusion follows since $x_0, y_0, t_0$ is arbitrary.
\end{proof}
Next, we show that if a function $u\in \texttt{S}(Q_{1}^-)$ is H\"{o}lder continuous,
then it is also Lipschitz continuous. This result is valid for values of $\gamma > 1$ (even for some $\gamma >2$), where $\gamma$ is as in (\ref{eq:gamma}). The range of $\gamma$ is given by (\ref{eq:lipschitz gamma alpha ineq}) and depends on the \textit{a priori} known H{\"o}lder exponent of $u$. We state the lemma in this generality only for curiosity; in our applications of the lemma $\gamma \leq 2$.

\begin{lem}[Lipschitz continuity for class $\texttt{S}$]\label{lem:Lipschitz for S}
 Suppose that $1< p\leq q<\infty$, $\beta_{1}\in[0,\infty),$
$\beta_{2}\in[0,\infty)$. Let $u\in \texttt{S}(Q_{1}^-)$ be bounded. Suppose
that there is $C_{H}>0$ and $\alpha\in(0,1)$ such that
\begin{equation}
\left|u(x,t)-u(y,t)\right|\leq C_{H}\left|x-y\right|^{\alpha}\quad\text{for all }(x,t),(y,t)\in Q_{1}.\label{eq:lipschitz holder assumption}
\end{equation}
Suppose moreover that 
\begin{equation}
\frac{\alpha}{2}+(\gamma-1)(1-\alpha)<1,\label{eq:lipschitz gamma alpha ineq}
\end{equation}
where $\gamma$ is as in (\ref{eq:gamma}). Then there exists $L>0$
such that
\[
\left|u(x_{0},t_{0})-u(y_{0},t_{0})\right|\leq L\left|x_{0}-y_{0}\right|\quad\text{for all }(x_{0},t_{0}),(y_{0},t_{0})\in Q_{1/2}^-.
\]
The constant $L$ depends only on $N$, $p$, $q$, $||a||_{L^\infty (Q_1)}$, $||Da||_{L^\infty (Q_1)}$, $C_f$, $\beta_1$, $\beta_2$, $\osc_{Q_{1}}u$, $\alpha$ and
$C_{H}$.
\end{lem}

\begin{proof}
Following the notation in Lemma \ref{lem:Ishii-Lions lemma}, we set
$K:=8\osc_{Q_{1}}u$ and fix points $(x_{0},t_{0}),(y_{0},t_{0})\in Q_{1/2}^-$.
Consider the auxiliary function
\[
\Psi(x,y,t):=u(x,t)-u(y,t)-L\varphi(\left|x-y\right|)-\frac{K}{2}\left|x-x_{0}\right|^{2}-\frac{K}{2}\left|y-y_{0}\right|^{2}-\frac{K}{2}\left|t-t_{0}\right|^{2}
\]
for some large $L>0$ to be selected later. Here $\varphi:[0,2]\rightarrow[0,\infty)$
is defined by 
\[
\varphi(s)=s-\kappa s^{\beta},\quad\text{where}\quad\kappa:=\beta^{-1}2^{-\beta-1},
\]
and
\[
\beta:=\min\Big(\frac{\alpha}{2}+1,\frac{1+\big(2-\big(\frac{\alpha}{2}+(\gamma-1)(1-\alpha)\big)\big)}{2}\Big)
\]
While the particular choice of $\beta$ is not important, this definition
together with the assumption (\ref{eq:lipschitz gamma alpha ineq})
ensures that $\beta\in(1,2)$ as well as that
\begin{equation}
\beta<\frac{\alpha}{2}+1\quad\text{and}\quad\beta+\frac{\alpha}{2}+(\gamma-1)(1-\alpha)<2.\label{eq:lipschitz beta cnd}
\end{equation}
By (\ref{eq:lemma main est}) in Lemma \ref{lem:Ishii-Lions lemma}
we have
\begin{align}
-K & \leq C(L\varphi^{\prime})^{p-2}\Big(L\varphi^{\prime\prime}+(L\varphi^{\prime}){}^{\gamma}+\sqrt{K}(L\varphi^{\prime})^{\gamma-1}\frac{\omega^{1/2}(\left|z\right|)}{\left|z\right|}\Big)\nonumber \\
 & \ \ \ +C(a(\hat{x},\hat{t})+a(\hat{y},\hat{t}))(L\varphi^{\prime})^{q-2}\left(L\varphi^{\prime\prime}+(L\varphi^{\prime})^{\gamma}\right).\label{eq:lipschitz first}
\end{align}
Moreover, by (\ref{eq:lemma derivative est}) we have the estimate
\begin{equation}
\varphi^{\prime}(\left|z\right|)\leq\frac{\omega(\left|z\right|)}{L\left|z\right|}\leq\frac{C_{H}\left|z\right|^{\alpha}}{L\left|z\right|}=\frac{C_{H}}{L}\left|z\right|^{\alpha-1},\label{eq:lipschitz varphi deriv est}
\end{equation}
where we used (\ref{eq:lipschitz holder assumption}). On the other
hand, by definition of $\varphi$ 
\[
\varphi^{\prime}(s)=1-\kappa\beta s^{\beta-1}=1-2^{-\beta-1}s^{\beta-1}\in(3/4,1],
\]
 so that we from (\ref{eq:lipschitz varphi deriv est}) we obtain
\begin{equation}
\left|z\right|\leq\left(\frac{4C_{H}}{3L}\right)^{\frac{1}{1-\alpha}}.\label{eq:lipschitz z est}
\end{equation}
Now using $\varphi^{\prime}(|z|)\leq 1$ and $\alpha/2+1-\beta>0$ which holds by the first inequality in \eqref{eq:lipschitz beta cnd}, we estimate 
\begin{align}
 & L\varphi^{\prime\prime}+(L\varphi^{\prime})^{\gamma}+\sqrt{K}(L\varphi^{\prime})^{\gamma-1}\frac{\omega^{1/2}(\left|z\right|)}{\left|z\right|}\nonumber \\
 & \leq-\kappa\beta(\beta-1)\left|z\right|^{\beta-2}L+L^{\gamma}+\sqrt{K}L^{\gamma-1}C_{H}^{1/2}\left|z\right|^{\alpha/2-1}\nonumber \\
 & =\left|z\right|^{\beta-2}L\Big(-\kappa(\beta-1)+L^{\gamma-1}\left|z\right|^{2-\beta}+\sqrt{K}C_{H}^{1/2}L^{\gamma-2}\left|z\right|^{\alpha/2+1-\beta}\Big)\nonumber \\
 & =\left|z\right|^{\beta-2}L\Big(-\kappa(\beta-1)+C_0(K,C_{H},\beta,\gamma)\Big(L^{\gamma-1-\frac{2-\beta}{1-\alpha}}+L^{\gamma-2-\frac{\alpha/2+1-\beta}{1-\alpha}}\Big)\Big)\nonumber \\
 & \leq\left|z\right|^{\beta-2}L\Big(-\kappa(\beta-1)+C_0(K,C_{H},\beta,\gamma)L^{\gamma-2-\frac{\alpha/2+1-\beta}{1-\alpha}}\Big),\label{eq:Lipschitz mid}
\end{align}
where in the last inequality we used that $L\geq1$, as well as that
\begin{align*}
\gamma-2-\frac{\alpha/2+1-\beta}{1-\alpha} & =\gamma-1-\left(\frac{\alpha/2+1-\beta+1-\alpha}{1-\alpha}\right)\\
 & =\gamma-1-\left(\frac{2-\beta-\alpha/2}{1-\alpha}\right)\\
 & \geq\gamma-1-\frac{2-\beta}{1-\alpha}.
\end{align*}
Observe that by the second inequality in (\ref{eq:lipschitz beta cnd})
we have
\begin{align*}
\gamma-2-\frac{\alpha/2+1-\beta}{1-\alpha} & =-\frac{2-\beta-\alpha/2-(\gamma-1)(1-\alpha)}{1-\alpha}\\
 & =\frac{-2+\beta+\alpha/2+(\gamma-1)(1-\alpha)}{1-\alpha}<0.
\end{align*}
Therefore, the power of $L$ at the RHS of (\ref{eq:Lipschitz mid})
is negative. It follows that by taking large enough $L$, depending
on $K,$ $C_{H}$, $\alpha$, $\kappa$, $\beta$, and $\gamma$,
that
\[
-L\varphi^{\prime\prime}+(L\varphi^{\prime})^{\gamma}+\sqrt{K}(L\varphi^{\prime})^{\gamma-1}\frac{\omega^{1/2}(\left|z\right|)}{\left|z\right|}\leq-\frac{\kappa(\beta-1)}{2}L\left|z\right|^{\beta-2}.
\]
The above display means in particular that therm in (\ref{eq:lipschitz first})
with $a$ is negative and can be discarded. Hence
\begin{align*}
-K & \leq-C\frac{\kappa(\beta-1)}{2}(\varphi^{\prime})^{p-2}L^{p-1}\left|z\right|^{\beta-2}.
\end{align*}
Since $\varphi^{\prime}\in[3/4,1]$, $\beta<2$ and $\left|z\right|\leq2$,
this implies a contradiction for large enough $L$.
\end{proof}
\begin{rem}
We conclude this section with some remarks on the role of the exponents $p$, $q$, $\beta_1$, and $\beta_2$ used throughout the analysis. First, observe that Lemma~\ref{lem:Ishii-Lions lemma} does not rely on the gap condition $q - p \leq 1$, nor on specific bounds for $\beta_1$ and $\beta_2$.  The gap condition $q - p \leq 1$ and the assumptions $\beta_1 \leq p$, $\beta_2 \leq q$ become relevant in the next step, namely Lemma~\ref{lem:holder for S}, where we establish the H\"older continuity of $u$. Finally, the Lipschitz continuity result in Lemma~\ref{lem:Lipschitz for S} is stated in its full generality. It shows that, given a H\"older continuous solution, the regularity can be improved to Lipschitz continuity if the \textit{a-priori} known H\"older exponent of the solution is large enough depending on how large $\beta_1$, $\beta_2$ and the gap $q-p$ are (as dictated by the condition \eqref{eq:lipschitz gamma alpha ineq}).
\end{rem}
\section{H{\"o}lder continuity in time and completion of Theorem \ref{Lip thm}}\label{sec: time holder}

In this section, we complete the proof of Theorem \ref{Lip thm}. The result follows by combining the spatial Lipschitz estimate established in the previous section with the following two lemmas, where we establish H{\"o}lder continuity in time. Their proof is based on comparison against suitable barriers, and so the assumption \eqref{assumption increasing} appears again.

\begin{lem}\label{lem:time holder}
Suppose that $1 < p \leq q \leq p + 1$. Let $a\geq 0$ be continuous in time, Lipschitz continuous in space, and suppose that it satisfies \eqref{assumption increasing} in $Q_1$. Suppose that $u$ is a continuous and bounded weak solution to \eqref{eq:p-para f} in $Q_1$, where $f$ satisfies the growth condition \eqref{eq:gcnd}. Moreover, assume that $u$ is $L$-Lipschitz continuous in space variable. Then
\[
    |u(0,t) - u(0,s)| \leq C|t-s|^{\frac{p}{p+q}} \quad\text{for all } t,s\in (-1,0),
\]
where the constant $C$ depends only on $N$, $p$, $q$, $||a||_{L^\infty (Q_1)}$, $||Da||_{L^\infty (Q_1)}$, the constants in the growth condition (\ref{eq:gcnd}), $\osc_{B_1 \times (-1,0)} u$ and $L$.
\end{lem}
\begin{proof}
\textbf{Step 1:} Let $t\in[-1,0)$
and $s_{0}\in(t_{0},0]$. We borrow the barrier function from \cite{IJS19}. We define
\[
\varphi(x,t):=u(0,t_{0})+A+\Theta(t-t_{0})+K\left|x\right|^{\beta},
\]
 where $\beta:=p/(p-1)$, $\Theta\geq0$ is chosen later,
\[
A:=(s_{0}-t_{0})^{\frac{p}{p+q}}\quad\text{and}\quad K:=C_{0}A^{1-\beta}:=2(\osc u +1)(L+1)^{\beta}\beta A^{1-\beta}.
\]
We also set $\rho:=A^{\frac{\beta-1}{\beta}}\leq1$. Then, for $x\in\partial B_{\rho}$, we have that
\[
K\left|x\right|^{\beta}=C_{0}A^{1-\beta}\left|x\right|^{\beta}\geq2(\osc u)A^{1-\beta}\left|x\right|^{\beta}=2\osc u.
\]
Hence, $u\leq\varphi$ on $[t_{0},0]\times\partial B_{\rho}$. We
also note that
\begin{align}
\rho= & \left(\frac{1}{A^{1-\beta}}\right)^{\frac{1}{\beta}}=\left(\frac{C_{0}}{K}\right)^{\frac{1}{\beta}}=C_{0}^{\frac{1}{\beta}}K^{-\frac{1}{\beta}}.\label{eq:rho in terms of K}
\end{align}
On the other hand, since $u$ is $L$-Lipschitz in space variable,
for all $x\in B_{1}$, we have that 
\begin{align}
\varphi(x,t_{0})-u(x,t_{0}) & =u(0,t_{0})-u(x,t_{0})+A+ \Theta (t-t_0)+K\left|x\right|^{\beta}\nonumber \\
 & \geq-L\left|x\right|+A+K\left|x\right|^{\beta}=:f(\left|x\right|),\label{eq:holder 1}
\end{align}
where we denoted $f(r):=A-Lr+Kr^{\beta}$. Observe that
\[
f^{\prime}(r)=-L+K\beta r^{\beta-1}=0\iff r=\left(\frac{L}{K\beta}\right)^{\frac{1}{\beta-1}}=:r_{0}
\]
so that
\begin{align*}
\min_{r\in[0,\infty)}f(r)=f(r_{0}) & =A-L\left(\frac{L}{K\beta}\right)^{\frac{1}{\beta-1}}+K\left(\frac{L}{K\beta}\right)^{\frac{\beta}{\beta-1}}\\
 & =A+K^{-\frac{1}{\beta-1}}L^{\frac{\beta}{\beta-1}}\left(-\left(\frac{1}{\beta}\right)^{\frac{1}{\beta-1}}+\left(\frac{1}{\beta}\right)^{\frac{\beta}{\beta-1}}\right)\\
 & \geq A-\left(\frac{1}{K}\frac{L^{\beta}}{\beta}\right)^{\frac{1}{\beta-1}}\geq0,
\end{align*}
where in the last inequality we used the definition of $K$. Therefore,
by (\ref{eq:holder 1}) we have $u\leq\varphi$ on $B_{1}\times\left\{ t_{0}\right\} $.
So far we have shown that for any $\Theta\geq0$
\begin{equation}
u\leq\varphi\quad\text{on }\partial_{\mathcal{P}}(B_{\rho}\times[t_{0},0]).\label{eq:holder 7}
\end{equation}

\textbf{Step 2:} Next, we select $\Theta$ so large that $\varphi$
will be a weak supersolution to the double phase equation in $B_{\rho}\times(t_{0},0)$.
To this end, let $\phi\in C_{0}^{\infty}(B_{\rho}\times(t_{0},0))$
be non-negative, and consider a point $(x,t)\in\supp\phi,$ $x\not=0$.
Since $x\not=0$, we have 
\begin{align}
\left|\Delta_{p}\varphi(x,t)\right| & \leq (p-1)\left|D\varphi(x,t)\right|^{p-2}\left\Vert D^{2}\varphi(x,t)\right\Vert \nonumber \\
 &=(p-1)\left|K\beta\left|x\right|^{\beta-1}\right|^{p-2}\left\Vert K\beta\left|x\right|^{\beta-2}I+K\beta(\beta-2)x\otimes x\left|x\right|^{\beta-4}\right\Vert \nonumber \\
 & \le C(N,p)K^{p-1}\left|x\right|^{(\beta-1)(p-1)-1}\nonumber \\
 & =C(N,p)K^{p-1},\label{eq:timeholder plap est}
\end{align}
where we used that $K\geq1$ and that $(\beta-1)(p-1)-1=0$. On the
other hand, denoting by $a_{j}$ the standard mollification of $a$,
and using (\ref{eq:rho in terms of K}) as well as that $\left|x\right|\leq\rho$,
we have also
\begin{align}
a_{j}(x,t)\left|\Delta_{q}\varphi(x,t)\right| & \leq C(N,p,q,\left\Vert a\right\Vert _{L^{\infty}})K^{q-1}\left|x\right|^{(\beta-1)(q-1)-1}\nonumber \\
 & \leq C(N,p,q,\left\Vert a\right\Vert _{L^{\infty}})K^{q-1}\big(C_{0}^{\frac{1}{\beta}}K^{-\frac{1}{\beta}}\big)^{(\beta-1)(q-1)-1}\nonumber \\
 & =C(N,p,q,\left\Vert a\right\Vert _{L^{\infty}},\osc u,L)K^{\frac{q}{\beta}}\label{eq:timeholder qlap est}
\end{align}
for all large $j$. Note that, here we used $(\beta-1)(q-1)\geq 1.$ Further, we have similarly
\begin{align}
\left|D\varphi(x,t)\right|^{q-2}D\varphi(x,t)\cdot Da_{j}(x,t) & \leq\beta^{q-1}\left\Vert Da\right\Vert _{L^{\infty}}K^{q-1}\left|x\right|^{(\beta-1)(q-1)}\nonumber \\
 & \leq C(p,q,\left\Vert Da\right\Vert _{L^{\infty}},\osc u,L)K^{\frac{q}{\beta}}.\label{eq:timeholder firstorder est}
\end{align}
and using also growth condition (\ref{eq:gcnd})
\begin{align}
\left|f(x,t,D\varphi)\right| & \leq C_{f}\left(1+\left|D\varphi(x,t)\right|^{\beta_{1}}+a(x,t)\left|D\varphi(x,t)\right|^{\beta_{2}}\right)\nonumber \\
 & =C_{f}\left(1+\beta^{\beta_{1}}K^{\beta_{1}}\left|x\right|^{(\beta-1)\beta_{1}}+a(x,t)\beta^{\beta_{2}}K^{\beta_{2}}\left|x\right|^{(\beta-1)\beta_{2}}\right)\nonumber \\
 & =C_{f}\left(1+\beta^{\beta_{1}}C_{0}^{\frac{\beta_{1}}{\beta}}K^{\beta_{1}\big(1-\frac{\beta-1}{\beta}\big)}+a(x,t)\beta^{\beta_{2}}C_{0}^{\frac{\beta_{2}}{\beta}}K^{\beta_{2}(1-\frac{\beta-1}{\beta})}\right)\nonumber \\
 & \leq C(p,q,\left\Vert a\right\Vert _{L^{\infty}},C_{f},\left\Vert u\right\Vert _{L^{\infty}},L)K^{\frac{q}{\beta}},\label{eq:timeholder f est}
\end{align}
where in the last inequality we used that $\beta_{1},\beta_{2} < q$ and $K\geq 1$.
Combining the estimates (\ref{eq:timeholder plap est})-(\ref{eq:timeholder f est}), we see that in the support of $\phi$ and outside of the origin, for
all large $j$, it holds that
\begin{align*}
\div(\left|D\varphi\right|^{p-2}D\varphi+a_{j}\left|D\varphi\right|^{q-2}D\varphi)+f(x,t,D\varphi) & \leq C_{1}K^{\frac{q}{\beta}},
\end{align*}
for some $C_{1}$ that depends only on data, oscillation of $u$ and
$L$. Taking 
\[
\Theta:=C_{1}K^{\frac{q}{\beta}},
\]
we obtain that
\[
\partial_{t}\varphi-\div(\left|D\varphi\right|^{p-2}D\varphi+a_{j}\left|D\varphi\right|^{q-2}D\varphi)\geq0\quad\text{in }\supp\phi\cap\big((B_{1}\setminus\left\{ 0\right\} )\times(t_{0},0)\big)
\]
for all large $j$. Multiplying this by $\phi$ and applying the divergence
theorem, we see that for all $r>0$
\begin{align*}
0 & \leq\int_{t_{0}}^{0}\int_{B_{\rho}\setminus B_{r}}\phi\partial_{t}\varphi-\phi\div(\left|D\varphi\right|^{p-2}D\varphi+a_{j}\left|D\varphi\right|^{q-2}D\varphi)\d x\d t\\
 & =\int_{t_{0}}^{0}\int_{B_{\rho}\setminus B_{r}}-\varphi\partial_{t}\phi+\left|D\varphi\right|^{p-2}D\varphi\cdot D\phi+a_{j}\left|D\varphi\right|^{q-2}D\varphi\cdot D\phi\d x\d t\\
 & \ \ \ +\int_{t_{0}}^{0}\int_{\partial B_{r}}(\left|D\varphi\right|^{p-2}+a_{j}\left|D\varphi\right|^{q-2})D\varphi\cdot\frac{x}{\left|x\right|}\d S(x)\d t.
\end{align*}
Letting $r\rightarrow0$ and $j\rightarrow \infty$, we see that $\varphi$
is a weak supersolution.

\textbf{Step 3:} From the comparison principle it follows that $u\leq\varphi$
in $B_{\rho}\times[t_{0},0]$. In particular, we have by definition of
$\Theta$ and $K$ 
\begin{align*}
u(0,s_{0})-u(0,t_{0})\leq\varphi(0,s_{0})-u(0,t_{0}) & =A+C_{1}K^{\frac{q}{\beta}}(s_{0}-t_{0})\\
 & =A+C_{1}C_{0}^{\frac{q}{\beta}}A^{-\frac{q}{p}}(s_{0}-t_{0}).
\end{align*}
Recalling that $A=(s_{0}-t_{0})^{\frac{p}{p+q}}$, we obtain
\begin{align*}
u(0,s_{0})-u(0,t_{0}) & \leq(s_{0}-t_{0})^{\frac{p}{p+q}}+C(s_{0}-t_{0})^{1-\frac{p}{p+q}\cdot\frac{q}{p}}\\
 & =(s_{0}-t_{0})^{\frac{p}{p+q}}+C(s_{0}-t_{0})^{\frac{p}{p+q}}.
\end{align*}
A lower bound follows similarly, but using the subsolution $u(0,t_{0})-A-\Theta(t-t_{0})-K\left|x\right|^{\beta}$.
Since $s_{0}\in(t_{0},0]$ was arbitrary, it follows that $u$ is
$\frac{p}{p+q}$-H{\"o}lder continuous in time at $(0,t_{0})$.
\end{proof}

In the degenerate range we can improve the time H{\"o}lder exponent to
$1/2$. The reason is that if $p\geq2$ and $u$ is Lipschitz continuous,
then the term $\left|Du\right|^{p-2}$ that appears in $p$-Laplacian
is bounded, and we can choose $\beta=2$ in the barrier function.
\begin{lem}\label{lem:time holder 1/2}
Suppose the same assumptions as in Lemma \ref{lem:time holder}, with
the exponents $p$ and $q$ restricted to the degenerate range $2\leq p\leq q\leq p+1$.
Then we have 
\[
\left|u(0,t)-u(0,s)\right|\leq C\left|t-s\right|^{1/2}\quad\text{for all }t,s\in(-1,0),
\]
where the constant $C$ depends only on $N$, $p$, $q$, $\left\Vert a\right\Vert _{L^{\infty}(B_{1}\times(-1,0))}$,
$\left\Vert Da\right\Vert _{L^{\infty}(B_{1}\times(-1,0))}$, the
constants in the growth condition, $\osc_{B_{1}\times(-1,0)}$ and
$L$.
\end{lem}
\begin{proof}
\textbf{Step 1:} This step is same as in the proof of Lemma \ref{lem:time holder}, except
we consider the barrier
\[
\varphi(x,t):=u(0,t_{0})+A+\Theta(t-t_{0})+K\left|x\right|^{2},
\]
with $\Theta\geq0$ chosen later, 
\[
A:=(s_{0}-t_{0})^{\frac{1}{2}}\quad\text{and}\quad K:=C_{0}A^{-1}:=4(\osc u+1)(L+1)^{2}A^{-1}.
\]
One verifies as in Step 1 of proof of Lemma \ref{lem:time holder} that for any $\Theta\geq0$,
we have
\[
u\leq\varphi\quad\text{on }\partial_{\mathcal{P}}(B_{1}\times[t_{0},0]).
\]

\textbf{Step 2:} We use a comparison argument to show that, provided
$\Theta$ is chosen to be large enough, we have
\begin{equation}
u\leq\varphi\quad\text{in }B_{1}\times(t_{0},0].\label{eq:time2 step2 claim}
\end{equation}
To this end, consider the function
\[
\varphi_{h}(x,t):=\varphi(x,t)-\frac{h}{t}\quad\text{for }h>0.
\]
Then suppose on the contrary that $\sup_{B_{1}\times(t_{0},0)}u-\varphi_{h}>0$.
Since $u-\varphi_{h}\leq0$ on the boundary of $B_{1}\times(t_{0},0)$,
by continuity the function $u-\varphi_{h}$ has a maximum at some
$(\hat{x},\hat{t})\in B_{1}\times(t_{0},0)$. Since $\varphi_{h}$
is smooth and $u$ is a viscosity solution by Theorem \ref{thm:weak is visc}, we have by definition
of viscosity solutions that
\begin{align}
\liminf_{\underset{x\not=\hat x,(x,t)\not\in\Gamma}{(x,t)\rightarrow(\hat{x},\hat{t}),}} & \Big(\partial_{t}\varphi_{h}(x,t)-\div\big(\left|D\varphi_{h}(x,t)\right|^{p-2}D\varphi(x,t)+a(x,t)\left|D\varphi_{h}(x,t)\right|^{q-2}D\varphi_{h}(x,t)\big)\nonumber \\
 & -f(x,t,D\varphi_{h}(x,t)\Big)\leq0.\label{eq:time2 liminf}
\end{align}
Denote $\eta:=D\varphi_{h}(\hat{x},\hat{t})$. If $\eta\not=0$, then
the above implies
\begin{align*}
\partial_{t}\varphi_{h}(\hat{x},\hat{t}) & \leq\left|\eta\right|^{p-2}\Big(\tr D^{2}\varphi_{h}(\hat{x},\hat{t})+(p-2)\frac{\eta^{\prime}D^{2}\varphi_{h}(\hat{x},\hat{t})\eta}{\left|\eta\right|^{2}}\Big)\\
 & \ \ \ +a(\hat{x},\hat{t})\left|\eta\right|^{q-2}\Big(\tr D^{2}\varphi_{h}(\hat{x},\hat{t})+(q-2)\frac{\eta^{\prime}D^{2}\varphi_{h}(\hat{x},\hat{t})\eta}{\left|\eta\right|^{2}}\Big)\\
 & \ \ \ +\left\Vert Da\right\Vert _{L^{\infty}}\left|\eta\right|^{q-1}+f(x,t,\eta)\\
 & \leq2\left|\eta\right|^{p-2}(N+p-2)K+2\left\Vert a\right\Vert _{L^{\infty}}\left|\eta\right|^{q-2}(N+q-2)K\\
 & \ \ \ +\left\Vert Da\right\Vert _{L^{\infty}}\left|\eta\right|^{q-1}+C_{f}(1+\left|\eta\right|^{\beta_{1}}+\left|\eta\right|^{\beta_{2}}),
\end{align*}
where in the last estimate we used that $D^{2}\varphi_{h}(\hat{x},\hat{t})=2KI$
and the growth condition (\ref{eq:gcnd}). Observe that since $u-\varphi_{h}$ has
a maximum at $(\hat{x},\hat{t})$ and $u$ is $L$-Lipschitz in space,
we must have $\left|D\varphi_{h}(\hat{x},\hat{t})\right|\leq L$.
Consequently, the above display yields (using also that $\beta_{1},\beta_{2}\leq q$
and that $K\geq1$)
\begin{align}
\partial_{t}\varphi_{h}(\hat{x},\hat{t}) & \leq C(N,q,C_{f})(\left\Vert a\right\Vert _{L^{\infty}}+\left\Vert Da\right\Vert _{L^{\infty}})(L+1)^{q}K\nonumber \\
 & =C(N,q,\left\Vert a\right\Vert _{L^{\infty}},\left\Vert Da\right\Vert _{L^{\infty}},L)K.\label{eq:time2 timederiv est}
\end{align}
In the case $\eta=0$, the inequality (\ref{eq:time2 liminf}) still
implies (\ref{eq:time2 timederiv est}), as since $2\leq p\leq q$,
the divergence terms either vanish or reduce to Laplacians. Now, since
$\Theta\leq\partial_{t}\varphi_{h}(\hat{x},\hat{t})$, we have
\[
\Theta\leq C(N,q,\left\Vert a\right\Vert _{L^{\infty}},\left\Vert Da\right\Vert _{L^{\infty}},L)K,
\]
so that taking $\Theta=C_{1}K$, where $C_{1}$ depends only on $N$,
$q,$ $\left\Vert a\right\Vert _{L^{\infty}}$, $\left\Vert Da\right\Vert _{L^{\infty}}$
and $L$, we have a contradiction. Since $h>0$ was arbitrary, the
inequality (\ref{eq:time2 step2 claim}) follows.

\textbf{Step 3:} This step is again similar to Step 3 in the proof
of Lemma \ref{lem:time holder}. From (\ref{eq:time2 step2 claim}) we have in particular
\begin{align*}
u(0,s_{0})-u(0,t_{0})\leq\varphi(0,s_{0})-u(0,t_{0}) & =(s_{0}-t_{0})^{1/2}+C_{1}K(s_{0}-t_{0})\\
 & =(s_{0}-t_{0})^{1/2}+C_{1}C_{0}(s_{0}-t_{0})^{1/2}.
\end{align*}
A lower bound can be proven similarly, and so the claim follows.
\end{proof}

We are now ready to finish the proof of the main regularity result.

\begin{proof}[Proof of Theorem \ref{Lip thm}]
Let $u$ be a continuous weak solution to (\ref{eq:p-para f}) in $B_{3}\times(-3,0)$. In the case $q=p+1$, we suppose additionally that $u$ is uniformly continuous in space variable, i.e.\ there is a modulus $\omega$ such that
\[
\left|u(x,t)-u(y,t)\right|\leq\omega(\left|x-y\right|)\quad\text{for all }(x,t),(y,t)\in B_{3}\times(-3,0).
\]
Let $\varepsilon \in (0,1)$. By Lemma \ref{cws and S}, we have $u\in \texttt{S}(Q_{1}^-(x,t))$ for
all $(x,t)\in B_{1}\times(-1,-\varepsilon)$. Applying Lemma \ref{lem:holder for S}
in each of these cylinders, we find that
\[
\left|u(x,t)-u(y,t)\right|\leq C_{H}\left|x-y\right|^{\alpha}\quad\text{for all }(x,t),(y,t)\in B_{1}\times(-1-\varepsilon,-\varepsilon],
\]
where, if $p < q +1 $, the constant $C_{H}$ depends only on data ($N$, $p$, $q$, $||a||_{L^\infty (B_{3}\times(-3,0))}$, $||Da||_{L^\infty (B_{3}\times(-3,0))}$ and the constants in the growth condition (\ref{eq:gcnd})) and $\osc_{B_{3}\times(-3,0)}u$. If $p=q+1$, $C_{H}$ also depends on $\omega$. Next, we apply the Lemma \ref{lem:Lipschitz for S} in $B_{1}\times(-1-\varepsilon,-\varepsilon]$ to find that
\[
\left|u(x,t)-u(y,t)\right|\leq L\left|x-y\right| 
\quad\text{for all }(x,t),(y,t)\in B_{1/2}\times(-1/2-\varepsilon,-\varepsilon],
\]where $L$ depends only on data, $C_{H}$ and $\osc_{B_{3}\times(-3,0)}$.
Since $\varepsilon \in (0,1)$ was arbitrary, lemmas \ref{lem:time holder} and \ref{lem:time holder 1/2} together
with standard scaling and translation arguments finishes the proof.
\end{proof}
\section{Remarks on multi-phase equation}\label{sec: multiphase} 
In this section, we show that the Ishii-Lions method also works for parabolic multi-phase equation. We consider the parabolic version of \eqref{multi-phase} and only provide a sketch of the proof when $m=2:$
\begin{align}\label{multiphase eq}
    \partial_t u-\div\left(|Du|^{p-2}Du+a(z)|Du|^{q-2}Du+b(z)|Du|^{s-2}Du\right)=f(z, Du).
\end{align}
We also assume similar growth conditions on $f:$
\begin{align*}
    |f(z, Du)|\leq C_f(1+|Du|^{\beta_1}+a(z)|Du|^{\beta_2}+b(z)|Du|^{\beta_3})
\end{align*}
where $1\leq \beta_1<p, 1\leq\beta_2<q, \,\, \text{and}\,\, 1\leq\beta_3<s.$ Weak solutions to \eqref{multiphase eq} are defined analogously to Definition \ref{def:weak solutions} with the appropriate spaces corresponding to the modulus $$(z,\xi) \mapsto |\xi|^p+a(z)|\xi|^q+b(z)|\xi|^s .$$
Moreover, we point out that the comparison principle holds in a similar manner as in the double phase case. To prove the time approximation used in the proof of the comparison principle, we need both modulating coefficients to satisfy the \textit{almost increasing condition} in a suitable way. That is to say, we require that
\begin{align}\label{eq:multiphase ab cnd}
    a(x, t)\leq C \max(a(x, s),b(x,s)) \quad \text{and} \quad b(x, t)\leq C b(x, s)
\end{align}
for all $(x,t), (x,s) \in Q_1$ with $t\leq s$, where $C$ is some non-negative constant.
\begin{thm} \label{Lip thm 2}
Suppose that $1 < p \leq q \leq s \leq p + 1$ and assume that the coefficients $a$ and $b$
in the equation \eqref{multiphase eq} are nonnegative, continuous in time, Lipschitz continuous in space and satisfy \eqref{eq:multiphase ab cnd} in $Q_1$. If $1\leq p\leq q\leq s<p+1,$ and  $u$ is a weak solution to \eqref{multiphase eq} in $Q_1$, then there exists a constant $C>0$ such that
\begin{align}\label{Lip estimate sec 6}
|u(x_1,t_1) - u(x_2,t_2)| \leq C\left(|x_1-x_2|+|t_1-t_2|^{\frac{p}{p+s}}\right)
\end{align}
$\text{for all}\quad (x_1, t_1), (x_2, t_2) \in B_{1/2}\times(-1/2, 0), $ where $\alpha=\frac{p}{p+s}$ if $1<p<2$ and $\alpha=\frac{1}{2}$ if $p\geq 2.$
The constant $C$ depends only on $N$, $p$, $q$, $s,$ $\Vert a \Vert_{L^\infty}$, $\Vert b \Vert_{L^\infty}$, $\Vert Da \Vert_{L^\infty}$, $\Vert Db \Vert_{L^\infty}$, the constants appearing in the growth condition \eqref{eq:gcnd}, and $\osc_{B_1 \times (-1, 0)} u$. If $s=p+1,$ the constant $C$ will additionally depend on the optimal modulus of continuity of $u$ in space in $B_1 \times (-1,0)$.

\end{thm}
\begin{proof}
Since the proof is almost similar to the double phase case, we only point out the steps where it is different.

\noindent \textit{Step 1: Lipschitz regularity in space.} In this step, we sketch the proof of Lipschitz continuity of solution $u$ to the equation \eqref{multiphase eq}. Similar to the double phase case, for all $(x,t)\in Q_{1}$, $\eta\in\mathbb{R}^{N}$ and $X\in S(N)$,
we denote
\begin{align*}
F((x,t),\eta,X):= & \:\left|\eta\right|^{p-2}\left(\tr X+(p-2)\frac{\eta^{\prime}X\eta}{\left|\eta\right|^{2}}\right)\nonumber \\
 & \:+a(x,t)\left|\eta\right|^{q-2}\left(\tr X+(q-2)\frac{\eta^{\prime}X\eta}{\left|\eta\right|^{2}}\right)+b(x,t)\left|\eta\right|^{s-2}\left(\tr X+(s-2)\frac{\eta^{\prime}X\eta}{\left|\eta\right|^{2}}\right)
\end{align*}
and
\begin{align*}
g((x,t),\eta) & :=\left\Vert Da\right\Vert _{L^{\infty}}\left|\eta\right|^{q-1}+\left\Vert Db\right\Vert _{L^{\infty}}\left|\eta\right|^{s-1}+C_f(1+\left|\eta\right|^{\beta_{1}}+a(x,t)\left|\eta\right|^{\beta_{2}}+b(x,t)|\eta|^{\beta_3}),
\end{align*}
where $C_f\geq0$ and $0\leq\beta_{1}\leq p$, $0\leq\beta_{2}\leq q,$ and $0\leq\beta_{3}\leq s$. From the proof of Lemma \ref{lem:Ishii-Lions lemma}, we have the following.
\begin{align}\label{estimate T}
K(\hat{t}-t_{0})&\leq F((\hat{x},\hat{t}),\eta_{1},X+KI)-F((\hat{y},\hat{t}),\eta_{2},Y-KI)+g((\hat{x},\hat{t}),\left|\eta_{1}\right|)+g((\hat{y},\hat{t}),\left|\eta_{2}\right|)\nonumber\\
&=F((\hat{x},\hat{t}),\eta_{1},X)-F((\hat{y},\hat{t}),\eta_{2},Y)\nonumber\\
&+F((\hat{x},\hat{t}),\eta_{1},X+KI)-F((\hat{x},\hat{t}),\eta_{1},X)-F((\hat{y},\hat{t}),\eta_{2},Y-KI)+F((\hat{y},\hat{t}),\eta_{2},Y)\nonumber\\
&+g((\hat{x},\hat{t}),\left|\eta_{1}\right|)+g((\hat{y},\hat{t}),\left|\eta_{2}\right|)\nonumber\\
&=F((\hat{x},\hat{t}),\tilde{\eta},X)-F((\hat{y},\hat{t}),\tilde{\eta},Y)\nonumber\\
&+F((\hat{x},\hat{t}),\eta_{1},X)-F((\hat{x},\hat{t}),\tilde{\eta},X)+F((\hat{y},\hat{t}),\tilde{\eta},Y)-F((\hat{y},\hat{t}),\eta_{2},Y)\nonumber\\
&+F((\hat{x},\hat{t}),\eta_{1},X+KI)-F((\hat{x},\hat{t}),\eta_{1},X)-F((\hat{y},\hat{t}),\eta_{2},Y-KI)+F((\hat{y},\hat{t}),\eta_{2},Y)\nonumber\\
&+g((\hat{x},\hat{t}),\left|\eta_{1}\right|)+g((\hat{y},\hat{t}),\left|\eta_{2}\right|)\nonumber\\
& =:T_{1}+T_{2}+T_{3}+T_{4}.
\end{align}
\textbf{Estimate of $T_1$:} By definition, we have
\begin{align*}
T_1&=F((\hat{x}, \hat{t}), \tilde{\eta}, X)-F((\hat{y}, \hat{t}), \tilde{\eta}, Y)\\
&=|\tilde{\eta}|^{p-2}\tr (X-Y)+|\tilde{\eta}|^{p-2}(p-2)\frac{\tilde{\eta}^{\prime}(X-Y)\tilde{\eta}}{|\tilde{\eta}|^2}\\
&+\underbrace{\left|\tilde{\eta}\right|^{q-2}\left(\tr(a(\hat{x},\hat{t})X-a(\hat{y},\hat{t})Y)+(q-2)\frac{\tilde{\eta}^{\prime}(a(\hat{x},\hat{t})X-a(\hat{y},\hat{t})Y)\tilde{\eta}}{\left|\tilde{\eta}\right|^{2}}\right)}_{T_{11}}\\
&+\underbrace{\left|\tilde{\eta}\right|^{s-2}\left(\tr(b(\hat{x},\hat{t})X-b(\hat{y},\hat{t})Y)+(s-2)\frac{\tilde{\eta}^{\prime}(b(\hat{x},\hat{t})X-b(\hat{y},\hat{t})Y)\tilde{\eta}}{\left|\tilde{\eta}\right|^{2}}\right)}_{T_{12}}
\end{align*}
Following the estimates for the double phase case, we similarly have
\begin{align*}
    T_{11}\leq C(q)(a(\hat{x},\hat{t})+a(\hat{y},\hat{t}))\left|\tilde{\eta}\right|^{q-2}L\varphi^{\prime\prime}(\left|z\right|)+C(N,\left\Vert Da\right\Vert _{L^{\infty}})L\left|\tilde{\eta}\right|^{q-2}\varphi^{\prime}(\left|z\right|),\\
    T_{12}\leq C(s)(b(\hat{x},\hat{t})+b(\hat{y},\hat{t}))\left|\tilde{\eta}\right|^{s-2}L\varphi^{\prime\prime}(\left|z\right|)+C(N,\left\Vert Db\right\Vert _{L^{\infty}})L\left|\tilde{\eta}\right|^{s-2}\varphi^{\prime}(\left|z\right|)
\end{align*}
Combining these estimates, we get
\begin{align*}
    T_1 \leq &C(p)\left|\tilde{\eta}\right|^{p-2}L\varphi^{\prime\prime}(\left|z\right|)+C(q)(a(\hat{x},\hat{t})+a(\hat{y},\hat{t}))\left|\tilde{\eta}\right|^{q-2}L\varphi^{\prime\prime}(\left|z\right|)\\
 &+C(N,\left\Vert Da\right\Vert _{L^{\infty}})L\left|\tilde{\eta}\right|^{q-2}\varphi^{\prime}(|z|)+C(s)(b(\hat{x},\hat{t})+b(\hat{y},\hat{t}))\left|\tilde{\eta}\right|^{s-2}L\varphi^{\prime\prime}(\left|z\right|)\\
 &+C(N,\left\Vert Db\right\Vert _{L^{\infty}})L\left|\tilde{\eta}\right|^{s-2}\varphi^{\prime}(\left|z\right|)(\left|z\right|).
\end{align*}

\noindent {\bf Estimate of $T_2:$} Estimate of $T_2$ in this case is exactly same as the double phase case with a extra term and we have
\begin{align*}
    T_2 \leq\left( C(N,p)\left|\tilde{\eta}\right|^{p-3}+C(N,q,\left\Vert a\right\Vert _{L^{\infty}})\left|\tilde{\eta}\right|^{q-3}+C(N,s,\left\Vert b\right\Vert _{L^{\infty}})\left|\tilde{\eta}\right|^{s-3}\right)\sqrt{K}\omega^{1/2}(\left|z\right|)L\frac{\varphi^{\prime}(\left|z\right|)}{\left|z\right|}.
\end{align*}
\noindent {\bf Estimate of $T_3:$} This estimate is also same and we get
\begin{align*}
    T_3 \leq C(N,p,q, s, \left\Vert a\right\Vert _{L^{\infty}}, \left\Vert b\right\Vert _{L^{\infty}})K(\left|\tilde{\eta}\right|^{p-2}+\left|\tilde{\eta}\right|^{q-2}+ \left|\tilde{\eta}\right|^{s-2})
\end{align*}

\noindent {\bf Estimate of $T_4:$} For $T_4,$ we estimate
\begin{align*}
    T_{4} & =g((\hat{x},\hat{t}),\left|\eta_{1}\right|)+g((\hat{y},\hat{t}),\left|\eta_{2}\right|)\\
 & =\left\Vert Da\right\Vert _{L^{\infty}}\left|\eta_{1}\right|^{q-1}+\left\Vert Db\right\Vert_{L^{\infty}}|\eta_1|^{s-1}+C_{f}(1+\left|\eta_{1}\right|^{\beta_{1}}+a(\hat{x},\hat{t})\left|\eta_{1}\right|^{\beta_{2}}+b(\hat{x}, \hat{t})|\eta_1|^{\beta_3})\\
 & \ \ \ +\left\Vert Da\right\Vert _{L^{\infty}}\left|\eta_{2}\right|^{q-1}+ \left\Vert Db\right\Vert _{L^{\infty}}\left|\eta_{2}\right|^{s-1}+C_{f}(1+\left|\eta_{2}\right|^{\beta_{1}}+a(\hat{x},\hat{t})\left|\eta_{2}\right|^{\beta_{2}}+b(\tilde{x}, \tilde{t})|\eta_2|^{\beta_3})\\
 & \le C(q,\left\Vert Da\right\Vert _{L^{\infty}},\left\Vert Db\right\Vert _{L^{\infty}},C_{f},\beta_{1},\beta_{2}, \beta_3)(\left|\tilde{\eta}\right|^{q-1}+|\tilde{\eta}|^{s-1}+\left|\tilde{\eta}\right|^{\beta_{1}}\\
 &+(a(\hat{x},\hat{t})+a(\hat{y},\hat{t}))\left|\tilde{\eta}\right|^{\beta_{2}}+(b(\hat{x},\hat{t})+b(\hat{y},\hat{t}))\left|\tilde{\eta}\right|^{\beta_{3}}).
\end{align*}
Plugging the above estimates in \eqref{estimate T}, we have
\begin{align*}
-K \leq\,\, &C\left|\tilde{\eta}\right|^{p-2}L\varphi^{\prime\prime}(\left|z\right|)+C(a(\hat{x},\hat{t})+a(\hat{y},\hat{t}))\left|\tilde{\eta}\right|^{q-2}L\varphi^{\prime\prime}(\left|z\right|)\\
 &+CL\left|\tilde{\eta}\right|^{q-2}\varphi^{\prime}(|z|)+C(b(\hat{x},\hat{t})+b(\hat{y},\hat{t}))\left|\tilde{\eta}\right|^{s-2}L\varphi^{\prime\prime}(\left|z\right|)\\
 &+CL\left|\tilde{\eta}\right|^{s-2}\varphi^{\prime}(\left|z\right|)+ C \left(|\tilde{\eta}|^{p-3}+|\tilde{\eta}|^{q-3}+|\tilde{\eta}|^{s-3}\right)\sqrt{K}\omega^{1/2}(\left|z\right|)L\frac{\varphi^{\prime}(\left|z\right|)}{\left|z\right|}\\
 &+CK(\left|\tilde{\eta}\right|^{p-2}+\left|\tilde{\eta}\right|^{q-2}+ \left|\tilde{\eta}\right|^{s-2})+C(\left|\tilde{\eta}\right|^{q-1}+|\tilde{\eta}|^{s-1}+\left|\tilde{\eta}\right|^{\beta_{1}}\\
 &+(a(\hat{x},\hat{t})+a(\hat{y},\hat{t}))\left|\tilde{\eta}\right|^{\beta_{2}}+(b(\hat{x},\hat{t})+b(\hat{y},\hat{t}))\left|\tilde{\eta}\right|^{\beta_{3}}).
\end{align*}
Using that $\left|\tilde{\eta}\right|=L\varphi^{\prime}(\left|z\right|)$,
and rearranging the terms, we obtain
\begin{align*}
-K\leq\,\, &C\left|\tilde{\eta}\right|^{p-2}L\varphi^{\prime\prime}(\left|z\right|)+C (|\tilde{\eta}|^{q-1}+|\tilde{\eta}|^{s-1})+C(a(\hat{x},\hat{t})+a(\hat{y},\hat{t}))\left|\tilde{\eta}\right|^{q-2}L\varphi^{\prime\prime}(\left|z\right|)\\ &+C(b(\hat{x},\hat{t})+b(\hat{y},\hat{t}))\left|\tilde{\eta}\right|^{s-2}L\varphi^{\prime\prime}(\left|z\right|)\\
 &+ C \left(|\tilde{\eta}|^{p-2}+|\tilde{\eta}|^{q-2}+|\tilde{\eta}|^{s-2}\right)\sqrt{K}\frac{\omega^{1/2}(\left|z\right|)}{|z|}\\
 &+CK(\left|\tilde{\eta}\right|^{p-2}+\left|\tilde{\eta}\right|^{q-2}+ \left|\tilde{\eta}\right|^{s-2})+C(\left|\tilde{\eta}\right|^{q-1}+|\tilde{\eta}|^{s-1}+\left|\tilde{\eta}\right|^{\beta_{1}}\\
 &+(a(\hat{x},\hat{t})+a(\hat{y},\hat{t}))\left|\tilde{\eta}\right|^{\beta_{2}}+(b(\hat{x},\hat{t})+b(\hat{y},\hat{t}))\left|\tilde{\eta}\right|^{\beta_{3}})\\
& \leq C |\tilde{\eta}|^{p-2}\Big(L \varphi^{\prime \prime}(|z|)+|\tilde{\eta}|^{q-p+1}+|\tilde{\eta}|^{s-p+1}\\
&+\sqrt{K}(1+|\tilde{\eta}|^{q-p}+|\tilde{\eta}|^{s-p})\Big(\frac{\omega^{1/2}(|z|)}{|z|}+\sqrt{K}\Big)+|\tilde{\eta}|^{\beta_1+2-p}\Big)
\end{align*}
\begin{align*}
&+C (a(\hat{x}, \hat{t})+a(\hat{y}, \hat{t}))|\tilde{\eta}|^{q-2}\Big(L\varphi^{\prime\prime}(|z|)+|\tilde{\eta}|^{\beta_2-q+2}\Big)\\
&+C (b(\hat{x}, \hat{t})+b(\hat{y}, \hat{t}))|\tilde{\eta}|^{s-2}\Big(L\varphi^{\prime\prime}(|z|)+|\tilde{\eta}|^{\beta_3-s+2}\Big)
\end{align*}
Now taking $\gamma-1 = \max \left\{s-p,\beta_1-p+1, \beta_2-q+1, \beta_3-s+1\right\}$ and using $|\tilde{\eta}|\geq 1, $ we arrive at the estimate
\begin{equation}\label{eq:multi estimate}
\begin{aligned}
-K\leq & C(L\varphi^{\prime})^{p-2}\Big(L\varphi^{\prime\prime}+(L\varphi^{\prime})^{\gamma}+\sqrt{K}(L\varphi^{\prime})^{\gamma-1}\frac{\omega^{1/2}(\left|z\right|)}{\left|z\right|}\Big) \\
& +C(a(\hat{x},\hat{t})+a(\hat{y},\hat{t}))(L\varphi^{\prime})^{q-2}(L\varphi^{\prime\prime}+(L\varphi^{\prime})^{\gamma})\\
&+C(b(\hat{x},\hat{t})+b(\hat{y},\hat{t}))(L\varphi^{\prime})^{s-2}(L\varphi^{\prime\prime}+(L\varphi^{\prime})^{\gamma})  
\end{aligned}
\end{equation}
which is similar to (\ref{eq:lemma main est}). Now observe that, once we fixed the $\gamma, $ the H\"{o}lder estimate in Lemma \ref{lem:holder for S} and the final Lipschitz estimate in Lemma \ref{lem:Lipschitz for S} do not depend on the structure of the equation and solely depend on \eqref{eq:multi estimate}. Thus, the Lipschitz regularity for \eqref{multiphase eq} follows analogously.

\noindent \textit{Step 2: H\"{o}lder regularity in time.} The proof of time regularity follows from barrier arguments analogical to the case of the double phase equation.
\end{proof}
\section{Viscosity solutions are weak solutions}\label{sec: visc to weak}

This section is dedicated to the proof of Theorem \ref{final thm: visc is weak_intro}, i.e., that viscosity (super) solutions to \eqref{eq:p-para f} are weak (super) solutions. We use the method introduced in \cite{newequivalence} based on inf-convolution. The main difficulty is caused by the presence of the modulating coefficient $a$. Indeed, for equations that are independent of $(x,t)$, it is usually straightforward to show that the inf-convolution retains the supersolution property. This is because if a test function $\varphi$ touches the inf-convolution of a lower semicontinuous function from below, $\varphi$ also touches the original function from below at some nearby point $(x_{\varepsilon},t_{\varepsilon})$. However, since (\ref{eq:p-para f}) depends on $(x,t)$, this shift generates an error term. In particular, it can occur that $(x,t)$ and $(x_{\varepsilon},t_{\varepsilon})$ are on different sides of $\Gamma$. This makes the error term difficult to control, as $Da$ may be jump discontinuous near $\Gamma$. Our assumption on $a$ ensures that this scenario can occur only in a set of small measure, which gives sufficient control over the error. Similar issues occur near the set where $a$ vanishes. For these reasons, we need the assumption \eqref{eq:visc is weak acnd} to prove Theorem \ref{final thm: visc is weak_intro}.

In Lemma \ref{lem:infconv is super} we show that inf-convolution of a viscosity supersolution is a supersolution to a perturbed equation. In the proof, we apply the Theorem of sums, which causes jet-closures to appear. The following lemma verifies that the definition of viscosity supersolution behaves as expected in this context.

\begin{lem}\label{lem:jet closures}
Let $u$ be a viscosity supersolution in $\Xi$. Suppose that $a:\Xi \rightarrow [0, \infty)$ is continuous and Lipschitz continuous in space. Suppose that $(\theta,\eta,X)\in\overline{\mathcal{P}}^{2,-}u(x,t)$,
where $\eta\not=0$ and $(x,t)\in\Xi$. Then we have
\begin{align*}
 & 0 \leq  \theta-\left|\eta\right|^{p-2}(\tr X+(p-2)\frac{\eta^{\prime}X\eta}{\left|\eta\right|^{2}})-a(x,t)\left|\eta\right|^{q-2}\tr(X+(q-2)\frac{\eta^{\prime}X\eta}{\left|\eta\right|^{2}})-f(x, t, \eta)\\
 & \ \ \ \ \ -\liminf_{(y,s)\rightarrow(x,t),(y,s)\not\in\Gamma}\left|\eta\right|^{q-2}\eta\cdot Da(y,s).
\end{align*}
\end{lem}

\begin{proof}
By definition of jet-closure, we have $(x_{i},t_{i},\theta_{i},\eta_{i},X_{i})\rightarrow(x,t,\theta,\eta,X)$
as $i\rightarrow\infty$, where $(\theta_{i},\eta_{i},X_{i})\in\mathcal{P}^{2,-}u(x_{i},t_{i})$.
Since $\eta\not=0$, we can assume that $\left|\eta_{i}\right|\not=0$
for all $i$. For each $i$, there is $\varphi_{i}\in C^{2}$ such
that it touches $u$ from below at $(x_{i},t_{i})$ and 
\[
(\partial_{t}\varphi_{i}(x_{i},t_{i}),D\varphi_{i}(x_{i},t_{i}),D^{2}\varphi_{i}(x_{i},t_{i}))=(\theta_{i},\eta_{i},X_{i}).
\]
Since $u$ is a viscosity supersolution and $\varphi_{i}$ touches
$u$ from below at $(x_{i},t_{i})$, we have
\begin{align*}
0\leq\limsup_{(y,s)\rightarrow(x_{i},t_{i}),(y,s)\not\in\Gamma,y\not=x_{i}}\Big( & \partial_{t}\varphi_{i}(y,s)-\Delta_{p}\varphi_{i}(y,s)-a(y,s)\Delta_{q}\varphi_{i}(y,s)-f(y, s, D\varphi(y,s ))\\
 & -\left|D\varphi_{i}(y,s)\right|^{q-2}D\varphi_{i}(y,s)\cdot Da(y,s)\Big).
\end{align*}
From continuity, the definition of limit supremum and the fact $\left|\Gamma\right|=0$,
it follows that for each $i$ there is $(y_{i},s_{i})\in(B_{i^{-1}}(x_{i},t_{i})\times(t_{i}-i^{-1},t_{i}+i^{-1})\setminus\Gamma$
such that
\begin{equation}
\left|\theta_{i}-\partial_{t}\varphi_{i}(y_{i},s_{i})\right|+\left|\eta_{i}-D\varphi_{i}(y_{i},s_{i})\right|+\left\Vert X_{i}-D^{2}\varphi_{i}(y_{i},s_{i})\right\Vert <i^{-1},\label{eq:jet closures blabla}
\end{equation}
and by the continuity of $f,$ we also have
\begin{align*}
    |f(x_i, t_i, D\varphi(x_i, t_i))-f(y_i, s_i, D\varphi(y_i, s_i))|< \frac{1}{i},
\end{align*}
$D\varphi_{i}(y_{i},s_{i})\not=0$, and
\begin{align}
-i^{-1} & \leq\partial_{t}\varphi_{i}(y_{i},s_{i})-\Delta_{p}\varphi_{i}(y_{i},s_{i})-a(y_{i},s_{i})\Delta_{q}\varphi_{i}(y_{i},s_{i})-f(y_i, s_i, D\varphi(y_i, s_i))\nonumber \\
 & \ \ \ -\left|D\varphi_{i}(y_{i},s_{i})\right|^{q-2}D\varphi_{i}(y_{i},s_{i})\cdot Da(y_{i},s_{i}).\label{eq:jet closures blabla2}
\end{align}
For brevity, we denote
\begin{align*}
A_{i} & :=\partial_{t}\varphi_{i}(y_{i},s_{i})-\Delta_{p}\varphi_{i}(y_{i},s_{i})-a(y_{i},s_{i})\Delta_{q}\varphi_{i}(y_{i},s_{i})-f(y_i, s_i, D\varphi(y_i, s_i))\quad\text{and}\\
(\tilde{\theta}_{i},\tilde{\eta}_{i},\tilde{X}_{i}) & :=(\partial_{t}\varphi_{i}(y_{i},s_{i}),D\varphi_{i}(y_{i},s_{i}),D^{2}\varphi_{i}(y_{i},s_{i})).
\end{align*}
Then by (\ref{eq:jet closures blabla2}), we have
\begin{equation}
-\left|\tilde{\eta}_{i}\right|^{q-2}\tilde{\eta}_{i}\cdot Da(y_{i},s_{i})\geq-A_{i}-i^{-1}\label{eq:A_i ineq}
\end{equation}
and (\ref{eq:jet closures blabla}) implies that 
\begin{equation}
(\tilde{\theta_{i}},\tilde{\eta}_{i},\tilde{X}_{i})\rightarrow(\theta,\eta,X)\quad\text{as }i\rightarrow\infty.\label{eq:jet closures convergence}
\end{equation}
Since $\left|\eta\right|>0$, it follows from (\ref{eq:jet closures convergence})
that
\begin{align}
A_{i} & \rightarrow\theta-\left|\eta\right|^{p-2}\tr\Big(X+(p-2)\frac{\eta^{\prime}X\eta}{\left|\eta\right|^{2}}\Big)-a(x,t)\left|\eta\right|^{q-2}\tr\Big(X+(q-2)\frac{\eta^{\prime}X\eta}{\left|\eta\right|^{2}}\Big)-f(x, t, \eta)\nonumber \\
 & =:A\quad\text{as }i\rightarrow\infty.\label{eq:jet closures A_i convergence}
\end{align}
Now, by definition of limit supremum, since $(y_{i},s_{i})\rightarrow(x,t)$,
we have
\begin{align*}
 & \limsup_{(y,s)\rightarrow(x,t),(y,s)\not\in\Gamma}-\left|\eta\right|^{q-2}\eta\cdot Da(y,s)\\
 & \geq\limsup_{i\rightarrow\infty}-\left|\eta\right|^{q-2}\eta\cdot Da(y_{i},s_{i})\\
 & =\limsup_{i\rightarrow\infty}\Big(-\left|\tilde{\eta}_{i}\right|^{q-2}\tilde{\eta}_{i}\cdot Da(y_{i},s_{i})+\Big(\left|\tilde{\eta}_{i}\right|^{q-2}\tilde{\eta}_{i}-\left|\eta\right|^{q-2}\eta\Big)\cdot Da(y_{i},s_{i})\Big)\\
 & \geq\limsup_{i\rightarrow\infty}\Big(-A_{i}-i^{-1}-\left|\left|\eta\right|^{q-2}\eta-\left|\tilde{\eta}_{i}\right|^{q-2}\tilde{\eta}_{i}\right|\left\Vert Da\right\Vert _{L^{\infty}(\Xi)}\Big),\\
 & =-A,
\end{align*}
where in the last inequality we used (\ref{eq:A_i ineq}), and last
identity follows from (\ref{eq:jet closures convergence}), (\ref{eq:jet closures A_i convergence}).
This implies the claim by definition of $A$.
\end{proof}

For $\varepsilon>0$, we define the inf-convolution
\begin{equation}
u_{\varepsilon}(x,t):=\inf_{(y,s)\in\Xi}\left\{ u(y,s)+\frac{\left|x-y\right|^{\ell}}{\ell\varepsilon^{\ell -1}}+\frac{\left|t-s\right|^{2}}{2\delta_{\varepsilon}}\right\} ,\label{eq:inf convolution}
\end{equation}
where $\ell > \max(3, p/(p-1))$ and $\delta_{\varepsilon}>0$ is a function of $\varepsilon$ 
(defined via the modulus of continuity of $a$, see \eqref{eq:delta_e} below). We gather some well-known properties of inf-convolution in the following lemma (a proof can be found in the appendix of \cite{LPS25}).
\begin{lem}
\label{lem:inf conv properties}Assume that $u:\Xi\rightarrow\mathbb{R}$
is lower semicontinuous and bounded. Suppose that $\delta_{\varepsilon}\rightarrow0$
as $\varepsilon\rightarrow0$. Then $u_{\varepsilon}$ has the following
properties.
\begin{enumerate}
\item We have $u_{\varepsilon}\leq u$ in $\Xi$ and $u_{\varepsilon}\rightarrow u$
pointwise as $\varepsilon\rightarrow0$.
\item Denote $r(\varepsilon):=\max((\ell\varepsilon^{\ell-1}\osc_{\Xi}u)^{\frac{1}{\ell}},(2\delta_{\varepsilon}\osc_{\Xi}u)^{\frac{1}{2}})$
and set
\[
\Xi_{\varepsilon}:=\left\{ (x,t)\in\Xi:B_{r(\varepsilon)}(x)\times(t-r(\varepsilon),t+r(\varepsilon))\Subset\Xi\right\} .
\]
Then, for any $(x,t)\in\Xi_{\varepsilon}$ there exists $(x_\varepsilon, t_\varepsilon) \in \overline B_{r(\varepsilon)}(x)\times[t-r(\varepsilon),t+r(\varepsilon)]$ such that
\[
u_{\varepsilon}(x,t)=u(x_{\varepsilon},t_{\varepsilon})+\frac{\left|x-x_{\varepsilon}\right|^{\ell}}{\ell\varepsilon^{l-1}}+\frac{\left|t-t_{\varepsilon}\right|^{2}}{2\delta_{\varepsilon}}.
\]
\item The function $u_{\varepsilon}$ is semi-concave in $\Xi_{\varepsilon}$.
In particular, the function $u_{\varepsilon}(x,t)-(C\left|x\right|^{2}+t^{2}/\delta_{\varepsilon})$
is concave in $\Xi_{\varepsilon}$, where $C=(\ell-1)r(\varepsilon)^{\ell-2}/\varepsilon^{\ell-1}$.
\item Suppose that $u_{\varepsilon}$ is twice differentiable in space and
time at $(x,t)\in\Xi_{\varepsilon}$. Then 
\begin{align*}
\partial_{t}u_{\varepsilon}(x,t) & =\frac{t-t_{\varepsilon}}{\delta_{\varepsilon}},\\
Du_{\varepsilon}(x,t) & =(x-x_{\varepsilon})\frac{\left|x-x_{\varepsilon}\right|^{\ell-2}}{\varepsilon^{\ell-1}},\\
D^{2}u_{\varepsilon}(x,t) & \leq(\ell-1)\frac{\left|x-x_{\varepsilon}\right|^{\ell-2}}{\varepsilon^{\ell-1}}I.
\end{align*}
\end{enumerate}
\end{lem}

The next lemma shows that by taking $\delta_{\varepsilon}$ to be
a suitable function of $\varepsilon$ depending on the time-modulus
of continuity of $a$, we can always ensure that $|a(x,t_{\varepsilon})-a(x,t)|<\varepsilon^{q(\ell -1)}$,
where $t_{\varepsilon}$ is as in Lemma \ref{lem:inf conv properties}.
\begin{lem} \label{lem:inf delta lemma} Let $a:\Xi \rightarrow [0, \infty)$ be uniformly continuous.
Suppose that $u:\Xi\rightarrow\mathbb{R}$ is lower semicontinuous
and bounded. Let $u_{\varepsilon}$ be defined by (\ref{eq:inf convolution}),
where 
\begin{equation}
\delta_{\varepsilon}:=\frac{(\omega_{a}^{-1}(\varepsilon^{q(\ell-1)}))^{2}}{2\osc_\Xi u}.\label{eq:delta_e}
\end{equation}
Here $\omega_{a}: [0, \infty) \rightarrow [0, \infty)$ is a bounded, strictly increasing function such that $\omega(s) \rightarrow 0$ as $s\rightarrow 0$, $\omega(0)= 0$, and
\[
    |a(x, t) - a(x, s)| \leq \omega_a(|t-s|) \quad \text{for all } (x, t), (x, s) \in \Xi.
\]
Let $(x_{\varepsilon},t_{\varepsilon})$ be such
that $u_{\varepsilon}(x,t)=u(x_{\varepsilon},t_{\varepsilon})+\frac{\left|x_{\varepsilon}-x\right|^{\ell}}{\ell\varepsilon^{\ell-1}}+\frac{\left|t_{\varepsilon}-t\right|^{2}}{2\delta_{\varepsilon}}$.
Then
\[
\left|t-t_{\varepsilon}\right|\leq\omega_{a}^{-1}(\varepsilon^{q(\ell-1)}).
\]
\end{lem}

\begin{proof}
By the definition of inf-convolution, we have
\[
\frac{\left|t-t_{\varepsilon}\right|^{2}}{2\delta_{\varepsilon}}\le u(x_{\varepsilon},t_{\varepsilon})-u(x_{\varepsilon},t).
\]
Therefore
\begin{align*}
\left|t-t_{\varepsilon}\right| & \leq\sqrt{2\delta_{\varepsilon}\osc_{\Xi}u}=\sqrt{(\omega_{a}^{-1}(\varepsilon^{q(\ell-1)}))^{2}}=\omega_{a}^{-1}(\varepsilon^{q(\ell-1)}).
\end{align*}
\end{proof}

\begin{lem}
\label{lem:nullset convergence} Suppose that $\mathcal{C}\subset\mathbb{R}^{N+1}$
is compact and $\left|\mathcal{C}\right|=0$. Then 
\[
\left|\mathcal{C}_{r}\right|:=\left|\left\{ x:\text{dist}(x,\mathcal{C})\leq r\right\} \right|\rightarrow0\quad\text{as }r\rightarrow0.
\]
\end{lem}
\begin{proof}
Since $\left|\mathcal{C}\right|=0$, there exists an open set $A\supset \mathcal{C}$
such that $\left|A\right|<\varepsilon$. Since $\mathcal{C}$ is compact, we
have $\text{dist}(\mathcal{C},\partial A)>\delta$ for some $\delta>0$. Therefore,
for any $r<\delta$, we have $\mathcal{C}_{r}\subset A$ and thus $\left|\mathcal{C}_{r}\right|<\varepsilon$.
\end{proof}
We can now prove the key lemma of this section. It states that the inf-convolution of a viscosity supersolution is a supersolution to a perturbed equation. The proof is based on the Theorem of sums and its overall structure goes back to \cite{Ishii95}.
\begin{lem}\label{lem:infconv is super}
Let $1<p\leq q\leq p+1$. Suppose that $a:\Xi \rightarrow [0, \infty)$ is Lipschitz in space, uniformly continuous in time, and that \eqref{eq:visc is weak acnd} holds. Let $u$ be a bounded viscosity supersolution to (\ref{eq:p-para f})
in $\Xi$ and $K>0$ be a number. Let $u_\varepsilon$ be the inf-convolution of $u$ as in Lemma \ref{lem:inf delta lemma} Fix a smaller domain $\Xi ^\prime \Subset \Xi$ and suppose that $\varepsilon>0$ is so small that $\Xi^\prime \Subset \Xi_\varepsilon$. Let $(\theta,\eta,X)\in\mathcal{P}^{2,-}u_{\varepsilon}(x,t)$
with $(x,t)\in \Xi_\varepsilon \setminus\Gamma$.
Then, if $\eta \not=0$, we have
\begin{align} \label{eq:inf conv eq}
 & \theta-\left|\eta\right|^{p-2}\big(\tr X+(p-2)\frac{\eta^{\prime}X\eta}{\left|\eta\right|^{2}}\big)-a(x, t)\left|\eta\right|^{q-2}\big(\tr X+(q-2)\frac{\eta^\prime Y 
 \eta)}{|\eta|^2}\big)-\left|\eta\right|^{q-2}\eta \cdot Da(x, t) \nonumber \\ 
 & \geq f_{K,\varepsilon}(x, t, \eta)-(1+\left|\eta\right|^{q-1})E_{\varepsilon}(x,t),
\end{align}
where $f_{K,\varepsilon}$ is given in \eqref{eq:f_K,epsilon} and $E_{\varepsilon}: \Xi^\prime \rightarrow [0, \infty)$ is a bounded, measurable function that converges to zero almost everywhere in $\Xi^{\prime}$. Moreover, we have the estimate
\[
    E_\varepsilon(x,t) \leq C(N, q, \operatorname{diam}\Xi, \left \Vert Da \right \Vert_{L^\infty(\Xi)}) \quad \text{for all } (x,t)\in \Xi^\prime.
\]
Finally, in the case that $(\theta, 0, X) \in \mathcal{P}^{2,-} u_\varepsilon (x,t)$, we have $\theta \geq f_{K,\varepsilon}(x,t,0)$.
\end{lem}

\begin{proof}
Let us consider the case $\eta \not = 0$ first. Since $(x,t)\in\Xi_{\varepsilon}$, by Lemma \ref{lem:inf conv properties},
there is $(x_{\varepsilon},t_{\varepsilon})$ such that 
\begin{equation}\label{eq: inf lemma 21}
u_{\varepsilon}(x,t)=u(x_{\varepsilon},t_{\varepsilon})+\frac{\left|x-x_{\varepsilon}\right|^{\ell}}{\ell\varepsilon^{\ell -1}}+\frac{\left|t-t_{\varepsilon}\right|^{2}}{2\delta_{\varepsilon}}. 
\end{equation}
Since $(\theta,\eta,X)\in\mathcal{P}^{2,-}u_{\varepsilon}(x, t)$, there exists
$\varphi\in C^{2}$ such that it touches $u_{\varepsilon}$ from below
at $(x,t)$ and $(\partial_{t}\varphi(x,t),D\varphi(x,t),D^{2}\varphi(x,t))=(\theta,\eta,X)$. Using the definition of inf-convolution and that $\varphi$ touches $u_{\varepsilon}$ from below at $(x,t)$, we obtain for all $(y,s),(z,\tau)\in\Xi$
that
\begin{align} \label{eq:infblabla}
-u(y,s)+\varphi(z,\tau)-\frac{\left|y-z\right|^{\ell}}{\ell\varepsilon^{\ell-1}}-\frac{\left|s-\tau\right|^{2}}{2\delta_{\varepsilon}} & =-u_{\varepsilon}(z,\tau)+\varphi(z,\tau)\nonumber\\
 & \leq-u_{\varepsilon}(x,t)+\varphi(x,t)\nonumber\\
 & =-u(x_{\varepsilon},t_{\varepsilon})+\varphi(x,t)-\frac{\left|x_{\varepsilon}-x\right|^{\ell}}{\ell\varepsilon^{\ell-1}}-\frac{\left|t_{\varepsilon}-t\right|^{2}}{2\delta_{\varepsilon}}.
\end{align}
In particular, \eqref{eq:infblabla} implies that $(z, \tau) \mapsto \varphi(z, \tau) -\frac{|x_\varepsilon -z|^\ell}{\ell \varepsilon^{\ell -1}}-\frac{|t_\varepsilon -\tau|^2}{2\delta_\varepsilon}$ has a maximum at $(z, \tau) = (x,t)$. Therefore, by direct computation
\begin{equation}
\theta=\frac{t-t_{\varepsilon}}{\delta_{\varepsilon}}\quad\text{and}\quad\eta=\frac{x-x_{\varepsilon}}{\varepsilon^{\ell-1}}\left|x-x_{\varepsilon}\right|^{\ell-2}.\label{eq:infconv is super derivatives}
\end{equation}
On the other hand, if we denote
\[
\Phi(y,s,z,\tau):=\frac{\left|y-z\right|^{\ell}}{\ell\varepsilon^{\ell-1}}+\frac{\left|s-\tau\right|^{2}}{2\delta_{\varepsilon}},
\]
then it follows from \eqref{eq:infblabla} that $(y,s,z,\tau)\mapsto-u(y,s)+\varphi(z,\tau)-\Phi(y,s,z,\tau)$ has a maximum at the point $(y,s,z,\tau)=(x_{\varepsilon},t_{\varepsilon},x,t)$. Invoking the Jensen-Ishii lemma, we conjure symmetric $N\times N$ matrices $Z$ and $Y$ such that
\begin{align*}
(\partial_{s}\Phi(b),D_{y}\Phi(b),Y) & \in\overline{\mathcal{P}}^{2,+}(-u(x_{\varepsilon},t_{\varepsilon})),\\
(-\partial_{\tau}\Phi(b),-D_{z}\Phi(b),Z) & \in\overline{\mathcal{P}}^{2,-}(-\varphi(x,t)).
\end{align*}
Since $\partial_{s}\Phi(b)=\frac{t_{\varepsilon}-t}{\delta_{\varepsilon}}=-\theta=-\partial_{\tau}\Phi(a)$
and $D_{y}\Phi(b)=\frac{x_{\varepsilon}-x}{\varepsilon^{q-1}}\left|x-x_{\varepsilon}\right|^{\ell-2}=-\eta=-D_{z}\Phi(x,x_{\varepsilon})$,
these can be written as
\[
(\theta,\eta,-Y)\in\overline{\mathcal{P}}^{2,-}u(x_{\varepsilon},t_{\varepsilon})\quad\text{and}\quad(\theta,\eta,-Z)\in\overline{\mathcal{P}}^{2,+}\varphi(x,t).
\]
Moreover, we have
\begin{align}
\begin{pmatrix}Y & 0\\
0 & -Z
\end{pmatrix} & \leq D_{(y,z)}^{2}\Phi(x_{\varepsilon},x)+\varepsilon^{\ell-1}(D_{(y,z)}^{2}\Phi(x_{\varepsilon},x))^{2}\nonumber \\
 & =\begin{pmatrix}B & -B\\
-B & B
\end{pmatrix}+\varepsilon^{\ell-1}\begin{pmatrix}B^{2} & -B^{2}\\
-B^{2} & B^{2}
\end{pmatrix},\label{eq:matrix ineq}
\end{align}
where 
\[
B=\frac{1}{\varepsilon^{\ell-1}}\left|x-x_{\varepsilon}\right|^{\ell-4}((q-2)(x_{\varepsilon}-x)\otimes(x_{\varepsilon}-x)+\left|x_{\varepsilon}-x\right|^{2}I).
\]
The matrix inequality \eqref{eq:matrix ineq} implies in particular that 
\begin{equation}\label{eq:XYZ eq}
    Y\leq Z\leq-D^{2}\varphi(x,t)=-X.
\end{equation}

Now, we denote
\begin{align*}
F(z, \tau,\xi,\mathcal{X}) & :=\left|\xi\right|^{p-2}(\tr\mathcal{X}+(p-2)\frac{\xi^{\prime}\mathcal{X}\xi}{\left|\xi\right|^{2}})+a(z,\tau)\left|\xi\right|^{q-2}(\tr\mathcal{X}+(q-2)\frac{\xi^{\prime}\mathcal{X}\xi}{\left|\xi\right|^{2}})\\
 & \phantom{:=} +\liminf_{(y,s)\rightarrow(z,\tau), (y,s) \not \in \Gamma}\left|\xi\right|^{q-2}\xi\cdot Da(y, s)
\end{align*}
for all $(z,\tau)\in \Xi,\xi\in\mathbb{R}^{N}$ and $\mathcal{X}\in S(N)$. Since $(\theta, \eta, -Y) \in \mathcal {\overline{\mathcal P}}^{2,-}u(x_\varepsilon,t_\varepsilon)$, by Lemma \ref{lem:jet closures} we have that
\begin{equation}
    \theta - F(x_\varepsilon, t_\varepsilon, \eta, -Y) \geq f(x_\varepsilon, t_\varepsilon, \eta).
\end{equation}
Therefore, since $u$ is a viscosity supersolution and $Da$ is continuous at
$(x, t)\not\in\Gamma$, we have by degenerate ellipticity of $-F$ (recall $X \leq - Z$ by \eqref{eq:XYZ eq})
\begin{align*}
 & \theta-F(x, t, \eta,X)\\
 & = F(x_{\varepsilon},t_{\varepsilon},\eta,-Y)-F(x,t,\eta,X) + \theta-F(x_{\varepsilon},t_{\varepsilon},\eta,-Y)\\
 & \geq F(x_\varepsilon, t_\varepsilon,\eta,-Y)-F(x, t, \eta,-Z)+f(x_\varepsilon, t_\varepsilon, \eta)\\
 & =\left|\eta\right|^{p-2}(\tr(Z-Y)+(p-2)\frac{\eta^{\prime}(Z-Y)\eta}{\left|\eta\right|^{2}}))\\
 & \phantom{=} +a(x,t)\left|\eta\right|^{q-2}(\tr Z+(q-2)\frac{\eta^{\prime}Z\eta}{\left|\eta\right|^{2}})-a(x_{\varepsilon},t_{\varepsilon})\left|\eta\right|^{q-2}(\tr Y+(q-2)\frac{\eta^{\prime}Y\eta}{\left|\eta\right|^{2}})\\
 & \phantom{=} +\left|\eta\right|^{q-2}\eta\cdot(Da(x,t)-\liminf_{(y,s)\rightarrow(x_{\varepsilon},t_{\varepsilon}), (y,s) \not \in \Gamma } Da(x_{\varepsilon}, t_\varepsilon))\\
 & \phantom{=} + f(x_\varepsilon, t_\varepsilon, \eta) \\
 & \geq a(x,t)\left|\eta\right|^{q-2}(\tr Z+(q-2)\frac{\eta^{\prime}Z\eta}{\left|\eta\right|^{2}})-a(x_{\varepsilon},t_{\varepsilon})\left|\eta\right|^{q-2}(\tr Y+(q-2)\frac{\eta^{\prime}Y\eta}{\left|\eta\right|^{2}})\\
 & \phantom{=} +\left|\eta\right|^{q-2}\eta\cdot(Da(x,t)-\liminf_{(y, s)\rightarrow (x_\varepsilon, t_\varepsilon), (y, s) \not \in \Gamma}Da(x_{\varepsilon},t_{\varepsilon}))\\
 & \phantom{=} + f(x_\varepsilon, t_\varepsilon, \eta)\\
 & =:T_{1}+T_{2} + T_3,
\end{align*}
where in the last inequality we used that $Z-Y \geq 0$ by \eqref{eq:XYZ eq}. We proceed to estimate these terms.

\textbf{Estimate of $T_{1}$:} We multiply the inequality (\ref{eq:matrix ineq})
by the vector $(\sqrt{a(x_{\varepsilon},t_{\varepsilon})}\eta,\sqrt{a(x,t)}\eta)$
to obtain
\begin{align*}
a(x,t)\frac{\eta^{\prime}Z\eta}{\left|\eta\right|^{2}}-a(x_{\varepsilon},t_{\varepsilon})\frac{\eta^{\prime}Y\eta}{\left|\eta\right|^{2}} & \geq-C\frac{\left|x-x_{\varepsilon}\right|^{\ell-2}}{\varepsilon^{\ell-1}}\left(\sqrt{a(x,t)}-\sqrt{a(x_{\varepsilon},t_{\varepsilon})}\right)^{2}.
\end{align*}
Similarly, multiplying (\ref{eq:matrix ineq}) by $(\sqrt{a(x_{\varepsilon},t_{\varepsilon})}e_{i},\sqrt{a(x,t)}e_{i})$,
$i\in\left\{ 1,\ldots,N\right\} $, and summing up the inequalities,
we get
\[
a(x,t)\tr Z-a(x_{\varepsilon},t_{\varepsilon})\tr Y\geq-C(N)\frac{\left|x-x_{\varepsilon}\right|^{\ell-2}}{\varepsilon^{\ell-1}}\left(\sqrt{a(x,t)}-\sqrt{a(x_{\varepsilon},t_{\varepsilon})}\right)^{2}.
\]
Thus,
\begin{align}\label{eq:T_1 first est}
T_{1} & \geq-C(N,q)\left|\eta\right|^{q-2}\frac{\left|x-x_{\varepsilon}\right|^{\ell-2}}{\varepsilon^{\ell-1}}\left(\sqrt{a(x,t)}-\sqrt{a(x_{\varepsilon},t_{\varepsilon})}\right)^{2} \nonumber \\
 & =-C(N,q)\frac{\left|x-x_{\varepsilon}\right|^{(q-1)(\ell-1)-1}}{\varepsilon^{(q-1)(\ell-1)}}\left(\sqrt{a(x,t)}-\sqrt{a(x_{\varepsilon},t_{\varepsilon})}\right)^{2},
\end{align}
where we used the formula of $\eta$ in (\ref{eq:infconv is super derivatives}).
Observe that the right-hand side vanishes if both $a(x,t)$ and $a(x_{\varepsilon},t_{\varepsilon})$
vanish. Therefore, we only need to estimate $T_{1}$ when $a(x,t)+a(x_{\varepsilon},t_{\varepsilon})>0$,
which we now assume in the rest of the estimate of $T_{1}$. By Lemma \ref{lem:infconv is super} we have that $\left|t-t_{\varepsilon}\right|\leq\omega_{a}^{-1}(\varepsilon^{q(\ell-1)})$, where $\omega_a$ is a strictly increasing modulus of continuity of $a$.
Using also that $a$ is Lipschitz in space, we obtain
\begin{align}\label{eq:a shift est}
|a(x, t) - a(x_\varepsilon, t_\varepsilon)|  \leq \omega_a(|t-t_\varepsilon|)+||Da||_{L^\infty}|x-x_\varepsilon| \nonumber 
&\leq \omega_a(\omega_a^{-1}(\varepsilon^{q(\ell -1)}))+||Da||_{L^\infty}|x-x_\varepsilon| \nonumber \\
& =\varepsilon^{q(\ell -1)} + ||Da||_{L^\infty}|x-x_\varepsilon|.
\end{align}
Now we denote 
\[
\mathcal{E}:=\frac{|a(x, t)-a(x_{\varepsilon}, t_{\varepsilon})|}{\left(\sqrt{a(x, t)}+\sqrt{a(x_{\varepsilon}, t_{\varepsilon})}\right)^2}
\]
 and continue the estimate of $T_{1}$ in \eqref{eq:T_1 first est} to obtain
\begin{align}
T_{1} & \geq-C(N,q)\frac{\left|x-x_{\varepsilon}\right|^{(q-1)(\ell-1)-1}}{\varepsilon^{(q-1)(\ell-1)}}\frac{|a(x, t)-a(x_{\varepsilon}, t_{\varepsilon})|^2}{\left(\sqrt{a(x, t)}+\sqrt{a(x_{\varepsilon}, t_{\varepsilon})}\right)^2}\nonumber \\
 & \geq-C(N,q)\frac{\left|x-x_{\varepsilon}\right|^{(q-1)(\ell-1)-1}}{\varepsilon^{(q-1)(\ell-1)}}\left(\varepsilon^{q(\ell-1)}+\left\Vert Da\right\Vert _{L^{\infty}}\left|x-x_{\varepsilon}\right|\right)\frac{|a(x, t)-a(x_{\varepsilon}, t_{\varepsilon})|}{\left(\sqrt{a(x, t)}+\sqrt{a(x_{\varepsilon}, t_{\varepsilon})}\right)^2}\nonumber \\
 & \geq-\left(C(N,q)\left|x-x_{\varepsilon}\right|^{(q-1)(\ell-1)-1}+C(N,q,\left\Vert Da\right\Vert _{L^{\infty}})\frac{\left|x-x_{\varepsilon}\right|^{(q-1)(\ell-1)}}{\varepsilon^{(q-1)(\ell-1)}}\right)\mathcal{E}\nonumber \\
 & \geq-C(N,q,\left\Vert Da\right\Vert _{L^{\infty}})(1+\left|\eta\right|^{q-1})\mathcal{E},\label{eq:T1 est}
\end{align}
where in the last estimate we used that $\left|x-x_{\varepsilon}\right|\leq \operatorname{diam}(\Xi)$,
$(q-1)(\ell-1)-1>0$ (since $\ell>\frac{p}{p-1}>\frac{q}{q-1}$), and the formula of $\eta$ in (\ref{eq:infconv is super derivatives}).
It now remains to estimate $\mathcal{E}$. To this end, denote 
\[
\Gamma_{\varepsilon}^{0}:=\left\{ (x,t)\in\Xi:(\overline B_{r(\varepsilon)}(x)\times[t-r(\varepsilon),t+r(\varepsilon)]\cap\Gamma^{0}\not=\emptyset\right\} ,
\]
where $r(\varepsilon)$ is the function given in Lemma \ref{lem:inf conv properties}
and $\Gamma^{0}=\partial\left\{ (x,t)\in\Xi:a(x,t)>0\right\} \cap\Xi$.
Observe that the assumption $a(x,t)+a(x_{\varepsilon},t_{\varepsilon})>0$
together with continuity of $a$ implies that if $a(x,t)=0$, then
$(x_{\varepsilon},t_{\varepsilon})\in\Gamma_{\varepsilon}^{0}$. We
use this to estimate
\begin{align}
\mathcal{E} & \leq (\chi_{\left\{ a>0\right\} }(x,t)+\chi_{\left\{ a=0\right\} }(x,t))\frac{\left|a(x,t)-a(x_{\varepsilon},t_{\varepsilon})\right|}{a(x,t)+a(x_{\varepsilon},t_{\varepsilon})}\nonumber \\
 & \leq\chi_{\left\{ a>0\right\} }(x,t)\frac{\left|a(x,t)-a(x_{\varepsilon},t_{\varepsilon})\right|}{a(x,t)}+\chi_{\Gamma_{\varepsilon}^{0}}(x,t)\frac{\left|a(x,t)-a(x_{\varepsilon},t_{\varepsilon})\right|}{a(x,t)+a(x_{\varepsilon},t_{\varepsilon})}\nonumber \\
 & \leq\chi_{\left\{ a>0\right\} }(x,t)\max_{(y,s)\in \overline B_{r(\varepsilon)}\times[t-r(\varepsilon),t+r(\varepsilon)]}\frac{\left|a(x,t)-a(y,s)\right|}{a(x,t)}+\chi_{\Gamma_{\varepsilon}^{0}}(x,t).\label{eq:mathcalE est}
\end{align}
The first term on the right-hand side is measurable as it is continuous almost everywhere in $\Xi^\prime$ by uniform continuity of $a$ and the assumption $|\Gamma^0| = 0$. Moreover, it is bounded and converges to zero almost everywhere
in $\Xi^{\prime}$. Furthermore, since $\left|\Gamma^{0}\right|=0$, it follows from Lemma \ref{lem:nullset convergence}
that $\chi_{\Gamma_{\varepsilon}^{0}}\rightarrow0$ almost everywhere.
Thus by (\ref{eq:T1 est}) and (\ref{eq:mathcalE est}) we have
\[
T_{1}\geq-E_{1}(x,t)(1+\left|\eta\right|^{q-1})
\]
for some bounded and measurable $E_{1}:\Xi^{\prime}\rightarrow[0,\infty)$
such that $E(x,t)\rightarrow0$ for almost every $(x,t)\in\Xi^{\prime}$.

\textbf{Estimate of $T_{2}$:} We denote 
\[\Gamma_{\varepsilon}:=\left\{ (x,t)\in\Xi:\overline B_{r(\varepsilon)}(x)\times[t-r(\varepsilon),t+r(\varepsilon)]\cap\Gamma\not=\emptyset\right\}.
\]
Note that if $(x,t)\in\Xi\setminus\Gamma_{\varepsilon}$, then $(x_{\varepsilon},t_{\varepsilon})\in\Xi\setminus\Gamma$,
in which case $Da$ is continuous at $(x_{\varepsilon},t_{\varepsilon})$.
Using this we estimate
\begin{align*}
T_{2} & =\left|\eta\right|^{q-2}(\eta\cdot Da(x,t)-\liminf_{(y,s)\rightarrow(x_{\varepsilon},t_{\varepsilon}),(y,s)\not\in\Gamma}\eta\cdot Da(y,s))\\
 & \geq\left|\eta\right|^{q-2}\chi_{\Xi\setminus\Gamma_{\varepsilon}}(x,t)(\eta\cdot Da(x,t)-\eta\cdot Da(x_{\varepsilon},t_{\varepsilon}))-2\left|\eta\right|^{q-1}\chi_{\Gamma_{\varepsilon}}(x,t)\left\Vert Da\right\Vert _{L^{\infty}(\Xi)})\\
 & \geq -\left|\eta\right|^{q-1}\chi_{\Xi\setminus\Gamma_{\varepsilon}}(x,t)\max_{(y,s)\in\overline B_{r(\varepsilon)}(x)\times[t-r(\varepsilon),t+r(\varepsilon)]}\left|Da(x,t)-Da(y,s)\right|\\
 & \phantom{=\,\,}-|\eta|^{q-1}2\chi_{\Gamma_\varepsilon}(x,t)\left\Vert Da\right\Vert _{L^{\infty}(\Xi)}.
\end{align*}
Here on the right-hand side, the coefficient of $\left|\eta\right|^{q-1}$
is measurable since the first term is continuous almost everywhere
by continuity of $Da$ in the set $\Xi\setminus\Gamma$. Moreover,
it is bounded and converges to zero almost everywhere in $\Xi^{\prime}$.
Finally, again by Lemma \ref{lem:nullset convergence} we have that
the second term converges to zero almost everywhere. Thus $T_{2}$
has an estimate of the type 
\[
T_{2}\leq-\left|\eta\right|^{q-1}E_{2}(x,t),
\]
where $E_{2}:\Xi^{\prime}\rightarrow[0,\infty)$ is bounded, measurable,
and converges to zero almost everywhere in $\Xi^{\prime}$.

\textbf{Estimate of $T_3$: } Using \eqref{eq:a shift est} and the formula $|\eta| = |x-x_\varepsilon|^{\ell-1}/\varepsilon^{\ell-1}$, we obtain
\begin{align*}
    |a(x_\varepsilon, t_\varepsilon)-a(x,t)||\eta|^{\beta_2} & \leq (\varepsilon^{\ell - 1} + ||Da||_{L^\infty}|x-x_\varepsilon|)|\eta|^{\beta_2} \\
    & = (|x-x_\varepsilon|^{\ell-1}+||Da||_{L^\infty}\frac{|x-x_\varepsilon|^{\ell}}{\varepsilon^{\ell-1}})|\eta|^{\beta_2-1} \\
    & \leq (|x-x_\varepsilon|^{\ell-1}+||Da||_{L^\infty}\ell(u(x,t)-u(x_\varepsilon, t_\varepsilon)))|\eta|^{\beta_2 -1} \\
    & \leq (2^{\ell-1} + \ell||Da||_{L^\infty}\osc_{\Xi}u)|\eta|^{\beta_2-1}
\end{align*}
and so by the growth condition \eqref{eq:gcnd} we have the estimate
\begin{align*}
    |f(x_\varepsilon, t_\varepsilon, \eta)| & \leq C_f (1 + |\eta|^{\beta_1}+a(x_\varepsilon, t_\varepsilon)|\eta|^{\beta_2}) \\
    & \leq C(1+|\eta|^{\beta_1}+|\eta|^{\beta_2-1}+a(x,t)|\eta|^{\beta_2}),
\end{align*}
where $C$ depends only on $C_f$, $\ell$, $||Da||_{L^\infty}$ and $\osc_{\Xi} u$.
Using the above estimate, we get for any $K>0$
\begin{align}\label{eq:f_K,epsilon}
    T_3 = f(x_\varepsilon, t_\varepsilon, \eta) & \geq \begin{cases}\inf_{(y,s)\in B_r(\varepsilon)\times(t-r(\varepsilon), t+r(\varepsilon))}f(y,s,\eta) &\text{if } |\eta| < K,\nonumber \\
    -C(1+|\eta|^{\beta_1}+|\eta|^{\beta_2  -1}+a(x,t)|\eta|^{\beta_2}) & \text{if }|\eta|\geq K.\end{cases} \\
    & =: f_{K, \varepsilon}(x, t, \eta).
\end{align}
Combining the estimates of $T_{1}$, $T_{2}$ and $T_3$, we obtain the claim in the case $\eta \not = 0$.

Finally, we consider the case $\eta = 0$. Since now $x = x_\varepsilon$, by (4) of Lemma \ref{lem:inf conv properties} we have
\begin{align*}
    u_{\varepsilon}(x, t)=u(x, t_{\varepsilon})+\frac{|t-t_{\varepsilon}|^2}{2\delta_{\varepsilon}}.
\end{align*}
Then by the definition of inf-convolution,
\begin{align*}
    u(y, s)+\frac{|x-y|^{\ell}}{\ell \varepsilon^{\ell-1}}+\frac{|t-s|^2}{2\delta_{\varepsilon}}\geq u_{\varepsilon}(x, t)=u(x, t_{\varepsilon})+\frac{|t-t_{\varepsilon}|^2}{2\delta_{\varepsilon}}
\end{align*}
it follows that the function $\phi(y,s) :=u(x, t_\varepsilon) - \frac{|y-x|^\ell}{\ell \varepsilon ^{\ell -1}} - \frac{|s-t|^2}{2\delta_\varepsilon} + \frac{|t_\varepsilon -t|^2}{2\delta_\varepsilon} \leq u(y, s)$ touches $u$ from below at $(y,s) = (x, t_\varepsilon)$. Hence, since $u$ is a viscosity supersolution and $D\phi(y,s) \not = 0$ when $y \not=x$, we have
\begin{align*}
\limsup_{\substack{\substack{(y,s)\rightarrow(x,t_\varepsilon)\\
y\not=x, (y,s)\not \in \Gamma
}
}
}&\Bigg[\partial_t\phi(y,s)-\Delta_p\phi(y,s)-a(y,s)\Delta_q\phi(y,s)-|D\phi(y,s)|^{q-2}D\phi(y,s)\cdot Da(y,s)\\
&-f\left(y, s, D\phi(y, s)\right)\Bigg]\geq0.    
\end{align*}
On the other hand, since $\ell > p/(p-1) \geq q/(q-1)$ and $\ell \geq 3 >2$, we have $\Delta_p \phi(y,s), \Delta_q\phi(y,s) \rightarrow 0$ as $y \rightarrow x$ whenever $1<p\leq q<\infty$. So the above display implies that
\[
\theta = \frac{t - t_\varepsilon}{\delta_\varepsilon} = \partial_s \phi(x,t_\varepsilon) \geq f(x,t_\varepsilon,0) \geq f_{K,\varepsilon}(x,t,0). \qedhere
\]
\end{proof}

\begin{rem}
Note that if $1<p\leq q <p+1<\infty$, then by the growth condition \eqref{eq:gcnd}, the function $f_{K,\varepsilon}$ in Lemma \ref{lem:infconv is super} satisfies the growth estimate
\begin{equation}\label{eq:f_M,epsilon growth estimate}
    |f_{K,\varepsilon}(x,t,\eta)| \leq C(1+|\eta|^{\tilde \beta} + a(x,t)|\eta|^{\beta_2}),
\end{equation}    
where $\tilde \beta = \max(\beta_1, \beta_2 -1) < p$ and $C$ depends only on $C_f$, $\beta_1$, $\beta_2$, $q$, $\ell$, $||a||_{L^\infty}$, $||Da||_{L^\infty}$, $\osc_\Xi u$ and $K$. In other words, $f_{K, \varepsilon}$ satisfies a similar growth estimate as $f$. This observation is used when deriving energy estimates for $u_\varepsilon$.
\end{rem}

\begin{lem}\label{lem:inf conv is weak}
Let $1<p\leq q\leq p+1$. Suppose that $a:\Xi \rightarrow [0, \infty)$ is Lipschitz in space, uniformly continuous in time, and that \eqref{eq:visc is weak acnd} holds. Let $u$ be a bounded viscosity supersolution to (\ref{eq:p-para f})
in $\Xi$ and $K>0$. Let $u_\varepsilon$ be the inf-convolution of $u$ as in Lemma \ref{lem:inf delta lemma}. Fix a smaller domain $\Xi^\prime \Subset \Xi$ and suppose that $\varepsilon >0$ is so small that $\Xi^\prime \Subset \Xi_\varepsilon$. Then $u_{\varepsilon}$ is a weak supersolution to
\begin{equation}\label{eq:inf conv weak}
\partial_t u_\varepsilon -\div(|Du_\varepsilon|^{p-2}Du_\varepsilon+a(z)|Du_\varepsilon|^{q-2}Du_\varepsilon) - f_{K,\varepsilon}(z,Du_\varepsilon)\geq-(1+|Du_\varepsilon|^{q-1})E_{\varepsilon}(z)
\end{equation}
in $\Xi^\prime$, where $f_{K,\varepsilon}$ and $E_\varepsilon$ are as in Lemma \ref{lem:infconv is super}.
\end{lem}
\begin{proof}
\textbf{Step 1:} We define a concave function
\[
\phi(x,t):=u_{\varepsilon}(x,t)-C_\varepsilon\left(\left|x\right|^{2}+t^{2}\right)
\]
in $\smash{\Xi_{\varepsilon}}$, where $C_\varepsilon \geq 0$ is a semi-concavity constant from Lemma \ref{lem:inf conv properties}. Since $u_{\varepsilon}$ is concave, Alexandrov's theorem implies that $\smash{u_{\varepsilon}}$ is twice differentiable almost everywhere in $\smash{\Xi_{\varepsilon}}.$
Furthermore, the proof of Alexandrov's theorem in \cite[p273]{measuretheoryevans} establishes that if $\smash{\phi_{j}}$ is the standard mollification of $\phi$, then $\smash{D^{2}\phi_{j}\rightarrow D^{2}\phi}$ almost everywhere in $\smash{\Xi_{\varepsilon}}$ and thus we can approximate
it by smooth concave functions $\smash{\phi_{j}}$ so that $$\smash{\left(\phi_{j},\partial_{t}\phi_{j},D\phi_{j},D^{2}\phi_{j}\right)\rightarrow\left(\phi,\partial_{t}\phi,D\phi,D^{2}\phi\right)}$$
a.e.$\:$in $\smash{\Xi^\prime}$. We define
\[
u_{\varepsilon,j}(x,t):=\phi_{j}(x,t)+C_\varepsilon\left(\left|x\right|^{2}+t^{2}\right)
\]
and denote by $a_j$ the standard mollification of $a$. Consider the regularized $p$-Laplacian
\begin{equation}
\Delta_{p,\delta}u:=\left(\delta+\left|Du\right|^{2}\right)^{\frac{p-2}{2}}\bigg(\Delta u+\frac{p-2}{\delta+\left|Du\right|^{2}}\Delta_{\infty}u\bigg),\label{eq:regplap}
\end{equation}
and $\Delta_{q, \delta}$ defined in the analogical way, where $\Delta_{\infty}u:=\left\langle D^{2}uDu,Du\right\rangle$ and $\delta >0$. Fix a non-negative test function $\varphi\in C_{0}^{\infty}(\Xi^\prime)$.  Since $u_{\varepsilon,j}$ is smooth and $\varphi$ is compactly supported in $\Xi^\prime$, we calculate via integration by parts
\begin{align*}
&\int_{\Xi^\prime} \varphi\bigg(\partial_{t}u_{\varepsilon,j}-\left(\delta+\left|Du_{\varepsilon,j}\right|^{2}\right)^{\frac{p-2}{2}}\bigg(\Delta u_{\varepsilon,j}+\frac{p-2}{\delta+\left|Du_{\varepsilon,j}\right|^{2}}\Delta_{\infty}u_{\varepsilon,j}\bigg)\bigg)\d z\\
&-\int_{\Xi^\prime} a_j \left(\delta+\left|Du_{\varepsilon,j}\right|^{2}\right)^{\frac{q-2}{2}}\bigg(\Delta u_{\varepsilon,j}+\frac{q-2}{\delta+\left|Du_{\varepsilon,j}\right|^{2}}\Delta_{\infty}u_{\varepsilon,j}\bigg)+\left(\delta+|Du_{\varepsilon, j}|^2\right)^{\frac{q-2}{2}}Du_{\varepsilon, j}\cdot Da_j\, dz\\
&= \int_{\Xi^\prime}\varphi\partial_{t}u_{\varepsilon,j}-\varphi\div\left(\left(\delta+\left|Du_{\varepsilon,j}\right|^{2}\right)^{\frac{p-2}{2}}Du_{\varepsilon,j}+a_j \left(\delta+\left|Du_{\varepsilon,j}\right|^{2}\right)^{\frac{p-2}{2}}Du_{\varepsilon,j} \right)\d z\\
&= \int_{\Xi^\prime}-u_{\varepsilon,j}\partial_{t}\varphi+\left[\left(\delta+\left|Du_{\varepsilon,j}\right|^{2}\right)^{\frac{p-2}{2}}Du_{\varepsilon,j}+a_j \left(\delta+\left|Du_{\varepsilon,j}\right|^{2}\right)^{\frac{q-2}{2}}Du_{\varepsilon,j}\right]\cdot D\varphi\d z.
\end{align*}
Recalling the shorthand $\Delta_{p,\delta}$ defined in (\ref{eq:regplap}),
we deduce from the above that
\begin{align}
&\liminf_{j\rightarrow\infty} \int_{\Xi^\prime}\varphi\left(\partial_{t}u_{\varepsilon,j}-\Delta_{p,\delta}u_{\varepsilon,j}-a_j\Delta_{q, \delta}u_{\varepsilon, j}-\left(\delta+|Du_{\varepsilon, j}|\right)^{\frac{q-2}{2}}Du_{\varepsilon , j}\cdot Da_j\right)\d z\nonumber \\
&\leq  \lim_{j\rightarrow\infty}\int_{\Xi^\prime}-u_{\varepsilon,j}\partial_{t}\varphi+\left[\left(\delta+\left|Du_{\varepsilon,j}\right|^{2}\right)^{\frac{p-2}{2}}Du_{\varepsilon,j}+\left(\delta+\left|Du_{\varepsilon,j}\right|^{2}\right)^{\frac{q-2}{2}}Du_{\varepsilon,j}\right]\cdot D\varphi\d z.\label{eq:inf conv is weak p<2}
\end{align}
We use Fatou's lemma at the left-hand side and the dominated convergence at the right-hand side. Once we verify their assumptions, we arrive at the auxiliary inequality (note that $a_j \rightarrow a$ uniformly in $\Xi^\prime$ and $Da_j \rightarrow Da$ in $L^2(\Xi^\prime)$)
\begin{align}\label{eq:inf conv is weak aux p<2}
    & \int_{\Xi^\prime} \varphi\big(\partial_t u_\varepsilon - \Delta_{p,\delta}u_\varepsilon -a\Delta_{q, \delta} u_\varepsilon-(\delta+|Du_{\varepsilon, \delta}|^2)^{\frac{q-2}{q}}Du_\varepsilon \cdot Da \big) \d z \nonumber \\
    & \leq \int_{\Xi^\prime} -u_\varepsilon \partial_t\varphi + (\delta+|Du_\varepsilon|^2)^{\frac{p-2}{2}}Du_\varepsilon \cdot D\varphi+a(\delta+|Du_\varepsilon|^2)^{\frac{q-2}{2}}Du_\varepsilon\cdot D\varphi \d z.
\end{align}
Next we verify the assumptions of Fatou's lemma and the dominated
convergence theorem. Since $u_\varepsilon$ is Lipschitz in $\Xi^\prime$, $\smash{\left|u_{\varepsilon,j}\right|}$,
$\smash{\left|\partial_{t}u_{\varepsilon,j}\right|}$ and $\smash{\left|Du_{\varepsilon,j}\right|}$
are uniformly bounded by some constant $\smash{M>1}$ in the support
of $\varphi$, independently of $j$. Hence the assumptions of the dominated convergence theorem are satisfied. Observe then that by concavity of $\smash{\phi_{j}}$, we have $\smash{D^{2}u_{\varepsilon,j}\leq 2C_\varepsilon I}$.
Thus the integrand at the left-hand side of (\ref{eq:inf conv is weak p<2})
has a lower bound independent of $j$ when $\smash{Du_{\varepsilon,j}=0}$. When $\smash{Du_{\varepsilon,j}\not=0}$, we use that the operator $-\Delta v-\frac{p-2}{|Dv|^2}\Delta_\infty v$ is degenerate elliptic when $Dv \not= 0$ and $p>1$, to estimate \begingroup\allowdisplaybreaks
\begin{align*}
 & -\left(\delta+\left|Du_{\varepsilon,j}\right|^{2}\right){}^{\frac{p-2}{2}}\left(\Delta u_{\varepsilon,j}+\frac{p-2}{\delta+\left|Du_{\varepsilon,j}\right|^{2}}\Delta_{\infty}u_{\varepsilon,j}\right)\\
 & =-\frac{\left(\delta+\left|Du_{\varepsilon,j}\right|^{2}\right)^{\frac{p-2}{2}}}{\delta+\left|Du_{\varepsilon,j}\right|^{2}}\left(\left|Du_{\varepsilon,j}\right|^{2}\left(\Delta u_{\varepsilon,j}+\frac{p-2}{\left|Du_{\varepsilon,j}\right|^{2}}\Delta_{\infty}u_{\varepsilon,j}\right)+\delta\Delta u_{\varepsilon,j}\right)\\
 & \geq-\frac{\left(\delta+\left|Du_{\varepsilon,j}\right|^{2}\right)^{\frac{p-2}{2}}}{\delta+\left|Du_{\varepsilon,j}\right|^{2}}2C_\varepsilon\left(\left|Du_{\varepsilon,j}\right|^{2}(N+p-2)+\delta N\right)\\
 & \geq-2C_\varepsilon(\delta+\left|Du_{\varepsilon,j}\right|^{2})^{\frac{p-2}{2}}\max(N+p-2,N),
\end{align*}
which provides a lower bound independent of $j$ since $Du_{\varepsilon, j}$ is bounded. Similarly, one checks that the term with $q$ has a lower bound. Therefore, our use of Fatou's lemma is justified.

\textbf{Step 2:} We let $\delta\rightarrow0$ in the auxiliary inequality
(\ref{eq:inf conv is weak aux p<2}). Since $u_{\varepsilon}$ is
Lipschitz continuous, the dominated convergence theorem implies 
\begin{align}
\liminf_{\delta\rightarrow0} & \int_{\Xi^\prime}\varphi\left(\partial_{t}u_{\varepsilon}-\Delta_{p,\delta}u_{\varepsilon}-a\Delta_{q,\delta}u_\varepsilon-(\delta+|Du_\varepsilon|^2)^{\frac{q-2}{2}}Du_\varepsilon\cdot Da\right)\d z\nonumber \\
\leq & \int_{\Xi^\prime}-u_{\varepsilon}\partial_{t}\varphi+\left|Du_{\varepsilon}\right|^{p-2}Du_{\varepsilon}\cdot D\varphi+a|Du_\varepsilon|^{q-2}Du_\varepsilon\cdot D\varphi\d z.\label{eq:xiao1}
\end{align}
Applying Fatou's lemma (we verify assumptions at the end), we get
\begin{align}
\liminf_{\delta\rightarrow0} & \int_{\Xi^\prime}\varphi\left(\partial_{t}u_{\varepsilon}-\Delta_{p,\delta}u_{\varepsilon}-a\Delta_{q,\delta} u_\varepsilon -(\delta+|Du_\varepsilon|^2)^{\frac{q-2}{2}}Du_\varepsilon\cdot Da\right)\d z\nonumber \\
\geq & \int_{\Xi^\prime}\liminf_{\delta\rightarrow0}\varphi\left(\partial_{t}u_{\varepsilon}-\Delta_{p,\delta}u_{\varepsilon}-a\Delta_{q, \delta}u_\varepsilon -(\delta+|Du_\varepsilon|^2)^{\frac{q-2}{2}}Du_\varepsilon\cdot Da\right)\d z\nonumber \\
= & \int_{\Xi^\prime\cap\left\{ Du_{\varepsilon}\not=0\right\} }\liminf_{\delta\rightarrow0}\varphi\left(\partial_{t}u_{\varepsilon}-\Delta_{p,\delta}u_{\varepsilon}-a\Delta_{q, \delta} u_\varepsilon -(\delta+|Du_\varepsilon|^2)^{\frac{q-2}{2}}Du_\varepsilon\cdot Da\right)\d z\nonumber \\
 & +\int_{\Xi^\prime\cap\left\{ Du_{\varepsilon}=0\right\} }\liminf_{\delta\rightarrow0}\varphi(\partial_{t}u_{\varepsilon}-\delta^{\frac{p-2}{2}}\Delta u_{\varepsilon}-\delta^{\frac{q-2}{2}}a\Delta u_\varepsilon )\d z\nonumber \\
= & \int_{\Xi^\prime\cap\left\{ Du_{\varepsilon}\not=0\right\} }\varphi\left(\partial_{t}u_{\varepsilon}-\Delta_{p}u_{\varepsilon}-a\Delta_q u_\varepsilon - |Du_\varepsilon|^{q-2}Du_\varepsilon \cdot Da\right)\d z\nonumber \\
 & +\int_{\Xi^\prime\cap\left\{ Du_{\varepsilon}=0\right\} }\varphi\partial_{t}u_{\varepsilon}\d z\nonumber \\
 \geq & \int_{\Xi^\prime} \varphi (f_{K, \varepsilon}(x, t, Du_\varepsilon)-(1+|Du_\varepsilon|^{q-1})E_\varepsilon (x,t)) \d z,\label{eq:xiao2}
\end{align}
where the last inequality follows from Lemma \ref{lem:infconv is super}. The lemma could be applied at almost every $(x,t)\in \supp \varphi$, since $|\Gamma| = 0$ and $u_{\varepsilon}$ is twice differentiable almost everywhere in $\Xi_\varepsilon$, and at such points we have $(\partial_t u_\varepsilon(x,t),Du_\varepsilon(x,t),D^2 u_\varepsilon (x,t)) \in \mathcal P^{2,-} u(x,t)$.
Combining (\ref{eq:xiao1}) and (\ref{eq:xiao2}), we find that $u_{\varepsilon}$
is a weak supersolution in $\Xi^\prime$. It remains to verify
the assumptions of Fatou's lemma, i.e.\ that the integrand at the
left-hand side of (\ref{eq:xiao1}) has a lower bound independent
of $\delta$. When $Du_{\varepsilon}=0$, this follows directly from
the inequality 
\[
D^{2}u_{\varepsilon}\leq\frac{\ell-1}{\varepsilon}\left|Du_{\varepsilon}\right|^{\frac{\ell-1}{\ell-2}}I,
\]
which holds by Lemma \ref{lem:inf conv properties}. When $Du_{\varepsilon}\not=0$,
we recall that by Lipschitz continuity $\partial_{t}u_{\varepsilon}$
and $Du_{\varepsilon}$ are uniformly bounded in $\Xi_{\varepsilon}$,
and estimate
\begin{align*}
-\Delta_{p,\delta}u_{\varepsilon} & =\left(-\delta+\left|Du_{\varepsilon}\right|^{2}\right)^{\frac{p-2}{2}}\left(\Delta u_{\varepsilon}+\frac{p-2}{\delta+\left|Du_{\varepsilon}\right|^{2}}\left\langle D^{2}u_{\varepsilon}Du_{\varepsilon},Du_{\varepsilon}\right\rangle \right)\\
 & =-\frac{(\delta+\left|Du_{\varepsilon}\right|^{2})^{\frac{p-2}{2}}}{\delta+\left|Du_{\varepsilon}\right|^{2}}\left(\left|Du_{\varepsilon}\right|^{2}\left(\Delta u_{\varepsilon}+\frac{p-2}{\left|Du_{\varepsilon}\right|^{2}}\left\langle D^{2}u_{\varepsilon}Du_{\varepsilon},Du_{\varepsilon}\right\rangle \right)+\delta\Delta u_{\varepsilon}\right)\\
 & \geq-\frac{(\delta+\left|Du_{\varepsilon}\right|^{2})^{\frac{p-2}{2}}}{\delta+\left|Du_{\varepsilon}\right|^{2}}\frac{\ell-1}{\varepsilon}\left(\left|Du_{\varepsilon}\right|^{\frac{\ell-2}{\ell-1}+2}(N+p-2)+\left|Du_{\varepsilon}\right|^{\frac{\ell-2}{\ell-1}}\delta N\right)\\
 & \geq-\frac{(\delta+\left|Du_{\varepsilon}\right|^{2})^{\frac{p-2}{2}}}{\delta+\left|Du_{\varepsilon}\right|^{2}}\frac{\ell-1}{\varepsilon}\left((\left|Du_{\varepsilon}\right|^{2}+\delta)\left|Du_{\varepsilon}\right|^{\frac{\ell-2}{\ell-1}}(N+p-2)\right)\\
 & \geq-(\delta+\left|Du_{\varepsilon}\right|^{2})^{\frac{p-2}{2}}\left|Du_{\varepsilon}\right|^{\frac{\ell-2}{\ell-1}}\frac{\ell-1}{\varepsilon}\max(N,N+p-2)\\
 & \geq-(\delta+\left|Du_{\varepsilon}\right|^{2})^{\frac{1}{2}(p-2+\frac{\ell-2}{\ell-1})}\frac{\ell-1}{\varepsilon}\max(N,N+p-2),
\end{align*}
which is bounded independently of $\delta$ since $p-2 + \frac{\ell-2}{\ell-1}>0$. A similar estimate holds for $-\Delta_q u_{\varepsilon, \delta}$ and so the assumptions of Fatou's lemma hold.\endgroup
\end{proof}

\begin{lem}[Caccioppoli-type inequality]
\label{lem:Caccioppoli} Let $1<p\leq q \leq p+1<\infty$. Suppose that $a:\Xi \rightarrow [0, \infty)$ is bounded and measurable. Assume that $u$ is a  locally Lipschitz continuous weak supersolution to
\begin{align*}
    \partial_t u-\div\left(|Du|^{p-2}Du+a(z)|Du|^{q-2}Du\right)\geq -\tilde C(1+|Du|^{\tilde\beta_1}+a(z)|Du|^{\tilde \beta_2}) \quad \text{in } \Xi,
\end{align*}
where $\tilde C \geq0$ is a constant and $1\leq\tilde \beta_1 < p$, $1\leq\tilde \beta_2 <q$. Then there is a constant $C=C(p, q, \tilde C,\tilde\beta_1, \tilde\beta_2, )$ such that
for any non-negative test function $\xi\in C_{0}^{\infty}(\Xi)$, $\xi \leq 1$, we have
\begin{align*}
&\int_{\Xi}\xi^{q}\left(\left|Du\right|^{p}+a(z)|Du|^q\right)\d z\\
&\leq C\int_{\Xi}M^{2}\partial_{t}\xi^{q}+\max\{M^p, M^{q}\}\left(\left|D\xi\right|^{p}+a(z)|D\xi|^q\right)+(M^{\frac{p}{p-\beta_1}}+M^{\frac{p}{p-\beta_2}}+M)\xi^{q}\,dz,
\end{align*}
where $M=\left\Vert u\right\Vert _{L^{\infty}(\supp\xi)}$.
\end{lem}
\begin{proof}
Since $u$ is locally Lipschitz continuous, the function $\varphi:=\left(M-u\right)\xi^{q}$
is an admissible test function. Testing the weak formulation of (\ref{eq:p-para f})
with $\varphi$ yields
\begin{align}
&\int_{\Xi}\xi^{q}\left(\left|Du\right|^{p}+a(z)|Du|^q\right)\d z \nonumber\\
&\leq  \int_{\Xi}-u\partial_{t}\varphi \d z + \int _\Xi q\xi^{q-1}(M-u)\left(\left|Du\right|^{p-1}+a(z)|Du|^{q-1}\right)\left|D\xi\right|\,dz\nonumber\\
&\phantom{\leq} +\int _\Xi \varphi \tilde C (1+|Du|^{\tilde \beta_1} + a(z)|Du|^{\tilde \beta_2}) dz.\label{eq:caccioppoli 1}
\end{align}
We have by integration by parts
\begin{align*}
\int_{\Xi}-u\partial_{t}\varphi\d z= & \int_{\Xi}\xi^{p}u\partial_{t}u-u(M-u)\partial_{t}\xi^{p}\d z\\
= & \int_{\Xi}\frac{1}{2}\xi^{p}\partial_{t}u^{2}-u(M-u)\partial_{t}\xi^{p}\d z\\
= & \int_{\Xi}-\frac{1}{2}u^{2}\partial_{t}\xi^{p}-u(M-u)\partial_{t}\xi^{p}\d z\\
= &\int_{\Xi}\frac{1}{2}u^2\partial_t\xi^p-uM\partial_t\xi^p
\leq\int_{\Xi}CM^{2}\partial_{t}\xi^{p}\d z.
\end{align*}
Now we estimate the second term of \eqref{eq:caccioppoli 1}. Applying Young's inequality, we get for any $\delta >0$
\begin{align*}
& \int_{\Xi}q\xi^{q-1}(M-u)\left(\left|Du\right|^{p-1}+a(z)|Du|^{q-1}\right)\left|D\xi\right|\,dz\\
 & \leq \int_{\Xi} \delta \xi^{\frac{p(q-1)}{p-1}}|Du|^p+C(p, \delta)M^p|D\xi|^p\, dz\\
 &\ \ \ +\int_{\Xi} \delta a(z)\xi^q |Du|^q+C(q, \delta)M^qa(z)|D\xi|^q\, dz\\
 &\leq \int_{\Xi} \delta \xi^q\left(|Du|^p+a(z) |Du|^q\right)+C(p,q,\delta)\max\{M^p,M^q\}\left(|D\xi|^p+a(z)|D\xi|^q\right)\, dz.
\end{align*}
Next we estimate the third term of \eqref{eq:caccioppoli 1}. We have
\begin{align*}
 & \int_{\Xi}\varphi \tilde C (1+|Du|^{\tilde \beta_1}+a(z)|Du|^{\tilde \beta_2}\,dz  \\ & \leq \int_{\Xi}\tilde C(M-u)\xi^q\left(|Du|^{\tilde\beta_1}+a(z)|Du|^{\tilde\beta_2}\right)\,dz+\int_{\Xi}\tilde C(M-u)\xi^q \,dz.
\end{align*}
Applying Young's inequality with the exponent pairs $\left(\frac{p}{\tilde\beta_1}, \frac{p}{p-\tilde\beta_1}\right)$ and $\left(\frac{q}{\tilde\beta_1}, \frac{q}{q-\tilde\beta_2}\right), $ we obtain
\begin{align*}
&\int_{\Xi}\tilde C(M-u)\xi^q\left(|Du|^{\tilde\beta_1}+a(z)|Du|^{\tilde\beta_2}\right)\,dz\\
&\leq\int_{\Xi}\delta |Du|^p \xi^q + C(\delta) (\tilde C (M-u))^{\frac{p}{p-\tilde\beta_1}}\xi^{\left(q-\frac{q\tilde\beta_1}{p}\right)\frac{p}{p-\tilde\beta_1}}\,dz \\ &\ \ \ +\int_{\Xi}\delta a(z)|Du|^q \xi^q+ C(\delta) (\tilde C (M-u)^{\frac{q}{q-\tilde\beta_2}})\xi^q\, dz.
\end{align*}
Thus, the third term of \eqref{eq:caccioppoli 1} can be estimated by
\begin{equation*}
\int_{\Xi}\delta \left(|Du|^p+a(z)|Du|^q\right)\xi^q\,dz + C(\tilde C, \delta) \int_{\Xi}\left(M^{\frac{p}{p-\tilde\beta_1}}+M^{\frac{q}{q-\tilde\beta_2}}+M \right)\xi^q\, dz.
\end{equation*}
Combining these estimates with (\ref{eq:caccioppoli 1}) and absorbing the terms with $Du$ to the left-hand side by taking small enough $\delta>0$, we obtain the desired inequality.
\end{proof}
Before proving the next lemma, let us define for $1< r_1< p$ and $1<r_2< q$ the space
\begin{align*}
    L^{r_1, r_2}(\Xi):=\left\{u: \Xi\to \mathbb{R}\,\, \text{measurable}\,\, : \int_{\Xi}|u|^{r_1}+a(z)|u|^{r_2}\, dz< \infty\right\}.
\end{align*}
\begin{lem}
\label{lem:lr conv} Let $1<p\leq q\leq p+1$. Suppose that $a:\Xi \rightarrow [0, \infty)$ is bounded and measurable. Assume that $\smash{\left(u_{j}\right)}$
is a sequence of locally Lipschitz continuous weak supersolutions to 
\begin{align*}
    \partial_t u_j-\div\left(|Du_j|^{p-2}Du_j+a(z)|Du_j|^{q-2}Du_j\right)-f(z, Du_j)\geq \left(1+|Du_j|^{q-1}\right)E \quad \text{in } \Xi,
\end{align*}
where $E\geq0$ is a constant. Suppose moreover that for any $U\Subset\Xi$ we have $\sup_j ||u_j||_{L^\infty(U)}<\infty$, $u_j$ converges in $L^p(U)$ and that
\begin{equation}\label{eq:strong conv cnd}
    \sup_j \int _U |Du_j|^p + a(z)|Du_j|^q \,dz < \infty \quad \text{for any } U\Subset\Xi.
\end{equation}
Then $\smash{\left(Du_{j}\right)}$ is a Cauchy
sequence in $\smash{L_{loc}^{r_1, r_2}(\Xi)}$ for any $\smash{1<r_1<p}$ and $\smash{1<r_2<q}$. 
\end{lem}
\begin{proof}
Let $U\Subset\Xi$ and take a cut-off function $\theta\in C_{0}^{\infty}(\Xi)$
such that $0\leq\theta\leq1$ and $\theta\equiv1$ in $U$. For $\delta>0$,
we set
\[
w_{jk}=\begin{cases}
\delta, & u_{j}-u_{k}>\delta,\\
u_{j}-u_{k}, & \left|u_{j}-u_{k}\right|\leq\delta,\\
-\delta, & u_{j}-u_{k}<-\delta.
\end{cases}
\]
Then the function $(\delta-w_{jk})\theta$ is an admissible test function
with a time derivative since it is Lipschitz continuous. Since $u_{j}$
is a weak supersolution, testing the weak formulation of (\ref{eq:p-para f})
with $(\delta-w_{jk})\theta$ yields
\begin{align*}
0\leq & \int_{\Xi}-u_{j}\partial_{t}((\delta-w_{jk})\theta)+\left(\left|Du_{j}\right|^{p-2}Du_{j}+a(z)\left|Du_{j}\right|^{q-2}Du_{j}\right)\cdot D((\delta-w_{jk})\theta)\\
&-(\delta-w_{jk})\theta f(z, Du_{j})-(\delta-w_{jk})\theta\left(1+|Du_j|^{q-1}\right)E\d z\\
= &\int_{\Xi}-\theta\left(\left|Du_{j}\right|^{p-2}Du_{j}+a(z)\left|Du_{j}\right|^{q-2}Du_{j}\right)\cdot Dw_{jk}\\
&+(\delta-w_{jk})\left(\left|Du_{j}\right|^{p-2}Du_{j}+a(z)\left|Du_{j}\right|^{q-2}Du_{j}\right)\cdot D\theta-(\delta-w_{jk})\theta f(z, Du_{j})\\
 & -(\delta-w_{jk})\theta\left(1+|Du_j|^{q-1}\right)E +u_{j}\partial_{t}(w_{jk}\theta)-(\delta-w_{jk})u_{j}\partial_{t}\theta\d z.
\end{align*}
Since $\left|w_{jk}\right|\leq\delta$ and $Dw_{jk}=\chi_{\left\{ \left|u_{j}-u_{k}\right|<\delta\right\} }\left(Du_{j}-Du_{k}\right)$,
the above becomes
\begin{align*}
&\int_{\left\{ \left|u_{j}-u_{k}\right|<\delta\right\} }\theta\left(\left|Du_{j}\right|^{p-2}Du_{j}+a(z)\left|Du_{j}\right|^{q-2}Du_{j}\right)\cdot\left(Du_{j}-Du_{k}\right)\d z\\
&\leq \int_{\Xi}2\delta\left(\left|Du_{j}\right|^{p-1}+a(z)|Du_j|^{q-1}\right)\left|D\theta\right|+2\delta\theta\left|f(z, Du_{j})\right|+u_{j}\partial_{t}(w_{jk}\theta)+2\delta\left|u_{j}\right|\left|\partial_{t}\theta\right|\\
&+2\delta \theta \left(1+|Du_j|^{q-1}\right)E\d z.
\end{align*}
Since $u_{k}$ is a weak supersolution, the same arguments as above
but testing this time with $(\delta+w_{jk})\theta$ yield the analogous
estimate
\begin{align*}
&\int_{\left\{ \left|u_{j}-u_{k}\right|<\delta\right\} } -\theta\left(\left|Du_{k}\right|^{p-2}Du_{k}+a(z)\left|Du_{k}\right|^{q-2}Du_{k}\right)\cdot\left(Du_{j}-Du_{k}\right)\d z\\
&\leq \int_{\Xi}2\delta\left(\left|Du_{k}\right|^{p-1}+a(z)\left|Du_{k}\right|^{q-1}\right)\left|D\theta\right|+2\delta\theta\left|f(z, Du_{k})\right|-u_{k}\partial_{t}\left(w_{jk}\theta\right)+2\delta\left|u_{k}\right|\left|\partial_{t}\theta\right|\\
&+2\delta \theta \left(1+|Du_k|^{q-1}\right)E\d z.
\end{align*}
Summing up these two inequalities we arrive at
\begin{align}
&\int_{\left\{ \left|u_{j}-u_{k}\right|<\delta\right\} }\theta\Big[\left(\left|Du_{j}\right|^{p-2}Du_{j}+a(z)\left|Du_{j}\right|^{q-2}Du_{j}\right)\nonumber\\
&\ \ \ -\left(\left|Du_{k}\right|^{p-2}Du_{k}+a(z)\left|Du_{k}\right|^{q-2}Du_{k}\right)\Big]\cdot\left(Du_{j}-Du_{k}\right)\d z\nonumber \\
&\leq 2\delta\int_{\Xi}\left|D\theta\right|\left(\left|Du_{j}\right|^{p-1}+a(z)\left|Du_{j}\right|^{q-1}+\left|Du_{k}\right|^{p-1}+a(z)\left|Du_{k}\right|^{q-1}\right)\d z\nonumber\\
&\ \ \ +2\delta\int_{\Xi}\theta\left(\left|f(z, Du_{j})\right|+\left|f(z, Du_{k})\right|\right)\d z\nonumber \\
 &\ \ \ +\int_{\Xi}(u_{j}-u_{k})\partial_{t}\left(w_{jk}\theta\right)\d z+2\delta\int_{\Xi}\left(\left|u_{j}\right|+\left|u_{k}\right|\right)\left|\partial_{t}\theta\right|\nonumber\\
 &\ \ \ +2\delta \int_{\Xi}\theta \left(2+|Du_j|^{q-1}+|Du_k|^{q-1}\right)E\d z =: I_{1}+I_{2}+I_{3}+I_{4}+I_{5}.\label{eq:lr conv 1}
\end{align}
We proceed to estimate these integrals. By the assumption \eqref{eq:strong conv cnd}, we have for some $C_0 \geq 0$
\begin{equation}
\sup_{j}\int_{\supp\theta}\left|Du_{j}\right|^{p}+a(z)|Du_j|^q\d z\leq C_0.\label{eq:lr conv 2}
\end{equation}
We have by estimate (\ref{eq:lr conv 2}) and H\"{o}lder's inequality
\[
    \int_\Xi |D\theta|a(z)|Du_j|^{q-1} \, dz \leq \left(\int_\Xi |D\theta|a(z)\,dz\right)^{\frac{1}{q}}\left(\int_\Xi |D\theta| a(z) |Du_j|^q\right)^{\frac{q-1}{q}} \leq C(q,a,\theta,C_0).
\]
The other terms in $I_1$ can be estimated similarly, and so we have
\[
I_{1}\leq\delta C(p,q, a, \theta, C_0).
\]
To estimate $I_{2}$, we also use the growth condition (\ref{eq:gcnd})
and the assumption $1<\beta_1<p$ and $1<\beta_2< q$. We get
\begin{align*}
I_{2}&\leq2\delta\int_{\Xi} \theta C_{f}\left(2+\left|Du_{j}\right|^{\beta_1}+\left|Du_{k}\right|^{\beta_1}+a(z)\left(\left|Du_{j}\right|^{\beta_2}+\left|Du_{k}\right|^{\beta_2}\right)\right)\d z\\
&\leq\delta C(p,q,\beta_1,\beta_2,C_{f},\theta,C_0),
\end{align*}
where in the last estimate we again used \eqref{eq:lr conv 2} and H{\"o}lder's inequality like in the estimate of $I_1$.
The integral $I_{3}$ is estimated using integration by parts and
that $\left|w_{jk}\right|\leq\delta$ 
\begin{align*}
I_{3}= & \int_{\Xi}\theta(u_{j}-u_{k})\partial_{t}\left(w_{jk}\right)+\left(u_{j}-u_{k}\right)w_{jk}\partial_{t}\theta\d z=\int_{\Xi}\frac{1}{2}\theta\partial_{t}w_{jk}^{2}+\left(u_{j}-u_{k}\right)w_{jk}\partial_{t}\theta\d z\\
= & \int_{\Xi}-\frac{1}{2}w_{jk}^{2}\partial_{t}\theta+(u_{j}-u_{k})w_{jk}\partial_{t}\theta\d z\leq\delta C(\theta,M),
\end{align*}
where $M := \sup_j \Vert u \Vert_{L^\infty(U)}<\infty$ by assumption. We have directly $I_{4}\leq\delta C(\theta,M)$. The bound of $I_5\leq 2\delta C(p,q,\theta,E, C_0)$ follows from the assumption $q-1 \leq p$ and \eqref{lem:lr conv}.
Combining these estimates with (\ref{eq:lr conv 1}) we arrive at
\begin{align} \label{eq:lr conv 3}
    \delta C \geq &\int_{\left\{ \left|u_{j}-u_{k}\right|<\delta\right\} }\theta\Big[\left(\left|Du_{j}\right|^{p-2}Du_{j}+a(z)\left|Du_{j}\right|^{q-2}Du_{j}\right) \nonumber \\ 
    &\ \ \ -\left(\left|Du_{k}\right|^{p-2}Du_{k}+a(z)\left|Du_{k}\right|^{q-2}Du_{k}\right)\Big]\cdot\left(Du_{j}-Du_{k}\right)\d z \nonumber \\ 
    & = \int_{\{|u_j - u_k|<\delta\}} \theta\Big[\Big(|Du_j|^{p-2}Du_j-|Du_k|^{p-2}Du_k\Big)\nonumber \\ 
    &\ \ \ +\Big(a(z)|Du_j|^{q-2}Du_j-a(z)|Du_k|^{q-2}Du_k\Big)\Big]\cdot (Du_j - Du_k)\,dz,
\end{align}
where $C$ is independent of $j$, $k$ and $\delta$. Now, we estimate the right-hand side. If $1<p, q<2$, H\"{o}lder's inequality and the algebraic inequality (\ref{eq:algebraic ineq 1<p<2}) give the estimate (recall that $1<r_1<p$ and $1<r_2<q$ and $\theta\equiv1$ in $U$)
\begin{align*}
 & \int_{U\cap\left\{ \left|u_{j}-u_{k}\right|<\delta\right\} }\left|Du_{j}-Du_{k}\right|^{r_1}\d z\\
 & \ \leq\bigg(\int_{U\cap\left\{ \left|u_{j}-u_{k}\right|<\delta\right\} }\left(1+\left|Du_{j}\right|^{2}+\left|Du_{k}\right|^{2}\right)^{\frac{r\left(2-p\right)}{2\left(2-r_1\right)}}\d z\bigg)^{\frac{2-r_1}{2}}\\
 & \ \ \ \ \ \ \cdot\bigg(\int_{U\cap\left\{ \left|u_{j}-u_{k}\right|<\delta\right\} }\frac{\left|Du_{j}-Du_{k}\right|^{2}}{\left(1+\left|Du_{j}\right|^{2}+\left|Du_{k}\right|^{2}\right)^{\frac{2-p}{2}}}\d z\bigg)^{\frac{r_1}{2}}\\
 & \ \leq C(p,\beta,r_1,C_{f},\theta,M)\\
 & \ \ \ \ \ \ \cdot\bigg(\int_{\left\{ \left|u_{j}-u_{k}\right|<\delta\right\} }\theta\left(\left|Du_{j}\right|^{p-2}Du_{j}-\left|Du_{k}\right|^{p-2}Du_{k}\right)\cdot\left(Du_{j}-Du_{k}\right)\d z\bigg)^{\frac{r_1}{2}},
\end{align*}
where in the last inequality we also used \eqref{eq:lr conv 2} with
the knowledge $\frac{r_1(2-p)}{(2-r_1)}\leq\frac{p\left(2-p\right)}{2-p}=p.$
Similarly, we have
\begin{align*}
 & \int_{U\cap\left\{ \left|u_{j}-u_{k}\right|<\delta\right\} }a(z)\left|Du_{j}-Du_{k}\right|^{r_2}\d z=\int_{U\cap\left\{ \left|u_{j}-u_{k}\right|<\delta\right\} }a(z)^{\frac{2-r_2}{2}}a(z)^{\frac{r_2}{2}}|Du_j-Du_k|^{r_2}\, dz\\
 & \ \leq\bigg(\int_{U\cap\left\{ \left|u_{j}-u_{k}\right|<\delta\right\} }a(z)\left(1+\left|Du_{j}\right|^{2}+\left|Du_{k}\right|^{2}\right)^{\frac{r_2\left(2-q\right)}{2\left(2-r_2\right)}}\d z\bigg)^{\frac{2-r_2}{2}}\\
 & \ \ \ \ \ \ \cdot\bigg(\int_{U\cap\left\{ \left|u_{j}-u_{k}\right|<\delta\right\} }\frac{a(z)\left|Du_{j}-Du_{k}\right|^{2}}{\left(1+\left|Du_{j}\right|^{2}+\left|Du_{k}\right|^{2}\right)^{\frac{2-q}{2}}}\d z\bigg)^{\frac{r_2}{2}}\\
 & \ \leq C(q,\beta,r_2,C_{f},\theta,M)\\
 & \ \ \ \ \ \ \cdot\bigg(\int_{\left\{ \left|u_{j}-u_{k}\right|<\delta\right\} }\theta a(z)\left(\left|Du_{j}\right|^{q-2}Du_{j}-\left|Du_{k}\right|^{q-2}Du_{k}\right)\cdot\left(Du_{j}-Du_{k}\right)\d z\bigg)^{\frac{r_2}{2}},   
\end{align*}
where again we use \eqref{eq:lr conv 2} with $\frac{r_2(2-q)}{2-r_1}\leq q.$\\
If $p\geq2$, H\"{o}lder's inequality and the algebraic inequality (\ref{eq:algebraic ineq p>2})
imply
\begin{align*}
 & \int_{U\cap\left\{ \left|u_{j}-u_{k}\right|<\delta\right\} }\left|Du_{j}-Du_{k}\right|^{r_1}\d z\\
 & \ \leq\left(\int_{\Xi}1\d z\right)^{\frac{p-r_1}{p}}\bigg(\int_{U\cap\left\{ \left|u_{j}-u_{k}\right|<\delta\right\} }\left|Du_{j}-Du_{k}\right|^{p}\d z\bigg)^{\frac{r_1}{p}}\\
 & \ \leq C(p,r_1)\bigg(\int_{\left\{ \left|u_{j}-u_{k}\right|<\delta\right\} }\theta\left(\left|Du_{j}\right|^{p-2}Du_{j}-\left|Du_{k}\right|^{p-2}Du_{k}\right)\cdot\left(Du_{j}-Du_{k}\right)\d z\bigg)^{\frac{r_1}{p}}.
\end{align*}
and 
\begin{align*}
 & \int_{U\cap\left\{ \left|u_{j}-u_{k}\right|<\delta\right\} }a(z) \left|Du_{j}-Du_{k}\right|^{r_2}\d z=\int_{U\cap\left\{ \left|u_{j}-u_{k}\right|<\delta\right\} }a(z)^{1-r_2/q}a(z)^{r_2/q} \left|Du_{j}-Du_{k}\right|^{r_2}\d z\\
 & \ \leq\left(\int_{\Xi}a(z)\d z\right)^{\frac{q-r_2}{q}}\bigg(\int_{U\cap\left\{ \left|u_{j}-u_{k}\right|<\delta\right\} }a(z)\left|Du_{j}-Du_{k}\right|^{q}\d z\bigg)^{\frac{r_2}{q}}\\
 & \ \leq C(q,r_2)\bigg(\int_{\left\{ \left|u_{j}-u_{k}\right|<\delta\right\} }\theta a(z)\left(\left|Du_{j}\right|^{q-2}Du_{j}-\left|Du_{k}\right|^{q-2}Du_{k}\right)\cdot\left(Du_{j}-Du_{k}\right)\d z\bigg)^{\frac{r_2}{q}}.
\end{align*}
The case $1<p\leq 2\leq q$ can be handled by combining the above cases. Hence (\ref{eq:lr conv 3}) leads to
\[
\int_{U\cap\left\{ \left|u_{j}-u_{k}\right|<\delta\right\} }\left|Du_{j}-Du_{k}\right|^{r_1}+a(z)\left|Du_{j}-Du_{k}\right|^{r_2}\d z\leq\delta^{\frac{\min\{r_1, r_2\}}{\max\{2,q\}}}C
\]
On the other hand, H\"{o}lder's and Tchebysheff's inequalities with (\ref{eq:lr conv 2})
imply
\begin{align*}
 & \int_{U\cap\left\{ \left|u_{j}-u_{k}\right|\geq\delta\right\} }\left|Du_{j}-Du_{k}\right|^{r_1}\d z\\
 & \ \leq\left|U\cap\left\{ \left|u_{j}-u_{k}\right|\geq\delta\right\} \right|^{\frac{p-r_1}{p}}\bigg(\int_{U\cap\left\{ \left|u_{j}-u_{k}\right|\geq\delta\right\} }\left|Du_{j}-Du_{k}\right|^{p}\d z\bigg)^{\frac{r_1}{p}}\\
 & \ \leq\delta^{r_1-p}\left\Vert u_{j}-u_{k}\right\Vert _{L^{p}(U)}^{p-r_1}C
\end{align*}
and
\begin{align*}
    & \int_{U\cap\left\{ \left|u_{j}-u_{k}\right|\geq\delta\right\} }a(z)\left|Du_{j}-Du_{k}\right|^{r_2}\d z\\
 & \ \leq||a||_{\infty}\left|U\cap\left\{ \left|u_{j}-u_{k}\right|\geq\delta\right\} \right|^{\frac{q-r_2}{q}}\bigg(\int_{U\cap\left\{ \left|u_{j}-u_{k}\right|\geq\delta\right\} }a(z)\left|Du_{j}-Du_{k}\right|^{q}\d z\bigg)^{\frac{r_2}{q}}\\
 & \ \leq||a||_{\infty}\delta^{\frac{(r_2-q)p}{q}}\left\Vert u_{j}-u_{k}\right\Vert _{L^{p}(U)}^{\frac{(q-r_2)p}{q}}C.
\end{align*}
So we arrive at
\[
\int_{U}\left|Du_{j}-Du_{k}\right|^{r_1}+a(z)|Du_j-Du_k|^{r_2}\d z\leq\left(\delta^{\frac{\min\{r_1, r_2\}}{\max\{2,q\}}}+\delta^{\frac{(r_2-q)p}{q}}\left\Vert u_{j}-u_{k}\right\Vert _{L^{p}(U)}^{\frac{(q-r_2)p}{q}}\right)C,
\]
where $C$ is independent of $j$, $k$ and $\delta$. Taking first small $\delta>0$ and then large $j,k$, we can make
the right-hand side arbitrarily small.
\end{proof}

Now we are ready to show that bounded viscosity supersolutions are weak supersolutions. In the borderline case, we only consider supersolutions that are Lipschitz continuous in space. This is because if the error term in inequality \eqref{eq:inf conv eq} has the critical exponent $q-1 = p$, the Caccioppoli's inequality in Lemma \ref{lem:Caccioppoli} no longer yields a bound for $Du_\varepsilon$ in a suitable Lebesgue space. Nevertheless, since we know that viscosity \emph{solutions} are locally Lipschitz in space by Section \ref{sec:Lipschitz estimates}, this assumption causes no loss of generality in the case of solutions.

\begin{thm}\label{final thm: visc is weak}
Let $1<p\leq q \leq p+1$. Suppose that $a:\Xi \rightarrow [0, \infty)$ is locally Lipschitz in space and continuous in time, and that \eqref{eq:visc is weak acnd} holds. Let $u$ be a bounded viscosity supersolution to
(\ref{eq:p-para f}) in $\Xi$. If $q = p +1$, suppose, moreover, that $u$ is locally Lipschitz continuous in space. Then $u$ is a local weak supersolution
to (\ref{eq:p-para f}) in $\Xi$.
\end{thm}
\begin{proof}
Without losing generality, we may suppose that $a$ is globally Lipschitz in space and uniformly continuous in time in $\Xi$. Let $\Xi^{\prime\prime} \Subset \Xi^\prime  \Subset \Xi$. Let $u_\varepsilon$ be the inf-convolution of $u$ as defined in Lemma \ref{lem:inf delta lemma} and suppose that $\varepsilon>0$ is so small that $\Xi^\prime\Subset\Xi_{\varepsilon}$.
Then by Lemma \ref{lem:inf conv is weak},  $u_\varepsilon$ is a weak supersolution to \eqref{eq:inf conv weak} in $\Xi^\prime$, for any choice of $K>0$. Using this for now with $K=1$, recalling that $E_\varepsilon$ is bounded independently of $\varepsilon$, and applying estimate \eqref{eq:f_M,epsilon growth estimate}, we see that $u_\varepsilon$ is in particular a weak supersolution to 
\begin{equation}\label{eq:visc to weak aux ineq}
    \partial_t u_\varepsilon - \div(|Du_\varepsilon|^{p-2} Du_\varepsilon + a(z)|Du_\varepsilon|^{q-2}Du_\varepsilon) \geq -\tilde C (1+|Du|^{\tilde \beta_1}+a(x,t)|Du|^{\beta_2})
\end{equation}
in $\Xi^\prime$, where $\tilde C\geq 0$, $1 \leq\tilde \beta_1 < \max(p, q-1)$ and $1 \leq \tilde\beta_2 <q$ are independent of $\varepsilon$. Therefore, if $q<p+1$, we can apply Caccioppoli's inequality in Lemma \ref{lem:Caccioppoli} on $u_\varepsilon$ to find that
\begin{equation}\label{eq:visc is weak energy est}
    \int_{\Xi^{\prime\prime}} |Du_\varepsilon|^p + a(z)|Du_\varepsilon|^q \d z < C,
\end{equation}
where $C$ is independent of $\varepsilon$. If $p=q+1$, then we have \eqref{eq:visc is weak energy est} directly by the Lipschitz assumption on $u$. Since $L^H(\Xi^{\prime\prime};\mathbb R^N)$ is a reflexive Banach space (see \cite[Lemma 2.3.16 and Remark 3.3.3]{orclizbook}), there exists $G \in L^H(\Xi^{\prime\prime};\mathbb R ^N)$ such that $Du_\varepsilon \rightharpoonup G$ weakly in $L^H(\Xi^{\prime\prime}; \mathbb R^N)$ up to a subsequence. In particular, for any $\varphi \in C_0^\infty (\Xi^{\prime\prime})$ and $i \in \{{1,\ldots, N}\}$, we have 
\begin{equation*}
    -\int_{\Xi^{\prime\prime}}u\partial_i \varphi \d z=\lim_{\varepsilon \rightarrow0}-\int_{\Xi^{\prime\prime}}u_\varepsilon\partial_i\varphi \d z =\lim_{\varepsilon \rightarrow 0} \int_{\Xi^{\prime\prime}} \varphi\partial_i u_\varepsilon \d z = \int_{\Xi ^{\prime\prime}} \varphi G_i \d z,
\end{equation*}
where we also used that $u_\varepsilon \rightarrow u$ almost everywhere. It follows that $G$ is the weak gradient of $u$ and so $Du \in L^{H}(\Xi^{\prime\prime})$.

Since $\smash{u_{\varepsilon}}$ is a weak supersolution to \eqref{eq:inf conv weak}, we have for any $K>0$ and $\varphi \in C_0^\infty(\Xi^{\prime\prime})$
\begin{align}\label{eq:visc is weak ineq}
    & \int_{\Xi^{\prime\prime}} -u_\varepsilon\partial_t \varphi + |Du_\varepsilon|^{p-2}Du_\varepsilon\cdot D\varphi + a(z)|Du_\varepsilon|^{q-2}Du_\varepsilon \cdot D\varphi +(1+|Du_\varepsilon|^{q-1})E_\varepsilon(z)\varphi\d z \nonumber\\
    & \geq \int_{\Xi^{\prime\prime}} f_{K,\varepsilon}(z, Du_\varepsilon)\varphi \d z.
\end{align}
Take open domain $U$ such that $\supp \varphi \Subset U\Subset\Xi^{\prime\prime}$. By \eqref{eq:visc is weak energy est}, the condition \eqref{eq:strong conv cnd} in Lemma \ref{lem:lr conv} is satisfied by $Du_\varepsilon$. Therefore, we have $\smash{Du_{\varepsilon}\rightarrow Du}$ in $\smash{L^{r_1, r_2}(U)}$
for any $\smash{1<r_1<p},\smash{1<r_2<q} $ and in particular $Du_\varepsilon \rightarrow Du$ almost everywhere up to a subsequence.  Furthermore, we have $E_\varepsilon \rightarrow 0$ almost everywhere in $\Xi ^\prime$. Therefore, if $q < p +1$, we can let $\varepsilon \rightarrow 0$ on the LHS of \eqref{eq:visc is weak ineq} with the help of the vector inequality
\[
\left|\left|a\right|^{p-2}a-\left|b\right|^{p-2}b\right|\leq\begin{cases}
2^{2-p}\left|a-b\right|^{p-1} & \text{when }p<2,\\
(p-1)\left(|a|^{p-2}+\left|b\right|^{p-2}\right)\left|a-b\right| & \text{when }p\geq2.
\end{cases}
\]
If $q = p + 1$, then $Du$ is bounded in the support of $\varphi$ by the Lipschitz assumption, and so we can pass to the limit on the LHS using the dominated convergence theorem. On the RHS of \eqref{eq:visc is weak ineq}, we have
\begin{align*}
\int_{\Xi^{\prime\prime}}f_{K,\varepsilon}(z,Du_{\varepsilon})\varphi\d z & = \int_{\Xi^{\prime\prime}\cap\left\{ \left|Du_{\varepsilon}\right|<K\right\} }\varphi\inf_{(y,s)\in B_{r(\varepsilon)}\times(t-r(\varepsilon),t+r(\varepsilon))}f(y,s,Du)\d z\\
 & \phantom= +\int_{\Xi^{\prime\prime}\cap\left\{ \left|Du_{\varepsilon}\right|\geq K\right\} }C\varphi (1+\left|Du_{\varepsilon}\right|^{\beta_{1}}+a(z)\left|Du_{\varepsilon}\right|^{\beta_{2}})\d z\\
 & =: T_{1}+T_{2}.
\end{align*}
Since $f$ is continuous, it follows from the Lebesgue's dominated
converge theorem that 
\[
T_{1}\rightarrow\int_{\Xi^{\prime\prime}\cap\left\{ \left|Du\right|<K\right\} }\varphi f(z,Du)\d z\quad\text{as }\varepsilon\rightarrow0.
\]
For the second term, we have
\begin{align*}
T_{2}\geq & \int_{\Xi^{\prime\prime}\cap\left\{ \left|Du_{\varepsilon}\right|>K/2\right\} }C\varphi(1+\left|Du\right|^{\beta_{1}}+a(z)\left|Du\right|^{\beta_{2}})\d z\\
 & -\int_{\Xi^{\prime\prime}}C\varphi(1+\left|Du-Du_{\varepsilon}\right|^{\beta_{1}}+a(z)\left|Du-Du_{\varepsilon}\right|^{\beta_{2}})\d z\\
\rightarrow & \int_{\Xi^{\prime\prime}\cap\left\{ \left|Du\right|>K/2\right\} }C\varphi(1+\left|Du\right|^{\beta_{1}}+a(z)\left|Du\right|^{\beta_{2}})\d z \quad \text{as } \varepsilon\rightarrow0,
\end{align*}
where we used Lebesgue's
dominated convergence and that $Du_{\varepsilon}\rightarrow Du$
in $L^{r_{1},r_{2}}(U)$ together with $1\leq\beta_{1}<p$, $1\leq\beta_{2}<q$.
Thus, as $\varepsilon\rightarrow0$, inequality \eqref{eq:visc is weak ineq} leads to
\begin{align*}
 & \int_{\Xi^{\prime\prime}}-u\partial_{t}\varphi+\left|Du\right|^{p-2}Du\cdot D\varphi+a(z)\left|Du\right|^{q-2}Du\cdot D\varphi \d z\\
 & \ge\int_{\Xi^{\prime\prime}\cap\left\{ \left|Du\right|<K\right\} }\varphi f(z,Du)\d z -\int_{\Xi^{\prime\prime}\cap\left\{ \left|Du\right|>K/2\right\} }C\varphi (1+\left|Du\right|^{\beta_{1}}+a(z)\left|Du\right|^{\beta_{2}})\d z.
\end{align*}
Finally, we let $K\rightarrow\infty$ and again apply the dominated convergence (an integrable dominant exists since $f$ satisfies the growth condition \eqref{eq:gcnd} and $Du\in L^{H}(\Xi^{\prime\prime})$). This way, we see that $u$ satisfies
the weak formulation of \eqref{eq:p-para f}.
\end{proof}

\subsection*{Acknowledgment} The first author is supported by postdoctoral research grants from the Alexander von Humboldt Foundation, Germany. The second author was supported by a postdoctoral research grant from the Emil Aaltonen Foundation, Finland.
\bibliographystyle{alpha}

\begin{thebibliography}{BGKT16}
	
	\bibitem[APR17]{OptimalC1}
	A.~Attouchi, M.~Parviainen, and E.~Ruosteenoja.
	\newblock ${C}^{1,\alpha}$ regularity for the normalized $p$-{P}oisson problem.
	\newblock {\em {J}. {M}ath. {P}ures {A}ppl.}, 108(4):553--591, 2017.
	
	\bibitem[Att12]{Attouchi12}
	A.~Attouchi.
	\newblock Well-posedness and gradient blow-up estimate near the boundary for a
	{H}amilton-{J}acobi equation with degenerate diffusion.
	\newblock {\em J. Diff. Eq.}, 253(8):2474--2492, 2012.
    
    \bibitem[BBO21]{BBO21}
     S. Baasandorj, S.-S. Byun and J. Oh. 
    \newblock Gradient estimates for multi-phase problems, 
    \newblock {\em Calc. Var. Partial Differential Equations} {\bf 60} (2021), no.~3, Paper No. 104, 48 pp.;
    
    \bibitem[BCM15]{BCM15}
    P. Baroni, M. Colombo and G. Mingione, 
    \newblock Regularity for general functionals with double phase, 
    \newblock {\em Calc. Var. Partial Differential Equations} {\bf 57} (2018), no.~2, Paper No. 62, 48 pp.
	
	
	\bibitem[BT]{bobkovTakac18}
	V.~Bobkov and P.~Tak{\'a}{\v c}.
	\newblock On maximum and comparison principles for parabolic problems with the
	$p$-{L}aplacian.
	\newblock To appear in \textit{RACSAM}.
	
	\bibitem[BT14]{bobkovTakac14}
	V.~Bobkov and P.~Tak{\'a}{\v c}.
	\newblock A strong maximum principle for parabolic equations with the
	$p$-{L}aplacian.
	\newblock {\em J. Math. Anal. Appl.}, 419(1):218--230, 2014.

     \bibitem[BDM13]{BDM13}
     V. B\"{o}gelein, F. Duzaar, P. Marcellini.
     \newblock Parabolic equations with  $p,q$ -growth.
     \newblock {\em J. Math. Pures. Appl.}, (9) 100 (2013), no. 4, 535--563.

     \bibitem[B24]{B24}
    M. Borowski, I. Chlebicka, F. De Filippis and B. Miasojedow.
     \newblock Absence and presence of Lavrentiev's phenomenon for double phase functionals upon every choice of exponents.
     \newblock {\em Calc. Var. Partial Differential Equations} {\bf 63} (2024), no.~2, Paper No. 35, 23 pp.

     \bibitem[BGS22]{BGS22}
     M. Bul\'i\v cek, P. Gwiazda and J. Skrzeczkowski, 
     \newblock On a range of exponents for absence of Lavrentiev phenomenon for double phase functionals, 
     \newblock {\em Arch. Ration. Mech. Anal.} {\bf 246} (2022), no.~1, 209--240.
     
     \bibitem[BO20]{BO20}
     S.-S. Byun and J. Oh, 
     \newblock Regularity results for generalized double phase functionals, 
     \newblock {\em Anal. PDE }{\bf 13} (2020), no.~5, 1269--1300.

     \bibitem[CDG17]{CDG17}
     J.Q. Chagas, N. M. L. Diehl and P. L. Guidolin,
     \newblock Some properties for the Steklov averages
     \newblock {\em https://doi.org/10.48550/arXiv.1707.06368}.

     \bibitem[CG\'SW21]{orclizbook}
     I. Chlebicka, P. Gwiazda, A. \'Swierczewska-Gwiazda and A. Wr\'oblewska-Kami\'nska.
     \newblock {\em Partial Differential Equations in Anisotropic Musielak-Orcliz Spaces}.
     \newblock Springer Monographs in Mathematics, Springer, Cham, 2021.
     
     
     \bibitem[CM15a]{CM15a}
     M. Colombo and G. Mingione, 
     \newblock Regularity for double phase variational problems, 
     \newblock {\em Arch. Ration. Mech. Anal.} {\bf 215} (2015), no.~2, 443--496.
     \bibitem[CM15b]{CM15b}
     M. Colombo and G. Mingione, 
     \newblock Bounded minimisers of double phase variational integrals, 
     \newblock {\em Arch. Ration. Mech. Anal.} {\bf 218} (2015), no.~1, 219--273.
	
	\bibitem[CIL92]{userguide}
	M.~G. Crandall, H.~Ishii, and P.-L. Lions.
	\newblock User's guide to viscosity solutions of second order partial
	differential equations.
	\newblock {\em Bull. Amer. Math. Soc.}, 27(1):1--67, 1992.
    
    \bibitem[DFO14]{DFO14}
    J. Diehl, P.K. Friz, and H. Oberhauser.
    \newblock Regularity theory for rough partial differential equations and parabolic comparison revisited, In Crisan D., Hambly B., Zariphopoulou T. (eds) 
    \newblock {\em Stochastic Analysis and Applications}, Springer Proceedings in Mathematics and  Statistics 100 2014, 203--238.
    
    \bibitem[DFO19]{DFO19}
    C. De~Filippis and J. Oh, 
    \newblock Regularity for multi-phase variational problems, 
    \newblock {\em J. Differential Equations} {\bf 267} (2019), no.~3, 1631--1670.
    
\bibitem[DF22]{DF22}
    C. De~Filippis, 
    \newblock Optimal gradient estimates for multi-phase integrals, 
    \newblock {Math. Eng.} {\bf 4} (2022), no.~5, Paper No. 043, 36 pp.
\bibitem[DF84]{DF84}
     E. DiBenedetto and A. Friedman.
    \newblock Regularity of solutions of nonlinear degenerate parabolic systems, 
    \newblock {\em J. Reine Angew. Math.} {\bf 349} (1984), 83--128.
\bibitem[DF85a]{DF85a}
    E. DiBenedetto and A. Friedman, 
    \newblock H\"older estimates for nonlinear degenerate parabolic systems, 
    \newblock {\em J. Reine Angew. Math.} {\bf 357} (1985), 1--22.
    
\bibitem[DF85b]{DF85b}
    E. DiBenedetto and A. Friedman, 
    \newblock Addendum to: ``H\"older estimates for nonlinear degenerate parabolic systems'', 
    \newblock{\em J. Reine Angew. Math.} {\bf 363} (1985), 217--220. 
	
	\bibitem[DiB93]{dibenedetto93}
	E.~Di{B}enedetto.
	\newblock {\em Degenerate parabolic equations}.
	\newblock Springer-Verlag, 1993.
	
	\bibitem[EG15]{measuretheoryevans}
	L.~C. Evans and R.~F. Gariepy.
	\newblock {\em Measure theory and fine properties of functions}.
	\newblock CRC Press, revised edition, 2015.
    
    \bibitem[FRZ22]{FRZ22}
    Y. Fang, V.D. Radulescu and C. Zhang. 
    \newblock Gradient estimates for multi-phase problems in Campanato spaces, 
    \newblock {\em Indiana Univ. Math. J.} {\bf 71} (2022), no.~3, 1079--1099.
    
    \bibitem[FZ22]{FZ22}
    Y. Fang and C. Zhang, 
    \newblock Equivalence between distributional and viscosity solutions for the double-phase equation, 
    \newblock {\em Adv. Calc. Var.} {\bf 15} (2022), no.~4, 811--829.
    
    \bibitem[FRZ24]{FRZ24} 
    Y. Fang, V.~D. R\u adulescu and C. Zhang, 
    \newblock Equivalence of weak and viscosity solutions for the nonhomogeneous double phase equation, 
    \newblock {\em Math. Ann.} {\bf 388} (2024), no.~3, 2519--2559.
	
    
    \bibitem[IJS19]{IJS19}
    C. Imbert, T. Jin and L.~E. Silvestre, 
    \newblock H\"older gradient estimates for a class of singular or degenerate parabolic equations, \newblock {\em Adv. Nonlinear Anal.} {\bf 8} (2019), no.~1, 845--867.

    \bibitem[IL90]{IL90} H. Ishii and P.-L. Lions,
    \newblock Viscosity solutions of fully nonlinear second-order elliptic partial differential equations, 
    \newblock {\em J. Differential Equations} 83 (1990), no. 1, 26--78.
	
	\bibitem[Ish95]{Ishii95}
	H.~Ishii.
	\newblock On the equivalence of two notions of weak solutions, viscosity
	solutions and distribution solutions.
	\newblock {\em Funkcialaj Ekvacioj}, 38:101--120, 1995.

    \bibitem[IMJ97]{IMJ97}A. V. Ivanov, P. Z. Mkrtychan, W. J\"ager. 
    \newblock Existence and uniqueness of a regular solution of the Cauchy-Dirichet problem for a class of doubly nonlinear parabolic equations.
    \newblock {\em Journal of Mathematical Sciences}, 84, 845--855 (1997).
	
	\bibitem[JJ12]{newequivalence}
	V.~Julin and P.~Juutinen.
	\newblock A new proof for the equivalence of weak and viscosity solutions for
	the $p$-{L}aplace equation.
	\newblock {\em Comm. Partial Differential Equations}, 37(5):934--946, 2012.
    \bibitem[Juu01]{parabolic viscosity solutions}
    P.~Juutinen.
    \newblock On the definition of viscosity solutions for parabolic equations.
    \newblock {\em Proc. Amer. Math. Soc.}, 129(10):2907--2911, 2001.
    
	\bibitem[JLM01]{equivalence_plaplace}
	P.~Juutinen, P.~Lindqvist, and J.J. Manfredi.
	\newblock On the equivalence of viscosity solutions and weak solutions for a
	quasi-linear equation.
	\newblock {\em SIAM J. Math. Anal.}, 33(3):699--717, 2001.
	
	\bibitem[Jun93]{junning90}
	Z.~Junning.
	\newblock Existence and nonexistence of solutions for $u_{t}=\div(|\nabla
	u|^{p-2}\nabla u)+f(\nabla u,u,x,t)$.
	\newblock {\em J. Math. Anal. Appl.}, 172(1):130--146, 1993.
	
	\bibitem[Kat15]{nikos}
	N.~Katzourakis.
	\newblock {\em An Introduction To Viscosity Solutions for Fully Nonlinear PDE
		with Applications to Calculus of Variations in $L^{\infty}$}.
	\newblock Springer, 2015.

    \bibitem[KL96]{KL96}
    T. Kilpel\"ainen and P. Lindqvist.
    \newblock On the Dirichlet boundary value problem for a degenerate parabolic equation.
    \newblock {\em SIAM J. Math. Anal.} 27 (1996), no. 3, 661-683.

    \bibitem[KMS25]{KMS25}
    W. Kim, K. Moring and L. S\"{a}rki\"{o}
    \newblock H\"{o}lder regularity for degenerate parabolic double-phase equations.
    \newblock {\em J. Differential Equations}, https://doi.org/10.1016/j.jde.2025.113231
    
    \bibitem[KS24]{KS24}
    W. Kim and L. S\"arki\"o, 
    \newblock Gradient higher integrability for singular parabolic double-phase systems, 
    \newblock {\em NoDEA Nonlinear Differential Equations Appl.} {\bf 31} (2024), no.~3, Paper No. 40, 38
    pp.

    \bibitem[Kim24]{Kim24}
    \newblock Calder\'{o}n-Zygmund type estimate for the parabolic double-phase system.
    \newblock {\em To appear in  Ann. Sc. Norm. Super. Pisa}

    \bibitem[Kim25]{Kim25}
    W. Kim,
    \newblock Calder\'{o}n-Zygmund type estimate for the singular parabolic double-phase system,
    \newblock {\em Journal of Mathematical Analysis and Applications,}2025,https://doi.org/10.1016/j.jmaa.2025.129593.

    \bibitem[KO24]{KO24} 
    B. Kim and J. Oh.
    \newblock Higher integrability for weak solutions to parabolic multi-phase equations, 
    \newblock {\em J. Differential Equations} {\bf 409} (2024), 223--298. 

    \bibitem[KOS25]{KOS25}
    B. Kim, J. Oh and A. Sen. 
    \newblock Parabolic Lipschitz truncation for multi-phase problems: The degenerate case. 
    \newblock Advances in Calculus of Variations (2025). https://doi.org/10.1515/acv-2025-0003

    \bibitem[KKM23]{KKM23}
    W. Kim, J. Kinnunen and K. Moring, 
    \newblock Gradient higher integrability for degenerate parabolic double-phase systems, 
    \newblock {\em Arch. Ration. Mech. Anal.} {\bf 247} (2023), no.~5, Paper No. 79, 46 pp.
	
	\bibitem[KKP10]{korteKuusiParv10}
	R.~Korte, T.~Kuusi, and M.~Parviainen.
	\newblock A connection between a general class of superparabolic functions and
	supersolutions.
	\newblock {\em J. Evol. Eq.}, 10(1):1--20, 2010.
	
\bibitem[QL25]{QL25}
   Qifan Li. 
   \newblock Continuity estimates for degenerate parabolic double-phase equations via nonlinear potentials.
   \newblock {\em https://arxiv.org/abs/2502.01097}
	
	\bibitem[Lin12]{lindqvist12reg}
	P.~Lindqvist.
	\newblock Regularity of supersolutions.
	\newblock In {\em Regularity estimates for nonlinear elliptic and parabolic
		problems}, volume 2045 of {\em Lecture Notes in Math}, pages 73--131. 2012.
	
	\bibitem[Lin17]{lindqvist_plaplace}
	P.~Lindqvist.
	\newblock Notes on the $p$-{L}aplace equation (second edition).
	\newblock \textit{{U}niv. Jyv{\"a}skyl{\"a}, {R}eport 161}, 2017.
	
	\bibitem[LM07]{lindqvistManfredi07}
	P.~Lindqvist and J.~J. Manfredi.
	\newblock Viscosity supersolutions of the evolutionary $p$-{L}aplace equation.
	\newblock {\em Differential Integral Equations}, 20(11):1303--1319, 2007.
    
    \bibitem[LPS25]{LPS25} 
    P. Lindqvist, M. Parviainen and J. Siltakoski.
    \newblock Lipschitz continuity and equivalence of positive viscosity and weak solutions to Trudinger's equation.
    \newblock{\em https://doi.org/10.48550/arXiv.2502.14670}
    
	
	\bibitem[Mic00]{miculescu00}
	R.~Miculescu.
	\newblock Approximation of continuous functions by {L}ipschitz functions.
	\newblock {\em Real Anal. Exchange}, 26(1):449--452, 2000.
	
	\bibitem[MO]{chilepaper}
	M.~Medina and P.~Ochoa.
	\newblock On viscosity and weak solutions for non-homogeneous $p$-{L}aplace
	equations.
	\newblock \textit{{A}dv. {N}onlinear {A}nal.} (2017) https://doi.org/10.1515/anona-2017-0005
	
	
	\bibitem[PV]{ParvVaz}
	M.~Parviainen and J.~L. V\'azquez.
	\newblock Equivalence between radial solutions of different parabolic
	gradient-diffusion equations and applications.
	\newblock arXiv:1801.00613.
    
   \bibitem[Sen25]{Sen25}
    A. Sen.
    \newblock Gradient Higher Integrability for Degenerate/Singular Parabolic Multi-phase Problems. 
    \newblock {J.\em  Geom. Anal.} 35, 170 (2025). https://doi.org/10.1007/s12220-025-01950-4
	
	\bibitem[Sil18]{siltakoski18}
	J.~Siltakoski.
	\newblock Equivalence of viscosity and weak solutions for the normalized
	$p(x)$-{L}aplacian.
	\newblock {\em Calc. Var. Partial Differential Equations}, 57(95), 2018.

    \bibitem[Sil21]{Sil21}
    J. Siltakoski, 
    \newblock Equivalence of viscosity and weak solutions for a $p$-parabolic equation, 
    \newblock{\em J. Evol. Equ.} {\bf 21} (2021), no.~2, 2047--2080.
    
    \bibitem[Sin15]{Sin15}
	T. Singer.
    \newblock Parabolic equations with $p,q$-growth: the subquadratic case, 
    \newblock {\em Q. J. Math.} {\bf 66} (2015), no.~2, 707--742.
    
    
    \bibitem[The75]{The75}
     C.~M. Theobald, 
    \newblock An inequality for the trace of the product of two symmetric matrices, 
    \newblock{\em Math. Proc. Cambridge Philos. Soc.} {\bf 77} (1975), 265--267
    
\bibitem[Zhi93]{Zhi93}
    V.~V. Zhikov, 
    \newblock Lavrentiev phenomenon and homogenization for some variational problems, \newblock{\em C. R. Acad. Sci. Paris S\'er. I Math.} {\bf 316} (1993), no.~5, 435--439.
\bibitem[Zhi95]{Zhi95}
    V.~V. Zhikov, 
    \newblock On Lavrentiev's phenomenon, 
    \newblock {\em Russian J. Math. Phys.} {\bf 3} (1995), no.~2, 249--269.
    \bibitem[Zhi97]{Zhi97}
    V.~V. Zhikov, 
    \newblock On some variational problems, 
    \newblock {\em Russian J. Math. Phys.} {\bf 5} (1997), no.~1, 105--116. 
	
\end{thebibliography}

\Addresses
\end{document}